\documentclass[11pt,reqno]{amsart}

\usepackage{amsmath}
\usepackage{amsfonts}
\usepackage{amssymb}

\usepackage{amsthm}
\usepackage{mathrsfs}
\usepackage{latexsym}

\usepackage{cite} 

\usepackage{esint}

\numberwithin{equation}{section}

\usepackage{graphicx,subfigure}

\usepackage{ifthen} 

\provideboolean{shownotes} 
\setboolean{shownotes}{true} 
\newcommand{\margnote}[1]{
\ifthenelse{\boolean{shownotes}}%
{\marginpar{\raggedright\tiny\texttt{#1}}}%
{}%
}

\newcommand{\hole}[1]{
\ifthenelse{\boolean{shownotes}}%
{\begin{center} \fbox{ \rule {.25cm}{0cm}
\rule[-.1cm]{0cm}{.4cm} \parbox{.85\textwidth}{\begin{center}
\texttt{#1}\end{center}} \rule {.25cm}{0cm}}\end{center}}
{}
}

\usepackage[colorlinks=true,urlcolor=blue,
citecolor=red,linkcolor=blue,linktocpage,pdfpagelabels,
bookmarksnumbered,bookmarksopen]{hyperref}

\usepackage{fullpage}

\theoremstyle{plain}

\newtheorem{lemma}{Lemma}[section]
\newtheorem{theorem}[lemma]{Theorem}
\newtheorem{proposition}[lemma]{Proposition}
\newtheorem{corollary}[lemma]{Corollary}

\theoremstyle{definition}
\newtheorem{remark}[lemma]{Remark}
\newtheorem{definition}[lemma]{Definition}

\theoremstyle{remark}

\newcommand{\vertiii}[1]{{\left\vert\kern-0.25ex\left\vert\kern-0.25ex\left\vert #1 
    \right\vert\kern-0.25ex\right\vert\kern-0.25ex\right\vert}}

\newcommand{\R}{\mathbb{R}}

\newcommand{\C}{\mathbb{C}}
\newcommand{\Z}{\mathbb{Z}}
\newcommand{\N}{\mathbb{N}}

\newcommand{\vphi}{\boldsymbol{\phi}}
\newcommand{\vg}{\boldsymbol{g}}
\newcommand{\vh}{\boldsymbol{h}}
\newcommand{\vw}{\boldsymbol{w}}
\newcommand{\vu}{\boldsymbol{u}}
\newcommand{\vq}{\boldsymbol{q}}
\newcommand{\vx}{\boldsymbol{x}}

\newcommand{\DD}{\mathbf{D}}
\newcommand{\II}{\mathbf{I}}
\newcommand{\TT}{\mathbf{T}}
\newcommand{\GG}{\mathbf{G}}
\newcommand{\KK}{\mathbf{K}}
\newcommand{\tTT}{\widetilde{\mathbf{T}}}
\newcommand{\bU}{\overline{U}}
\newcommand{\bV}{\overline{V}}

\newcommand{\cC}{\check{C}}

\newcommand{\tiK}{\widetilde{K}}

\newcommand{\tiA}{\widetilde{A}}
\newcommand{\tiB}{\widetilde{B}}
\newcommand{\tiV}{\widetilde{V}}
\newcommand{\tiN}{\widetilde{N}}
\newcommand{\tiW}{\widetilde{W}}

\newcommand{\cL}{{\mathcal{L}}}
\newcommand{\cD}{{\mathcal{D}}}

\newcommand{\cA}{{\mathcal{A}}}

\newcommand{\cU}{{\mathcal{U}}}
\newcommand{\cV}{{\mathcal{V}}}
\newcommand{\cE}{{\mathcal{E}}}
\newcommand{\cZ}{{\mathcal{Z}}}

\newcommand{\vep}{\varepsilon}

\renewcommand{\Re}{\mathrm{Re}\,} 
\renewcommand{\Im}{\mathrm{Im}\,}
\newcommand{\tr}{\mathrm{tr}\,}

\newcommand{\bu}{\overline{u}}

\newcommand{\bmu}{\overline{\mu}}
\newcommand{\bkp}{\overline{\kappa}}

\newcommand{\bM}{\overline{M}}
\newcommand{\bm}{\overline{m}}

\newcommand{\brho}{\overline{\rho}}
\newcommand{\bthe}{\overline{\theta}}
\newcommand{\balp}{\overline{\alpha}}
\newcommand{\bp}{\overline{p}}
\newcommand{\bc}{\overline{c}}
\newcommand{\be}{\overline{e}}
\newcommand{\bk}{\overline{k}}

\newcommand{\hW}{\widehat{W}}

\newcommand{\hV}{\widehat{V}}

\DeclareMathOperator{\Hess}{Hess}

\newcommand{\ep}{\epsilon}

\newcommand{\<}{\langle}
\renewcommand{\>}{\rangle}

\begin{document}

\title[Global decay of heat conducting Korteweg compressible fluids]{Global decay of perturbations of equilibrium states for one-dimensional heat conducting compressible fluids of Korteweg type}

\author[R. G. Plaza]{Ram\'on G. Plaza}

\address{{\rm (R. G. Plaza)} Instituto de 
Investigaciones en Matem\'aticas Aplicadas y en Sistemas\\Universidad Nacional Aut\'onoma de 
M\'exico\\ Circuito Escolar s/n, Ciudad Universitaria, C.P. 04510\\Cd. de M\'{e}xico (Mexico)}

\email{plaza@mym.iimas.unam.mx}

\author[J. M. Valdovinos]{Jos\'{e} M. Valdovinos}

\address{{\rm (J. M. Valdovinos)} Instituto de 
Investigaciones en Matem\'aticas Aplicadas y en Sistemas\\Universidad Nacional Aut\'onoma de 
M\'exico\\ Circuito Escolar s/n, Ciudad Universitaria, C.P. 04510\\Cd. de M\'{e}xico (Mexico)}

\email{valdovinos94@comunidad.unam.mx}

\begin{abstract}
This paper studies the one dimensional Navier-Stokes-Fourier-Korteweg system of equations describing the evolution of a heat-conducting compressible fluid that exhibits viscosity and capillarity. The main goal of the present analysis is to examine the dissipative structure of the system and to prove the global existence and the asymptotic decay of perturbations of equilibrium states. For that purpose, a novel nonlinear change of perturbed state variables, which takes into account that the conserved quantities contain density gradients, is introduced. These new perturbation variables satisfy a partially symmetric system whose linearization fulfills the generalized genuine coupling condition of Humpherys \cite{Hu05} for higher order systems. It is shown that the linearized system is symbol symmetrizable and an appropriate compensating matrix is constructed. This procedure allows to obtain linear decay rates which underlie a dissipative mechanism of regularity-gain type. This linear dissipative structure implies, in turn, the global decay of small perturbations to constant equilibrium states as solutions to the full nonlinear system.
\end{abstract}

\keywords{Navier-Stokes-Fourier-Korteweg compressible fluids, dissipative structure, global decay, genuine coupling.}

\subjclass[2020]{35Q35, 35B35, 35B40, 35G20}

\maketitle

\setcounter{tocdepth}{1}



\section{Introduction}

In this paper we study the one dimensional Navier-Stokes-Fourier-Korteweg (NSFK) system of equations, which describes the evolution of a heat-conducting compressible fluid that exhibits viscosity and capillarity on an infinite one-dimen\-sional spatial domain. In Eulerian coordinates, the system under consideration reads (cf. \cite{Hasp09,CHZ17,DS85,HaLi96b}),
\begin{equation}
\label{NSFK}
\begin{aligned}
        \rho_t + (\rho u)_x &= 0,   \\
        (\rho u)_t + \big(\rho u^{2}+ p \big)_x &= \big( \mu u_{x} + K \big)_x, \\
        \big(\rho  \varepsilon + \tfrac{1}{2}\rho u^{2} \big)_t + \big( \rho u \big( \varepsilon + \tfrac{1}{2}u^2 \big) + pu \big)_x &= \big( \alpha \theta_{x} + \mu u u_{x} + uK + w \big)_x.       
\end{aligned} 
\end{equation}

Here $x \in \R$ and $t > 0$ denote space and time coordinates, respectively. The unknown functions are the mass density $\rho$, the velocity $u$ and the absolute temperature $\theta$. According to custom, $p$ denotes the thermodynamic pressure function, and $\varepsilon$ is the internal energy density (per unit mass) of the fluid. The Korteweg stress tensor, $K$, and the interstitial work flux, $w$, are given by
\begin{equation}
\label{defKw}
\begin{aligned}
K &= k \rho \rho_{xx} + \rho k_{x}\rho_{x} - \tfrac{1}{2}k_{\rho}\rho \rho_{x}^2 - \tfrac{1}{2} k\rho_{x}^2, \quad \text{and,}\\
w &= -k \rho \rho_{x} u_{x},
\end{aligned}
\end{equation}
respectively. Here the viscosity coefficient $\mu$, the thermal conductivity $\alpha$ and the capillarity coefficient $k$ are strictly positive smooth functions of the state variables $\rho$ and $\theta$.

System \eqref{NSFK} is the one dimensional version of the system rigorously derived by Dunn and Serrin \cite{DS85} under the framework of Rational Mechanics to account for compressible fluids endowed with internal capillarity. The model originated in the early XXth century when the Dutch physicist D. J. Korteweg proposed a form for the Cauchy stress tensor depending on variations of the mass density of the fluid (cf. Korteweg \cite{Kortw1901}). Korteweg's formulation was, however, incompatible with the Clasius-Dulhem inequality that expresses the Second Law of Thermodynamics and demands non-negative energy dissipation, which is the rate of internal entropy production per unit volume times the absolute temperature. Dunn and Serrin circumvented this problem with the introduction of the interstitial work flux into the energy equation, which accounts for an additional supply of mechanical energy (for details see \cite{DS85} or Section \ref{secKortew} below). It is important to emphasize that, in this framework, Korteweg-type models are based on an extended version of \emph{non-equilibrium thermodynamics} which assumes that the energy of the fluid not only depends on the standard thermodynamic variables but also on the gradient of the density. In the system \eqref{NSFK} under consideration, this fact is reflected in the particular ``non-standard'' form of the internal energy of the fluid,
\[
\vep = \vep(\rho, \theta,\rho_x) = e(\rho, \theta) + \tfrac{1}{2} \big(k(\rho, \theta) - \theta k_\theta(\rho, \theta) \big) \rho_x^2/\rho,
\]
where the function $e = e(\rho, \theta)$ is a ``standard'' internal energy from equilibrium thermodynamics and $k = k(\rho, \theta)$ is the capillarity coefficient appearing in \eqref{defKw} (see Section \ref{secKortew} below for a justification of this expression). Therefore, the vector of conserved quantities (namely, mass, momentum and total energy) depends upon density gradients as well, through the non-standard term in the internal energy. Surprisingly, many serious mathematical analyses of the non-isothermal Korteweg model (see, for example, \cite{ChXi13,WaYa14,ChZha14,ZhTa14,SSZh22,HYZ17}) overlook this important physical feature by assuming the internal energy to be in equilibrium form, that is, by taking $k \equiv 0$ in the expression of the internal energy (leading to $\vep = e = c_v \theta$) but keeping $k \neq 0$ in the momentum and energy equations. It is to be observed that in the isothermal version of the model, which only considers balance of mass and momentum (cf. \cite{PlV22,DD01,HaLi94}) and known as the Navier-Stokes-Korteweg (NSK) system, the conserved quantities do not depend on density gradients, even when the capillarity coefficient is switched on. Also notice that, when the capillarity coefficient is set to zero in both the expression of the internal energy and in the balance equations, then the NSFK system \eqref{NSFK} reduces to the classical Navier-Stokes-Fourier (NSF) model involving the standard thermodynamic potential for its internal energy, $e = e(\rho, \theta)$.

Since the seminal work of Dunn and Serrin, there has been a large amount of work on the mathematical analysis of models of Korteweg type reported in the literature. Many mathematical results pertain to the existence and uniqueness of strong \cite{HaLi94,HaLi96a,HaLi96b,Hasp09,Hasp16,Kot08,Kot10} and weak \cite{DD01,BrDjL03,Hasp09} solutions, as well as to the existence and stability of nonlinear elementary waves \cite{Sl83,ChXi13,CHZ15} (see also the results on the purely capillary, non-viscous model by Benzoni-Gavage \emph{et al.} \cite{BDDJ05, BDD06}). Most of these works study the isothermal version of the Korteweg system, namely the NSK model. The full, non-isothermal NSFK system has been, in contrast, less studied. Still, there is a vast literature on it, which includes the works by Hattori and Li \cite{HaLi96b}, Haspot \cite{Hasp09}, Chen \emph{et al.} \cite{ChXM14}, Freist\"uhler and Kotschote \cite{FrKo15a,FrKo17,FrKo19}, Kotschote \cite{Kot10,Kot12a}, Cai \emph{et al.} \cite{CTX15}, Chen and Zhao \cite{ChZha14}, Hou \emph{et al.} \cite{HPZ18}, Keim \emph{et al.} \cite{KMR23} and Tian \emph{et al.} \cite{TXKV15}, just to mention a few. The above list of references is by no means exhaustive, and the reader is invited to review the aforementioned works and to consult the many references therein. 

The main goal of this work is to study the existence and global decay of small perturbations of constant equilibrium states of system \eqref{NSFK}. Given an arbitrary constant state, say $\bU = (\brho, \bu, \bthe)$ with $\brho > 0$ and $\bthe > 0$, we pose the nonlinear problem for an eventual perturbation and show that these solutions exist globally in time and that they decay in appropriate function spaces. One course of action is to establish energy estimates on the perturbations directly, relying on the intrinsic form of the Korteweg system. Another approach is to examine whether the linearized system around the constant state exhibits some abstract symmetrizability and dissipative properties that can be extrapolated to the nonlinear problem. This refers to the \emph{dissipative structure} of the nonlinear system of equations \eqref{NSFK}. Therefore, our analysis follows this second approach and falls under the framework of Humpherys’s extension \cite{Hu05} to higher order systems of the classical results by Kawashima \cite{KaTh83} and Shizuta and Kawashima \cite{ShKa85,KaSh88a} for systems of hyperbolic-parabolic type. Humpherys considered generic linear systems of any order of the form,
\begin{equation}
\label{HuNSys}
U_{t} = -\sum_{k=0}^{m}D_{k}\partial_{x}^{k}U, \quad t>0, \quad x \in \R, \quad U\in \R^{n},
\end{equation}
where $D_{k}$, $k=0,1,\ldots, m$, are constant $n \times n$ matrices. After taking the Fourier transform of system \eqref{HuNSys} and splitting the symbol into odd and even order derivatives, Humpherys extended in a natural way the classical definitions of genuine coupling and strict dissipative for second order systems due to Kawashima and Shizuta, by working with Fourier symbols instead of with constant matrices. The main result in \cite{Hu05} is the equivalence between the above conditions and the existence of a compensating matrix symbol, which allows us to perform energy estimates at the linear level. Another important contribution by Humpherys is the concept of \emph{symbol symmetrizability}, which generalizes the classical definition of (termwise) symmetrizability by Friedrichs \cite{Frd54}. Actually, every symmetrizable system in the sense of Friedrichs is symbol symmetrizable, but no the other way around, as it was proved by Humpherys for the physically relevant isothermal Korteweg system in one dimension and in Lagrangian variables. The potential applications of Humpherys' results and concepts are numerous because they pertain to linear operators of any order. In a recent contribution \cite{PlV22}, we studied the dissipative structure of the isothermal Korteweg system under Humpherys' framework and were able to extend the linear decay structure to the nonlinear system in order to show the global decay of perturbations to equilibrium states. Thus, our motivation is to extrapolate these ideas to the full non-isothermal Korteweg system.

The study of the non-isothermal case presents new challenges. First, the conserved quantities depend upon density gradients. It is to be observed that one may linearize system \eqref{NSFK} around a constant state $\bU$ and examine its decay structure. Since the gradient of a constant density is zero, the linearization coincides with the standard version of the system endowed with standard thermodynamic potentials. In fact, there is a recent paper by Kawashima \emph{et al.} \cite{KSX22}, in which the authors perform a linearization of the full non-isothermal Korteweg system, even in several space dimensions, and obtain decay estimates for the solutions to the linearized equations. The authors do not employ Humpherys' characterization of symbol symmetrizability and coupling of the symbols, but handle the non-symmetric part of the equations with \emph{ad hoc} ``craftmanship'' conditions which are fulfilled by the Korteweg system under consideration. Their analysis holds at the linear level only (for details and further information, see \cite{KSX22}). To complete the study of nonlinear decay of perturbations, however, one needs to keep track of the dependence on density gradients of the conserved quantities. For the one dimensional NSFK system \eqref{NSFK}, the latter have the form
\[
F^0(U,U_x) = \begin{pmatrix} 
\rho \\ \rho u \\ \rho \vep(\rho, \theta,\rho_x) + \tfrac{1}{2} \rho u^2
\end{pmatrix},
\]
where $U = (\rho, u, \theta)$ denotes the state variables and $\vep = \vep (\rho, \theta,\rho_x)$ is the non-standard energy density. As we have already mentioned, when evaluated at a constant state $\bU$, the conserved quantities reduce to $F^0(\bU,0)$. Our analysis hinges on a novel (but quite natural) nonlinear change of perturbation variables,
\[
(U, U_x) \mapsto W = D_UF^0(\bU,0)^{-1} \big( F^0(U,U_x) - F^0(\bU,0)\big),
\]
where $D_UF^0$ denotes the Jacobian of the vector of three conserved quantities with respect to the three state variables, $U = (\rho, u, \theta)$, only. This transformation is motivated by the symmetrization procedure for the classical NSF system (wihout capillarity) performed by Kawashima and Shizuta \cite{KaSh88a}, which makes use of the notion of a convex extension for hyperbolic-parabolic systems of conservations laws (see also Chapter 4 in \cite{KaTh83}) and which constitutes an extension of the standard definition of a convex entropy/entropy flux pair introduced by Godunov \cite{Godu61a} and by Friedrichs and Lax \cite{FLa3} for hyperbolic systems. To perform the change of variables, we make use the non-standard entropy of the fluid which depends on $\rho_x$ as well. At first, this does not represent any difficulty regarding the linear terms of the new system, as the matrices used for the change of variables are evaluated at the constant equilibrium state, for which the non-standard part of all thermodynamic potentials vanish (and, hence, they coincide to those for the NSF system; for completeness, we recall the symmetrization procedure for the NSF system in Appendix \ref{ConvExtN-S}). However, one additional thermodynamic assumption is needed to handle the nonlinear part. After some calculations, one observes that the change of variables works and that the new nonlinear system for the $W$ is feasible, provided that a certain matrix field is invertible (see Proposition \ref{propgoodW} below). This requirement is fulfilled by assuming that the capillarity coefficient is a \emph{non-strictly concave function of the temperature at constant density} (see Remark \ref{remconcave} and assumption \hyperref[H4]{\rm{(H$_4$)}} below). This is a thermodynamic hypothesis that we have not found in the physical literature explicitly. It is, however, quite natural. It can be interpreted as a sufficient condition for the non-standard Hemholtz free energy density to be (strictly) concave as a function of the absolute temperature, a well-known stability requirement formulated by Hanley and Evans \cite{HaEv82} for non-equilibrium thermodynamic systems under shear (see also \cite{JCL10,LeJC08}). For convenience of the reader, in Subsection \ref{secKortew} we present a brief review of the derivation of the Korteweg system by Dunn and Serrin \cite{DS85}, which also contains a discussion on non-equilibrium thermodynamics. Even though this derivation can be found in many references (see, e.g., \cite{FrKo19,Szk20b,Szk20a}), we have included it here in order to convince the reader of the feasibility of the new thermodynamic assumption of this paper.

Under the assumption of non-strict concavity of the capillarity coefficient, the nonlinear change of variables is well defined and the new perturbed variables $W$ satisfy an equivalent nonlinear system of equations in conservation form. Moreover, it is proved that the new variables $W$ are equivalent to the perturbation, $U - \bU$, in appropriate function spaces (see Lemma \ref{lemregW}). Thanks to the convex extension procedure mentioned above, the resulting system is ``partially'' symmetric in the sense that, at the linear level, the transport and viscosity terms get symmetrized, but with a non-symmetric capillarity tensor. A significant part of our analysis is devoted to prove these features (see Section \ref{secconvexvar}). Once the nonlinear system is properly posed in the new variables, we proceed with its linearization and with the study of its decay properties. We prove, for instance, that the resulting linearized system is symbol symmetrizable, but not symmetrizable in the sense of Friedrichs (see Lemma \ref{Sym2Full} below), adding the NSFK system to the (still small) list of physically relevant models which satisfy this property. It is also shown that it fulfills the genuine coupling condition and, therefore, one may invoke the equivalence theorem by Humpherys (see Theorems 3.3 and 6.3 in \cite{Hu05} or Theorem \ref{HuThSym} below). However, as we pointed out in our previous article \cite{PlV22}, since we are working with symbols instead of with constant matrices, there is an extra degree of freedom in the choice of the compensating matrix function and, therefore, in applications it is more convenient to construct this symbol directly. We then construct an appropriate compensating matrix symbol for the NSFK system which is endowed with some extra properties which cannot be deduced from the equivalence theorem (see Lemma \ref{lemourK}). As a consequence of such construction, we are able to prove a linear decay structure of \emph{regularity-gain type}, yielding optimal pointwise decay rates of the solutions to the linear system in Fourier space (see \cite{UDK12, UDK18} or Remark \ref{remtypediss} below). These pointwise energy estimates imply, in turn, the decay rate for the semigroup associated to solutions of the linear system. Relying on the existing local well-posedness theory for the full NSFK model by Hattori and Li \cite{HaLi96b} (see also \cite{CHZ17}), we then combine the linear estimates with a priori nonlinear estimates in order to prove the global existence and the decay of small perturbations to equilibria for the full nonlinear system, via a standard continuation argument. The particular form of the equations in the new variables $W$ plays a fundamental role to close the nonlinear estimate and to obtain the appropriate regularity in the density variable. Our analysis culminates into the Main Theorem \ref{thmgloex} below.

There exist many results on the global existence and optimal decay of perturbations to constant equilibrium states for the isothermal NSK model, even in several space dimensions (see, e.g., \cite{WaTa11, TWX12,TZh14,GaZoYa15,GLZh20}). In these works, the authors perform energy estimates directly and exploit the intrinsic form of the system of equations. Our isothermal analysis in \cite{PlV22}, in contrast, yields similar results by studying the strict dissipativity of the system. Up to our knowledge, there exist only two results in the literature which study the stability of constant equilibrium states in the non-isothermal case (see \cite{SSZh22,ZhTa14}). These works, however, either do not consider non-equilibrium (non-standard) energy densities or do not add the interstitial work flux to the energy equation. Thus, as far as we know, ours is the first result on global existence and decay of perturbations of constant equilibrium states for the full non-isothermal nonlinear Korteweg system derived by Dunn and Serrin. More importantly, the present work applies the linear formulation and definitions by Humpherys \cite{Hu05}, showing how to extend the linear decay analysis to the nonlinear problem even in the presence of density gradients inside the conserved quantities. In other words, we consider our contributions not only novel for the analysis of Korteweg fluids, but also of methodological nature, with prospects for the study of larger systems or with higher order derivatives.

\subsection*{Plan of the paper}

Section \ref{secprel} contains preliminary material. In Subsection \ref{secKortew} we review the derivation of the Korteweg model by Dunn and Serrin \cite{DS85} and enumerate the main hypotheses of our analysis (assumptions \hyperref[H1]{\rm{(H$_1$)}} - \hyperref[H4]{\rm{(H$_4$)}} below). In Subsection \ref{Well-poss-Sec} we recall the local well-posedness theory for the non-isothermal Korteweg system, paying particular attention to the works by Hattori and Li \cite{HaLi96b} and by Chen \emph{et al.} \cite{CHZ17}, which establish the existence and uniqueness of classical solutions. The central Section \ref{secconvexvar} recasts the nonlinear system \eqref{NSFK} as a partially symmetric system in terms of new perturbation variables. The latter transformation works thanks to the thermodynamic assumption \hyperref[H4]{\rm{(H$_4$)}}. In Section \ref{seclindecay} we examine the linear part of this system and show that it satisfies the genuine coupling condition and the symbol symmetrizability in the sense of Humpherys. A suitable compensating matrix symbol, which conveys a dissipative mechanism of regularity-gain type, is constructed. This symbol is employed to obtain linear decay estimates for the associated semigroup. Finally, in Section \ref{secglobal} we use the previous information to prove the global decay of small perturbations to constant equilibrium states of the full nonlinear system. 

\subsection*{On notation} 
Standard Sobolev spaces of functions on the real line will be denoted as $H^s(\R)$, with $s \in \R$, endowed with the standard inner products and norms. The norm on $H^s(\R)$ will be denoted as $\| \cdot \|_s$ and the norm in $L^2$ will be henceforth denoted by $\| \cdot \|_0$. Any other $L^p$ -norm will be denoted as $\| \cdot \|_{L^p}$ for $p \geq 1$. We denote the real and imaginary parts of a complex number $\lambda \in \C$ by $\Re\lambda$ and $\Im\lambda$, respectively, as well as complex conjugation by ${\lambda}^*$. The standard inner product of two vectors $a$ and $b$ in $\C^n$ will be denoted by $\langle a,b \rangle = \sum_{j=1}^n a_j b_j^*$. Complex transposition of matrices are indicated by the symbol $A^*$, whereas simple transposition is denoted by the symbol $A^\top$. 

\section{Preliminaries}
\label{secprel}

\subsection{Heat conducting compressible fluids of Korteweg type}
\label{secKortew}

In order to model capillarity effects in fluids and based on an original idea by van der Waals \cite{vdW1894}, in 1901 the Dutch physicist D. J. Korteweg \cite{Kortw1901} proposed a constitutive equation for the Cauchy stress that includes density gradients. Korteweg's formulation was, however, incompatible with classical thermodynamics. In their seminal work, Dunn and Serrin \cite{DS85} circumvented this difficulty by introducing the concept of interstitial work into the balance of energy equation as an additional supply of mechanical energy, and the model that they derived is the one that we now review.

In Eulerian coordinates and in $d$ space dimensions, $d \geq 1$, the general system of equations expressing conservation of mass, conservation of momentum, balance of energy and the Clausius-Dulhem inequality for a compressible fluid has the form (cf. \cite{DS85}),
\begin{equation}
\label{gensystfluid}
\begin{aligned}
\rho_t + \nabla \cdot ( \rho \vu) &= 0,\\
(\rho \vu)_t + \nabla \cdot (\rho \vu \otimes \vu) &= \nabla \cdot \TT,\\
\big( \rho ( \vep + \tfrac{1}{2} |\vu|^2) \big)_t + \nabla \cdot \big( \rho ( \vep + \tfrac{1}{2} |\vu|^2) \vu ) &= - \nabla \cdot \vq + \nabla \cdot (\TT \vu),\\
\rho \theta \frac{Ds}{Dt} + \nabla \cdot \vq - \frac{\vq \cdot \nabla \theta}{\theta} &\geq 0,
\end{aligned}
\end{equation}
where
\[
\frac{D}{Dt} := \partial_t + \vu \cdot \nabla,
\]
denotes the material derivative operator, and
\begin{eqnarray*}
\rho = \rho(\vx,t) \in \R &  &\text{is the mass density (per unit volume) of the fluid,}\\
\vu = \vu(\vx,t) \in \R^d & &\text{is the velocity field of the fluid,}\\
\theta = \theta(\vx,t) > 0 & &\text{is the absolute temperature of the fluid,}\\
\vep > 0 & & \text{is the specific internal energy density (per unit mass),}\\
s \in \R & & \text{is the specific entropy (per unit mass),}\\
\vq \in \R^d & & \text{is the heat flux vector, and}\\
\TT \in \R^{d \times d} & & \text{is the Cauchy stress tensor.}
\end{eqnarray*}
Here, $\vx \in \R^d$ and $t > 0$ denote space and time variables. In the standard NSF model, the Cauchy stress tensor has the form $\TT = - p \II + \widetilde{\lambda} \, (\tr \DD) \II + 2 \widetilde{\nu} \, \DD$, where $\DD = \tfrac{1}{2} (\nabla \vu + \nabla \vu^\top)$ is the deformation tensor; $\tr(\cdot)$ is the trace operator; $\widetilde{\lambda}$ and $\widetilde{\nu}$ are the usual bulk and shear viscosity coefficients (depending on $\rho$ and $\theta$), respectively, which satisfy $\widetilde{\lambda} + 2 \widetilde{\nu} > 0$; $p$ is the pressure; and, $\II$ is the identity $d \times d$-dimensional tensor. Korteweg's original proposal \cite{Kortw1901} was to consider a Cauchy stress tensor depending upon higher order derivatives of the density, of the form
\[
\TT = - p \II + \widetilde{\lambda} \, (\tr \DD)\II + 2\widetilde{\nu} \, \DD + (\delta_1 \Delta \rho + \delta_2 |\nabla \rho|^{2}) \II + \delta_3 \nabla \rho \otimes \nabla \rho + \delta_4 \Hess( \rho ),
\]
where the coefficients $\delta_j$, $j = 1, \ldots, 4$, are material functions of the thermodynamic variables $\rho$ and $\theta$ and $\Hess (\cdot)$ denotes the Hessian square matrix of second-order partial derivatives of any scalar field. Applying standard arguments by Gurtin \cite{Gurt65} and Eringen \cite{Erng66}, Dunn and Serrin showed that the Clausius-Duhem inequality in \eqref{gensystfluid} forces the quantities $\vep$, $s$ and $\TT$ to depend on at most $\theta$ and $\rho$, yielding necessarily that $\delta_j \equiv 0$ for all $j$ (see \cite{DS85} for details). Whence, Dunn and Serrin proposed that the classical form of the balance laws of momenta be retained, while the energy balance be enlarged to incorporate long range effects in the medium through the concept of \emph{interstitial work flux}, denoted as $\vw$. The new system now reads,
\begin{equation}
\label{gensystfluid2}
\begin{aligned}
\rho_t + \nabla \cdot ( \rho \vu) &= 0,\\
(\rho \vu)_t + \nabla \cdot (\rho \vu \otimes \vu) &= \nabla \cdot \TT,\\
\big( \rho ( \vep + \tfrac{1}{2} |\vu|^2) \big)_t + \nabla \cdot \big( \rho ( \vep + \tfrac{1}{2} |\vu|^2) \vu ) &= - \nabla \cdot \vq + \nabla \cdot (\TT \vu) + \nabla \cdot \vw,\\
\rho \theta \frac{Ds}{Dt} + \nabla \cdot \vq - \frac{\vq \cdot \nabla \theta}{\theta} &\geq 0.
\end{aligned}
\end{equation}

The interstitial work flux is compatible with an extended principle of non-equilibrium thermodynamics that regards capillarity as a long range molecular interaction which penalizes high variations of the density (cf. \cite{Se83,Se81,ChHa11,ChDX21}). In terms of the Helmholtz free energy,
\[
\Psi = \vep - \theta s,
\]
for instance, this principle takes the form of a generalized volumetric Gibbs-Duhem relation (see, e.g., Jou \emph{et al.} \cite{JCL10}; see also \cite{HaEv82,Callen-2e}),
\[
d\Psi = - s d\theta + \frac{1}{\rho} dp + \frac{1}{\rho} \vphi \cdot \vg,
\]
where, in the additional term, $\vg$ stands for $\nabla \rho$ and $\vphi$ is an abstract potential. As a result, the free energy of the fluid not only depends on the standard variables but also on their higher order derivatives. In order to obtain the expression of the corresponding (capillarity) stress tensor, which we shall denote as $\tTT$, Dunn and Serrin assume a constitutive relation for the Helmholtz free energy of the form $\Psi = \Psi (\rho, \theta, \vg, \vh, \GG)$, with $\vg = \nabla \rho$, $\vh = \nabla \theta$ and $\GG = \Hess(\rho)$, as well as analogous relations for $s$, $\tTT$ and $\vw$. Upon application of fundamental physical principles and the Clausius-Duhem inequality, Dunn and Serrin proved that these response functions can depend at most on $\rho$, $\theta$ and $\vg = \nabla \rho$, and that the following relations must be observed,
\begin{equation}
\label{exprabc}
\begin{aligned}
s &= - \Psi_\theta,\\
\tTT &= -\rho (\nabla \rho \otimes \nabla_{\vg} \Psi) +  \big( \rho (\nabla \cdot (\rho  \nabla_{\vg} \Psi)) - \rho^2 \Psi_\rho\big) \II,\\
\vw &= \rho \frac{D\rho}{Dt} \vphi.
\end{aligned}
\end{equation}
Moreover, the objectivity of thermodynamic potentials (cf. Gurtin \cite{Gurt81}) demands that these physical quantities depend on $\nabla \rho$ only through $|\nabla \rho|^2$. For details, see Dunn and Serrin \cite{DS85} or Appendix A in \cite{Szk20a}. Hence, Dunn and Serrin proposed the potential $\vphi$ to have the form,
\begin{equation}
\label{choicepot}
\vphi = 2 \kappa \vg,
\end{equation}
with $\kappa = \kappa(\rho,\theta) > 0$ being a \emph{capillarity (or diffuse interface) coefficient}, depending on both $\rho$ and $\theta$, and reflecting the shear-rate dependence of the thermodynamic potential (see Jou \emph{et al.} \cite{JCL10}, p. 186). Consequently, the Helmholtz free energy is decomposed as
\begin{equation}
\label{Helmfe}
\Psi (\rho,\theta,\vg) = \psi(\rho,\theta) + \kappa(\rho,\theta) |\vg|^2 = \psi(\rho,\theta) + \kappa(\rho,\theta) |\nabla \rho|^2,
\end{equation}
where the function $\psi = \psi(\rho,\theta)$ represents the standard part of the free energy (that is, the usual Helmholtz free energy of equilibrium thermodynamics), and the dependence on the gradient is encoded in the non-standard term, $\kappa |\nabla \rho|^2$. In this fashion, the fluid is completely specified by its free energy $\Psi$, from which its internal energy $\vep$ derives via the Legendre transform, $\vep = \Psi + \theta s$, with temperature $\theta$ and specific entropy $s = - \Psi_\theta$ as dual variables. This yields similar decompositions for $\vep$ and $s$, namely,
\begin{subequations}
\label{TherRel}
\begin{align}
s(\rho, \theta, \vg) &= - \Psi_\theta = - \psi_\theta (\rho, \theta) - \kappa_\theta (\rho, \theta) |\vg|^2 =: \eta - \kappa_\theta |\nabla \rho|^2, \label{TherRelEnt}\\
\vep(\rho, \theta, \vg) &= \Psi + \theta s = \psi(\rho,\theta) - \theta \psi_\theta (\rho, \theta) + (\kappa (\rho, \theta) - \theta \kappa_\theta (\rho, \theta)) |\vg|^2 =: e + (\kappa - \theta \kappa_\theta)|\nabla \rho|^2, \label{TherRelEne}
\end{align}
\end{subequations}
where
\[
\begin{aligned}
\eta = \eta(\rho, \theta) &= - \psi_\theta(\rho,\theta),\\
e = e(\rho,\theta) &= \psi(\rho,\theta) - \theta \psi_\theta(\rho,\theta),
\end{aligned}
\]
denote the standard specific entropy and the standard specific internal energy, respectively. The pressure is then defined as
\begin{equation}
\label{defp}
p = p(\rho,\theta) = \rho^2 \psi_\rho(\rho,\theta) = \rho^2 \Psi_\rho(\rho,\theta,\boldsymbol{0}).
\end{equation}

As in \cite{DS85}, here we assume that the thermodynamic potentials involved in the standard part of the Helmholtz free energy satisfy the classical relations of equilibrium thermodynamics, starting with the First Law,
\begin{equation}
\label{StdTherRel}
de = \theta d\eta -p d\left(\frac{1}{\rho} \right), 
\end{equation}
together with the relations,
\[
e_\theta >0, \quad p_\rho > 0, \quad e_\rho = \frac{1}{\rho^2} (p - \theta p_\theta), \quad \eta_\theta = \frac{e_\theta}{\theta}, \quad \eta_\rho = - \frac{p_\theta}{\rho^2}.
\]

\begin{remark}
\label{remideal}
A typical example of a standard energy density is the case of a polytropic ideal gas, with equation of state given by $p(\rho, \theta) = R \rho \theta$, where $R > 0$ is the ideal gas constant. Under the calorically perfect gas approximation, this yields an internal energy density of the form $e(\rho, \theta) = R(\gamma -1)^{-1} \theta \equiv c_v \theta$, where $\gamma > 1$ is the (constant) adiabatic index (ratio of specific heats), $\gamma = c_p / c_v$, with $c_p$ and $c_v$ being the (in this case, constant) specific heat at constant pressure and volume, respectively. Here $c_v = (\partial e / \partial \theta) |_\rho = R(\gamma-1)^{-1}$ and Mayer's relation holds, $c_p - c_v = R$. Thus, the ideal gas law can be expressed as $p = (\gamma-1)\rho e$ (cf. \cite{Callen-2e,Pipp57}). In this case, the standard Helmholtz free energy can be taken as $\psi(\rho,\theta) = R \theta \big( \log \rho - (\gamma -1)^{-1} \log \theta\big)$, which satisfies the above relations as the reader may easily verify.
\end{remark}

\begin{remark}
\label{remconcave}
From classical equilibrium thermodynamics, it is known that the standard Helmhotz free energy is a strictly concave function of the temperature at constant volume (or equivalently, at constant density),
\[
\psi_{\theta \theta} = - \eta_\theta = - \frac{e_\theta}{\theta} < 0.
\]
This condition is tantamount to the specific heat at constant volume to be positive, $c_v(\rho,\theta) = \theta \eta_\theta = e_\theta > 0$, which is a necessary condition for an equilibrium state to be stable against spatial fluctuations (see, e.g., Lebon \emph{et al.} \cite{LeJC08}, p. 27; see also \cite{Callen-2e,Pipp57}). In the case of non-equilibrium thermodynamics, a similar condition was formulated by Hanley and Evans \cite{HaEv82} for thermodynamic systems under shear (such as Korteweg fluids). Using the generalized Gibbs relation, Hanley and Evans arrive at the (non-equilibrium) thermal stability condition
\[
C_v(\rho,\theta,\nabla \rho) = \theta \Big(\frac{\partial s}{\partial \theta}\Big)\Big|_{\rho} \geq 0
\]
(see Hanley and Evans \cite{HaEv82}, p. 3227; see also Jou \emph{et al.} \cite{JCL10}, \S 3.4), where $C_v = C_v(\rho,\theta,\vg)$ denotes the non-equilibrium specific heat. Therefore, from \eqref{Helmfe} and \eqref{TherRelEnt}, we obtain
\[
C_v = - \theta \Big( \psi_{\theta \theta} + \kappa_{\theta \theta} |\nabla \rho|^2 \Big) \geq 0.
\]
In order to fulfill this condition for any value of $|\nabla \rho|$ and in view of $\psi_{\theta \theta} < 0$, in this paper we assume that the capillarity coefficient is non-strictly concave with respect of the temperature, 
\[
\kappa_{\theta \theta} \leq 0.
\]
This is a condition imposed on the capillarity coefficient that guarantees the strict concavity of the non-standard Hemholtz free energy with respect to the absolute temperature. Moreover, notice that it is compatible with the assumption of a constant capillarity coefficient ($\kappa \equiv \kappa_0 > 0$), which appears frequently in the literature.
\end{remark}

\begin{remark}
The choice of the potential \eqref{choicepot} goes back to the work of van der Waals \cite{vdW1894},  who proposed an energy functional that penalizes high spatial variations on the density, typically at interfaces. This Helmholtz free energy functional allows us to select the correct physical solution, and has the form
\[
\mathcal{F}[\rho,\theta,\nabla\rho] = \int_\Omega
 \rho e (\rho,\theta) + \kappa(\rho,\theta) |\nabla \rho|^2 \, d\Omega,
 \]
 where $\Omega$ is a control volume (see, e.g., Anderson \emph{et al.} \cite{AMW98}).
\end{remark}

From the expression \eqref{Helmfe} we then have $\nabla_{\vg} \Psi = 2 \kappa \nabla \rho$ and $\Psi_\rho = \psi_\rho + \kappa_\rho |\nabla \rho|^2$. Upon substitution into \eqref{exprabc} we obtain
\begin{align}
\tTT &= - \rho^2 \psi_\rho \II - 2 \kappa \rho (\nabla \rho \otimes \nabla \rho) - \rho^2 \kappa_\rho |\nabla \rho|^2 \II + 2 \rho \nabla \cdot(\kappa \rho \nabla \rho) \II, \label{captensor}\\
\vw &= - 2 \kappa \rho^2 (\nabla \cdot \vu) \nabla \rho, \label{intwork}
\end{align}
where the continuity equation in \eqref{gensystfluid2} is used to derive \eqref{intwork}. Hence, $\tTT$ adopts the form $\tTT = -p \II + \KK$, with the Korteweg tensor defined as
\[
\KK :=  - 2 \kappa \rho (\nabla \rho \otimes \nabla \rho) - \rho^2 \kappa_\rho |\nabla \rho|^2 \II + 2 \rho \nabla \cdot(\kappa \rho \nabla \rho) \II.
\]
For convenience, we relabel the capillarity coefficient through
\begin{equation}
\label{defk}
k = k(\rho,\theta) := 2 \rho \kappa(\rho,\theta),
\end{equation}
so that the interstitial work flux now reads,
\begin{equation}
\label{intwork2}
\vw = - k \rho (\nabla \cdot \vu) \nabla \rho,
\end{equation}
and the Korteweg tensor takes the form
\begin{equation}
\label{Korttensor}
\KK =  - k (\nabla \rho \otimes \nabla \rho) + \rho \nabla \cdot (k \nabla \rho) \II - \tfrac{1}{2} \rho k_\rho  |\nabla \rho|^2  \II + \tfrac{1}{2} k |\nabla \rho|^2  \II.
\end{equation}

In addition, Dunn and Serrin reckon that dissipation or entropy production due to viscous friction and heat conduction can be modeled in the same way as in the conventional continuum equilibrium thermodynamics, so that we may assume Newton's law of viscosity through the addition of the viscosity tensor,
\begin{equation}
\label{visctensor}
\mathbf{S} = \widetilde{\lambda} \, (\tr \DD) \II + 2 \widetilde{\nu} \, \DD,
\end{equation}
with $\DD = \tfrac{1}{2}(\nabla \vu + \nabla \vu^\top)$, and heat conduction through Fourier's constitutive law,
\begin{equation}
\label{Fourierlaw}
\vq = - \alpha \nabla \theta,
\end{equation}
where $\alpha = \alpha(\rho,\theta) > 0$ is the heat conductivity coefficient. This principle holds because the interstitial work is precisely the work done by internal \emph{microforces} (in this case, \emph{microstresses}) which are considered to act between parts of the medium as in conventional continuum thermodynamics. Under the assumption that the same Clausius-Dulhem inequality is satisfied, these microforces are non-dissipative and, hence, they do not contribute to the entropy production (see, e.g., Suzuki \cite{Szk20b}). Actually, the concept of microforces introduced by Gurtin \cite{Gurt96} and applied to fluid flow by Gurtin \emph{et al.} \cite{GPV96} can also be alternatively used in order to obtain the Korteweg system of equations for compressible flow, so that the derivation by Dunn and Serrin is a particular case of the microforces formulation (a case in which \emph{external} microforces are absent; see Suzuki \cite{Szk20a,Szk20b} and Freist\"uhler and Kotschote \cite{FrKo19} for further information).

Therefore, after the addition of viscosity and heat conduction terms, the full Cauchy stress tensor takes the form
\begin{equation}
\label{fullT}
\TT = -p \II + \mathbf{S} + \KK,
\end{equation}
where $p$, $\mathbf{S}$ and $\KK$ are given by \eqref{defp}, \eqref{visctensor} and \eqref{Korttensor}, respectively. Upon substitution of \eqref{fullT} and \eqref{intwork2} in \eqref{gensystfluid2}, and since the derivation of the capillarity terms already guarantees that Clausius-Dulhem inequality is satisfied, Dunn and Serrin finally arrive at the thermodynamically consistent compressible Korteweg system of equations:
\begin{equation}
\label{fullKw}
\begin{aligned}
\rho_t + \nabla \cdot ( \rho \vu) &= 0,\\
(\rho \vu)_t + \nabla \cdot (\rho \vu \otimes \vu) + \nabla p &= \nabla \cdot \big( \widetilde{\lambda} \, (\nabla \cdot \vu) \II + \widetilde{\nu} (\nabla \vu + \nabla \vu^\top) \big) + \nabla \cdot \KK,\\
\big( \rho ( \vep + \tfrac{1}{2} |\vu|^2) \big)_t + \nabla \cdot \big( \rho ( \vep + \tfrac{1}{2} |\vu|^2) \vu + p\vu) &=  \nabla \cdot (\alpha \nabla \theta) + \nabla \cdot (\KK \vu)  + \nabla \cdot \vw \, + \\
&\quad + \big( \widetilde{\lambda} \, (\nabla \cdot \vu) \vu + \widetilde{\nu} (\nabla \vu + \nabla \vu^\top)\vu \big),
\end{aligned}
\end{equation}
with $\vw$ and $\KK$ given by \eqref{intwork2} and \eqref{Korttensor}, respectively. The fluid is fully specified by its Helmholtz free energy, $\Psi = \psi + \kappa |\nabla\rho|^2$, through its equilibrium free energy $\psi$ and its capillarity coefficient $\kappa$, as well as the viscosity and heat conducting coefficients. Notice that, if we assume that the non-standard part of the energy is absent, that is, when $\kappa \equiv 0$, then the full Korteweg system \eqref{fullKw} - \eqref{fullT} reduces to the standard NSF model for compressible heat conducting fluids.

Finally, if we specialize system \eqref{fullKw} to one space dimension ($d = 1$) then we arrive at model equations \eqref{NSFK} and \eqref{defKw}, where the combined viscosity coefficient $\mu$ is defined as
\begin{equation}
\label{defmu}
\mu = \mu(\rho,\theta) := \widetilde{\lambda}(\rho,\theta) + 2 \widetilde{\nu}(\rho,\theta) > 0.
\end{equation}
This is the model we study in the present paper.

To finish this section, let us state the main physical assumptions for the one-dimensional system under consideration.
\begin{itemize}
\item[(H$_1$)] \phantomsection
\label{H1} The independent thermodynamic variables are the mass density, $\rho > 0$, and the absolute temperature, $\theta > 0$. It is assumed that they belong to the fixed open domain
\begin{equation}
\label{defD}
\cD := \big\{ (\rho, \theta) \in \R^2 \, : \, \rho > \overline{C}_0, \, \theta > \overline{C}_1 \big\},
\end{equation} 
for uniform constants $\overline{C}_j > 0$, $j = 0,1$. The set of state variables is the open connected set
\begin{equation}
\label{defcU}
\cU := \big\{ (\rho, u, \theta) \in \R^3 \, : \, (\rho,\theta) \in \cD \big\}.
\end{equation}
The constants $\overline{C}_j > 0$ are fixed but arbitrary.
\item[(H$_2$)] \phantomsection
\label{H2} The fluid is specified by a non-standard Helmholtz free energy of the form 
\begin{equation}
\label{Helmfe1d}
\Psi(\rho,\theta,\rho_x) = \psi(\rho,\theta) + \kappa(\rho,\theta) \rho_x^2,
\end{equation}
for given thermodynamic functions, $\psi \in C^\infty(\overline{\cD})$ and $\kappa \in C^{\infty}(\overline{\cD})$, $\kappa > 0$, denoting the standard part of the free energy and a smooth, strictly positive capillarity coefficient, respectively. The specific entropy $s$ of the fluid is then given by expression 
\begin{equation}
\label{TherRelEnt1d}
s(\rho, \theta, \rho_x) = - \Psi_\theta = - \psi_\theta - \kappa_\theta  \rho_x^2 =: \eta - \kappa_\theta  \rho_x^2,
\end{equation}
whereas its internal energy $\vep$ is determined by 
\begin{equation}
\label{TherRelEne1d}
\vep(\rho, \theta, \rho_x) = \Psi + \theta s = \psi - \theta \psi_\theta + (\kappa - \theta \kappa_\theta) \rho_x^2 =: e + (\kappa - \theta \kappa_\theta) \rho_x^2.
\end{equation}
Moreover, the viscosity coefficient and the heat conductivity are strictly positive smooth functions of the independent thermodynamic variables,
\[
\mu, \alpha \in C^{\infty}(\overline{\cD}), \quad \mu > 0, \, \alpha > 0.
\]
 \item[(H$_3$)] \phantomsection
\label{H3} The standard part of the free energy, namely the function $\psi$, is such that the pressure, $p = \rho^2 \psi_\rho > 0$, $p \in C^{\infty}(\overline{\cD})$, the standard specific energy, $e = \psi - \theta \psi_\theta \in C^{\infty}(\overline{\cD})$, and the standard specific entropy, $\eta = - \psi_\theta \in C^{\infty}(\overline{\cD})$, fulfill the (equilibrium) volumetric First Law \eqref{StdTherRel} as well as the conditions
\begin{equation}
\label{Weyl}
p > 0, \quad p_\rho > 0, \quad p_\theta > 0, \quad e_\theta > 0,
\end{equation}
which imply, in turn, the relations
\begin{equation}
\label{las4}
e_\rho = \frac{1}{\rho^2} \big(p - \theta p_\theta \big), \quad \eta_\theta = \frac{e_\theta}{\theta}, \quad \eta_\rho = - \frac{p_\theta}{\rho^2},
\end{equation}
for all $(\rho,\theta)\in \cD$.
 \item[(H$_4$)] \phantomsection
\label{H4} It is assumed that the capillarity coefficient $\kappa \in C^\infty(\overline{\cD})$ is a non-strictly concave function of the temperature at constant density, that is,
\begin{equation}
\label{concavekappa}
\kappa_{\theta \theta} \leq 0, \qquad \text{for all } (\rho,\theta) \in \cD.
\end{equation}
\end{itemize}

\begin{remark}
Hypothesis \hyperref[H1]{\rm{(H$_1$)}} simply selects the independent thermodynamic variables. The mass density and absolute temperature are assumed to be uniformly bounded below in order to avoid both vacuum and absolute zero temperature. The constants $\overline{C}_j > 0$ are fixed but can be arbitrarily chosen depending upon the constant equilibrium state to be considered. Hypothesis \hyperref[H2]{\rm{(H$_2$)}} describes the non-equilibrium thermodynamics of the fluid according to the derivation by Dunn and Serrin \cite{DS85}. In \hyperref[H3]{\rm{(H$_3$)}} we impose the classical equilibrium assumptions on the standard part of the free energy. In fact, hypotheses \eqref{Weyl} are very general and define what are known as \emph{Weyl fluids} \cite{We49}, for which the internal energy increases due to an increase in temperature at constant volume ($e_\theta > 0$), there is an adiabatic increase of pressure due to compression ($p_\rho > 0$), and a generalized Gay-Lussac law holds ($p_\theta > 0$). Assumption \hyperref[H4]{\rm{(H$_4$)}} is imposed on the capillarity coefficient in order to ensure a thermal stability condition. As discussed in Remark \ref{remconcave}, this concavity condition on the coefficient $\kappa$ is sufficient for the full Helmholtz free energy to be strictly concave with respect to the temperature. Notice that \eqref{concavekappa} implies 
\[
k_{\theta \theta} \leq 0, \qquad \text{for all } (\rho,\theta) \in \cD,
\]
where $k$ is the modified capillarity coefficient defined in \eqref{defk}. Hypothesis \hyperref[H4]{\rm{(H$_4$)}} will allow us to define an invertible change of perturbation variables suitable for the stability analysis (see section \ref{secconvexvar} below).
\end{remark}

\begin{remark}
There have been attempts in the literature to provide a thermodynamic derivation of compressible fluids of Korteweg type without the introduction of the concept of interstitial work; see, for example, Heida and M\'alek \cite{HdMl10} and Freist\"uhler and Kotschote \cite{FrKo17,FrKo19}, resulting into new definitions of the Korteweg tensor $\KK$. These new systems satisfy the principle of objectivity in continuum dynamics, as well as the principle of material frame indifference. From a Rational Mechanics viewpoint, they are motivated by the requirement to keep the contributions of the stress tensor $\TT$ to momentum and energy balances related as $\nabla \cdot \TT$ and $\nabla \cdot (\TT \vu)$, respectively (compare the forms of systems \eqref{gensystfluid} and \eqref{gensystfluid2}, for instance). The model by Dunn and Serrin remains, however, as the standard thermodynamically consistent system for compressible fluids exhibiting both capillarity and energy exchange, essentially because there is physical evidence that microstresses do exist (see discussions in \cite{Szk20b,FrKo19,MoVi16} and in the references therein). Since the differences with these alternative systems involve higher order derivatives in the Korteweg tensor, we claim that the results of this paper remain valid for these other models, inasmuch as we linearize around constant equilibrium states for which higher order derivatives are treated as perturbations.
\end{remark}

\subsection{Local well-posedness of perturbations of constant equilibrium states}
\label{Well-poss-Sec}

In this section, we review the local well-posedness theory for system \eqref{NSFK} - \eqref{defKw} in terms of perturbations of a given equilibrium state. The global existence of solutions to compressible fluid models of Korteweg type has been addressed by many authors (for an abridged list of references see \cite{HaLi94,HaLi96a,HaLi96b,DD01,CHZ17,CCD15,Kot08,Kot10}), particularly in the isothermal case. Here, we follow the analysis of the non-isothermal system by Hattori and Li \cite{HaLi96b} and by Chen \emph{et al.} \cite{CHZ17}. Although these works establish the global existence of smooth perturbations of constant equilibrium states, we limit ourselves to invoke the local well-posedness in appropriate function spaces, because the global existence (and, more importantly, the decay) of such perturbations will be a consequence of the stability estimates that we obtain in the forthcoming sections.

Let $\bU = (\brho, \bu, \bthe) \in \cU$ be a constant equilibrium state.  For positive constants $T > 0$, $M_i \geq m_i > 0$, $i = 1,2$ and any $s \geq 3$ let us define the following space of functions:
\begin{equation}
\label{defXs}
\begin{aligned}
X_s\big((0,T); \, m_1,M_1, & m_2,M_2\big) :=  \\
\Big\{ (\rho,u,\theta) :  \, &\rho-\brho \in C((0,T); H^{s+1}(\R)) \cap C^1((0,T); H^{s-1}(\R)), \\
& (u-\bu,\theta - \bthe) \in  C((0,T); H^s(\R) \times H^s(\R)) \cap C^1((0,T); H^{s-2}(\R) \times H^{s-2}(\R)),\\
&(\rho_x, u_x,\theta_x) \in L^2((0,T); H^{s+1}(\R) \times H^s(\R) \times H^s(\R)), \; \text{and} \\
&m_1 \leq \rho(x,t) \leq M_1, \, m_2 \leq \theta(x,t) \leq M_2, \; \text{a.e. in } \, (x,t) \in \R \times (0,T) \Big\}.
\end{aligned}
\end{equation}

In order to distinguish from the standard norm of a vector valued function $V = (V_1, V_2, V_3)$ in $H^s(\R)^3$, namely, $\| V \|_s^2 := \sum \| V_j \|_s^2$, let us introduce the following notation,
\begin{equation}
\label{triplenorm}
\vertiii{V}_\ell^2 := \| V_1 \|_{\ell+1}^2 + \| V_2 \|_\ell^2 + \| V \|_\ell^2,\qquad V \in H^{s+1}(\R) \times H^{s}(\R)^2, \quad  0 \leq \ell \leq s.
\end{equation} 
Henceforth, for any $U = (\rho, u, \theta) \in X_s\big((0,T); m_1,M_1,m_2,M_2\big)$ we denote
\begin{align}
E_s(t) &:= \sup_{\tau \in [0,t]} \Big[ \| \rho(\tau)-\brho\|^2_{s+1} + \| u(\tau)-\bu\|^2_{s} + \| \theta(\tau)-\bthe\|^2_{s}\Big] = \sup_{\tau \in [0,t]}  \vertiii{U(\tau) - \bU}_s^2, \label{defEs}\\
F_s(t) &:= \int_0^t \Big[ \| \rho_x(\tau) \|_{s+1}^2 + \| u_x(\tau) \|_s^2 + \| \theta_x(\tau) \|_s^2 \Big] \, d\tau = \int_0^t \, \vertiii{U_x(\tau)}_s^2 \, d\tau ,\label{defFs}
\end{align}
for all $t \in [0,T]$. 

The following local existence theorem for perturbations of equilibria can be obtained by the application of a dual argument and by an iteration technique, as performed in the proof of Theorem 2.1 by Hattori and Li \cite{HaLi96b}. Details are omitted.

\begin{theorem}
\label{themlocale}
Under assumptions \hyperref[H1]{\rm{(H$_1$)}} - \hyperref[H3]{\rm{(H$_3$)}}, let $\bU = (\brho, \bu, \bthe)$ be a constant equilibrium state with $\brho >0$ and $\bthe > 0$. Suppose that
\begin{equation}
\label{incond}
\rho_0 - \brho \in H^{s+1}(\R), \quad (u_0 - \bu, \theta_0 - \bthe) \in H^s(\R) \times H^s(\R),
\end{equation}
for some $s \geq 3$. Then there exists a positive constant $a_0 > 0$ such that, if
\[
E_s(0)^{1/2} = \Big[ \| \rho_0-\brho\|^2_{s+1} + \| u_0-\bu\|^2_{s} + \| \theta_0-\bthe\|^2_{s} \Big]^{1/2} < a_0,
\]
then there hold $m_1 \leq \rho_0(x) \leq M_1$ and $m_2 \leq \theta_0(x) \leq M_2$, \emph{a.e.} in $x \in \R$ for some positive constants $M_i \geq m_i > 0$, $i = 1,2$, and there exists a positive time $T_1 = T_1(a_0) > 0$ such that a unique smooth solution,
\[
U = (\rho,u,\theta) \in X_s\big((0,T_1); \tfrac{1}{2}m_1,2M_1,\tfrac{1}{2}m_2,2M_2\big),
\]
exists for the Cauchy problem of system \eqref{NSFK} - \eqref{defKw} with initial data $U(0) = (\rho_0,u_0,\theta_0)$. Moreover, the solution satisfies the energy estimate
\begin{equation}
\label{localEE}
E_s(T_1) + F_s(T_1) \leq C_1 E_s(0),
\end{equation}
for some constant $C_1 > 0$ depending only on $a_0$.
\end{theorem}

\begin{corollary}[a priori estimate]
\label{cor26}
Under \hyperref[H1]{\rm{(H$_1$)}} - \hyperref[H3]{\rm{(H$_3$)}}, let 
\[
U = (\rho,u,\theta) \in X_s\big((0,T); \tfrac{1}{2}m_1,2M_1,\tfrac{1}{2}m_2,2M_2\big),
\]
be a local solution of the initial value problem of \eqref{NSFK} - \eqref{defKw} with initial data $U(0) = (\rho_0,u_0,\theta_0)$ satisfying \eqref{incond}. Then there exists $0 < a_2 \leq a_0$ sufficiently small such that, if for any $0 < t  \leq T$ we have $E_s(t)^{1/2} \leq a_2$, then there holds
\begin{equation}
\label{localaprioriEE}
\big( E_s(t) + F_s(t) \big)^{1/2} \leq C_2 E_s(0)^{1/2},
\end{equation}
where $C_2 = C_2(a_2) > 0$ is a positive constant independent of $t > 0$.
\end{corollary}

\begin{proof}
See Proposition 2.3 in Hattori and Li \cite{HaLi96b} (see also the related Proposition 2.2 in \cite{CHZ17}).
\end{proof}

\begin{remark}
\label{remclassical}
Notice that Hattori and Li \cite{HaLi96b} demand high regularity on the initial condition and on the solutions ($s \geq 3$), yielding classical local solutions to the Cauchy problem by virtue of Sobolev's embedding theorem: since $s \geq 3$ then we have the continuous embeddings $H^{s+j}(\R) \hookrightarrow C^{j+2}(\R)$ for any $j \geq 0$, with 
\[
\sup_{k \leq j+2} \; \sup_{x \in\R^n} |\partial_x^k u| \leq C \| u \|_{s+j},
\]
for all $u \in H^{s+j}(\R)$ and some $C = C(s,j) > 0$. Hence, the local solutions to the Cauchy problem for system \eqref{NSFK} satisfy $\rho(t) \in H^{s+1} \hookrightarrow C^{3}(\R)$, $u(t), \theta(t) \in H^s(\R) \hookrightarrow C^2(\R)$, for all $t \in [0,T]$.
\end{remark}

\begin{remark}
\label{remLinftybound}
From estimate \eqref{localEE} we have
\[
\| \partial_x^j (\rho(t) - \brho) \|_0 \leq C_1 E_s(0),
\]
with $C_1 = C_1(a_0)> 0$, for all $t \in [0,T]$ and $0 \leq j \leq 3$. Therefore, from the classical Sobolev's inequality, namely $\| v \|_\infty \leq 2 \| v \|_0^{1/2} \| v_x \|_0^{1/2}$ for all $v \in H^1(\R)$, we readily obtain
\[
\| \partial_x^j (\rho(t) - \brho) \|_\infty \leq C,
\]
for all $t\in [0,T]$, $0 \leq j \leq 3$, and some constant $C = C(a_0,s) > 0$ independent of $T$. In the same fashion we have
\[
\| \partial_x^j (u(t) - \bu) \|_\infty, \, \| \partial_x^j (\theta(t) - \bthe) \|_\infty \leq C, \qquad j = 0,1,2, \quad t \in [0,T].
\]

Finally, using the above $L^\infty$-estimates, it is not hard to verify that we can choose $a_0 > 0$ in the statement of Theorem \ref{themlocale} sufficiently small such that
\[
(\rho, u, \theta)(t) \in \cU, \quad \text{a.e. in }\, x \in \R, \quad \forall t \in [0,T].
\]
This holds mainly because the constants $\overline{C}_j > 0$, which define the space of state variables $\cU$ in \eqref{defD} and \eqref{defcU}, are arbitrary. The constant $a_0 > 0$ in the original theorem in \cite{HaLi96b} depends only on $\bU$. Hence, for a given equilibrium state $\bU$ with $\brho, \bthe > 0$ we may choose appropriate constants $\overline{C}_j > 0$, and $a_0$ sufficiently small, such that $\tfrac{1}{2}m_j > \overline{C}_j$, yielding $U(t) \in \cU$ for all $t \leq T$. 
\end{remark}

\section{Partially symmetric perturbation system}
\label{secconvexvar}

In this section we reformulate the nonlinear Korteweg system \eqref{NSFK} - \eqref{defKw} in terms of perturbations of constant equilibrium states. For that purpose, we employ the convex extension and symmetrization procedure implemented in \cite{KaSh88a} for the NSF system (for convenience of the reader, we recall such calculation in Appendix \ref{ConvExtN-S}). Essentially, we use the Hessian of the equilibrium convex entropy as a (in our case, partial) symmetrizer, exactly as in the NSF case: since the change of variables involves matrix functions evaluated at equilibrium states, the non-standard part of the entropy vanishes.

\subsection{Convex extension and new perturbation variables}
\label{secnewpert}

Once again, let $\bU = (\brho, \bu, \bthe) \in \cU$ be a constant equilibrium state. For given initial conditions $U(0) = (\rho_0,u_0,\theta_0)$ satisfying \eqref{incond} for some $s \geq 3$, Theorem \ref{themlocale} guarantees the existence of an evolution,
\[
U = (\rho,u,\theta) \in X_s\big((0,T); \tfrac{1}{2}m_1,2M_1,\tfrac{1}{2}m_2,2M_2\big),
\]
which is a local solution to the Cauchy problem for some $T > 0$, with constants $m_j, M_j$, $j = 1,2$, such that $U(\cdot, t) \in \cU$ a.e. in $x\in \R$ and for all $t \in [0,T]$ (see Remark \ref{remLinftybound}). Moreover, it is a classical solution to system \eqref{NSFK}. After some straightforward calculations, we recast the former as the following system in conservation form,
\begin{equation}
\label{NSFKc}
F^0(U,U_x)_t + F^1(U,U_x)_x = (G(U)U_x)_x + (H(U)U_{xx})_x + \widetilde{g}(U,U_x)_x.
\end{equation}
Here the conserved quantities are
\begin{equation}
\label{defF0}
F^0(U,U_x) := \begin{pmatrix} 
\rho \\ \rho u \\ \rho (\vep + \tfrac{1}{2} u^2)
\end{pmatrix},
\end{equation}
whereas
\begin{equation}
\label{defF1}
F^1(U,U_x) := \begin{pmatrix} 
\rho u \\ \rho u^2 + p \\ \rho u(\vep + \tfrac{1}{2} u^2) + pu
\end{pmatrix}, 
\end{equation}
\begin{equation}
\label{defGH}
G(U) := \begin{pmatrix} 
0 & 0 & 0 \\ 0 & \mu & 0 \\ 0 & \mu u & \alpha
\end{pmatrix}, \quad H(U) := \begin{pmatrix} 
0 & 0 & 0 \\ k \rho & 0 & 0 \\ k \rho u & 0 & 0
\end{pmatrix}.
\end{equation}

From assumptions \hyperref[H1]{\rm{(H$_1$)}} - \hyperref[H4]{\rm{(H$_4$)}}, these matrix and vector fields are smooth with respect to $U$ and $U_x$. It is to be observed that the dependence of the functions $F^j$, $j = 0,1$, on $U_x$ is solely determined by the (non-standard) internal energy, 
\[
\vep = \vep(\rho,\theta,\rho_x) = e(\rho,\theta) + (\kappa(\rho,\theta) - \theta\kappa_\theta(\rho,\theta)) \rho_x^2,
\]
which only involves density gradients. Hence $F^j \in C^\infty(\widetilde{\cU}; \R^3)$, $j = 0,1$, where $\widetilde{\cU} := \{ (\rho, u, \theta, \rho_x) \in \R^4 \, : \, (\rho, u, \theta) \in \cU \}$ and $G, H \in C^\infty(\cU;\R^{3 \times 3})$ because the latter do not depend on density gradients. Finally, the term $\widetilde{g} = \widetilde{g}(U,U_x)$ has the form
\begin{equation}
\label{deftildeg}
\widetilde{g}(U,U_x) := \begin{pmatrix} 
0 \\ \tfrac{1}{2} \rho \rho_x^2 k_\rho + \rho \rho_x \theta_x k_\theta - \tfrac{1}{2} k \rho_x^2 \\ 
\tfrac{1}{2} u \rho \rho_x^2 k_\rho + u\rho \rho_x \theta_x k_\theta - \tfrac{1}{2} k u \rho_x^2 - \rho \rho_x u_x k
\end{pmatrix}.
\end{equation}
Notice that $\widetilde{g}(U,U_x)$ includes gradients of both density and temperature. Moreover, $\widetilde{g}(U,U_x) = O(|U_x|^2)$ and the terms of order $O(|U||U_x|)$ are included in the expressions involving $G$ and $H$.

The following result guarantees the invertibility of an appropriate change of coordinates.
\begin{proposition}
\label{propgoodW}
Under assumptions \hyperref[H1]{\rm{(H$_1$)}} - \hyperref[H4]{\rm{(H$_4$)}}, let 
\[
U = (\rho,u,\theta) \in X_s\big((0,T); \tfrac{1}{2}m_1,2M_1,\tfrac{1}{2}m_2,2M_2\big),
\]
be the local solution to the Cauchy problem \eqref{NSFK} with suitable initial conditions satisfying the hypotheses of Theorem \ref{themlocale}, with constants $m_j, M_j$, $j = 1,2$, such that $U(\cdot, t) \in \cU$ \emph{a.e.} in $x\in \R$ and for all $t \in [0,T]$, for any $T \leq T_1$. Then the matrix field $D_U F^0(U,U_x)$ is invertible.
\end{proposition}
\begin{proof}
From Theorem \ref{themlocale} we know that the solution is classical (thanks to Sobolev's embedding theorem with $s \geq 3$). Thus, we may compute
\[
D_U F^0(U,U_x) = \begin{pmatrix}
1 & 0 & 0 \\ u & \rho & 0 \\ \vep + \tfrac{1}{2}u^{2} + \rho \vep_{\rho} & \rho u & \rho \vep_{\theta}
\end{pmatrix}.
\]
In view that hypotheses \hyperref[H3]{\rm{(H$_3$)}} and \hyperref[H4]{\rm{(H$_4$)}} hold, its determinant is clearly positive
\[
\det \big[D_U F^0(U,U_x) \big]= \rho^2 \vep_\theta = \rho^2 (e_\theta - \kappa_{\theta \theta} \rho_x^2) > 0.
\]
Therefore, $D_U F^0(U,U_x)^{-1}$ exists for all solutions $U$. 
\end{proof}

\begin{remark}
\label{remdiffcnsrvd}
Notice that the conserved variables in the NSFK system differ from the conserved variables for the NSF model due to the presence of density gradients,
\begin{equation}
\label{decF0}
F^0(U,U_x) = f^0(U) + \Gamma^0(U,U_x), \qquad
 \Gamma^0(U,U_x) := \begin{pmatrix} 0 \\ 0 \\ \rho(\kappa - \theta \kappa_\theta) \rho_x^2
\end{pmatrix},
\end{equation}
where $f^0 = f^0(U)$ denotes the conserved quantities for the NSF system as defined in \eqref{3}. In other words, the difference lies in the non-standard part of the energy. Since the derivatives of a constant state $\bU$ vanish, these conserved quantities coincide when they are evaluated at equilibrium states,
\[
F^0(\bU,0) = f^0(\bU).
\]
The same happens to the nonlinear flux function
\begin{equation}
\label{decF1}
F^1(U,U_x) = f^1(U) + \Gamma^1(U,U_x), \qquad
 \Gamma^1(U,U_x) := \begin{pmatrix} 0 \\ 0 \\ \rho u(\kappa - \theta \kappa_\theta) \rho_x^2
\end{pmatrix}
\end{equation}
(where $f^1 = f^1(U)$ is defined in \eqref{3}), and to their Jacobians with respect to $U$. In other words, we have $F^j(\bU,0) =f^j(\bU)$ and $D_UF^j(\bU,0) = D_U f^j(\bU)$, for each $j = 0,1$.

For shortness, in the sequel and for the rest of the paper we shall write all physical quantities evaluated at the constant state $\bU = (\brho, \bu, \bthe)$ with overlined variables (for example, $\bp = p(\brho,\bthe)$, $\be_\theta = e_\theta(\brho,\bthe)$, $p_\rho(\brho,\bthe) = \bp_\rho$, etc.)
\end{remark}

Next, let us recast system \eqref{NSFKc} for perturbations of a given constant equilibrium state $\bU = (\brho, \bu, \bthe) \in \cU$. For instance, in view that Proposition \ref{propgoodW} holds, we may write
\[
\begin{aligned}
F^1(U,U_x)_x &= \partial_x \big( F^1(U,U_x) - F^1(\bU,0)\big)\\
&= D_U F^1(\bU,0) D_UF^0(\bU,0)^{-1} F^0(U,U_x)_x \, + \\ &\qquad + \partial_x \Big[ F^1(U,U_x) - F^1(\bU,0) - D_U F^1(\bU,0) D_UF^0(\bU,0)^{-1} \big(F^0(U,U_x) - F^0(\bU,0) \big)\Big].
\end{aligned}
\]
Likewise, the viscous and capillar terms can be written as,
\[
\begin{aligned}
(G(U)U_x)_x &= G(\bU) D_U F^0(\bU,0)^{-1} F^0(U,U_x)_{xx} \, +\\
&\quad + \partial_x \Big[ G(U) D_U F^0(U,U_x)^{-1} D_UF^0(U,U_x) U_x - G(\bU) D_UF^0(\bU,0)^{-1} F^0(U,U_x)_x \Big]\\
&= G(\bU) D_U F^0(\bU,0)^{-1} F^0(U,U_x)_{xx} \, + \\
&\quad + \partial_x \Big[ \big( G(U) D_UF^0(U,U_x)^{-1} - G(\bU) D_U F^0(\bU,0)^{-1} \big) D_U F^0(U,U_x) U_x +\\
&\quad \qquad - G(\bU) D_U F^0(\bU,0)^{-1} D_{U_x} F^0(U,U_x) U_{xx} \Big],
\end{aligned}
\]
and as
\[
\begin{aligned}
(H(U)U_{xx})_x &= (H(U)U_{xx})_x + H(\bU) D_U F^0(\bU,0)^{-1} F^0(U,U_x)_{xxx} - H(\bU) D_U F^0(\bU,0)^{-1} F^0(U,U_x)_{xxx}\\
&= H(\bU) D_U F^0(\bU,0)^{-1} F^0(U,U_x)_{xxx} \, + \\
&\quad + \partial_x \Big[ H(U) D_U F^0(U,U_x)^{-1} D_U F^0(U,U_x) U_{xx} - H(\bU) D_UF^0(\bU,0)^{-1} F^0(U,U_x)_{xx} \Big]\\
&= H(\bU) D_U F^0(\bU,0)^{-1} F^0(U,U_x)_{xxx} \, + \\
&\quad + \partial_x \Big[ \big[ H(U) D_U F^0(U,U_x)^{-1} - H(\bU) D_UF^0(\bU,0)^{-1} \big] D_U F^0(U,U_x) U_{xx} + \\
&\qquad \qquad - H(\bU) D_UF^0(\bU,0)^{-1} \big[ \partial_x (D_U F^0(U,U_x)) U_x + \\
&\qquad \qquad + D_{U_x} F^0(U,U_x) U_{xxx} + \partial_x (D_{U_x} F^0(U,U_x))  U_{xx}\big] \Big],
\end{aligned}
\]
respectively. Henceforth, system \eqref{NSFKc} is recast as
\begin{equation}
\label{newNSFKc}
\begin{aligned}
F^0(U,U_x)_t &+ D_U F^1(\bU,0) D_UF^0(\bU,0)^{-1} F^0(U,U_x)_x + \\&- G(\bU) D_U F^0(\bU,0)^{-1} F^0(U,U_x)_{xx} +\\
&- H(\bU) D_U F^0(\bU,0)^{-1} F^0(U,U_x)_{xxx}\\
&= \partial_x \Big[ \widetilde{r}(U,U_x) + \widetilde{R}(U,U_x) U_x + \widetilde{I} (U,U_x,U_{xx}) + \widetilde{g}(U,U_x)\Big],
\end{aligned}
\end{equation}
where
\[
\begin{aligned}
\widetilde{r}(U,U_x) &:= - \big(F^1(U,U_x) - F^1(\bU,0) \big) + D_U F^1(\bU,0) D_UF^0(\bU,0)^{-1} \big(F^0(U,U_x) - F^0(\bU,0) \big),\\
\widetilde{R}(U,U_x) &:= \big[ G(U) D_UF^0(U,U_x)^{-1} - G(\bU) D_U F^0(\bU,0)^{-1} \big] D_U F^0(U,U_x),
\end{aligned}
\]
\[
\begin{aligned}
\widetilde{I} (U,U_x,U_{xx}) &:= - G(\bU) D_U F^0(\bU,0)^{-1} D_{U_x} F^0(U,U_x) U_{xx} + \\
&\quad + \big[ H(U) D_U F^0(U,U_x)^{-1} - H(\bU) D_UF^0(\bU,0)^{-1} \big] D_U F^0(U,U_x) U_{xx} +\\
&\quad - H(\bU) D_UF^0(\bU,0)^{-1} \big[ \partial_x (D_U F^0(U,U_x)) U_x + D_{U_x} F^0(U,U_x) U_{xxx} + \partial_x (D_{U_x} F^0(U,U_x))  U_{xx}\big],
\end{aligned}
\]
and $\widetilde{g} = \widetilde{g}(U,U_x)$ is given by \eqref{deftildeg}.

Let us now define the new perturbation variables as
\begin{equation}
\label{defWvar}
W:= D_UF^0(\bU,0)^{-1} \big( F^0(U,U_x) - F^0(\bU,0)\big).
\end{equation}
We now show that system \eqref{newNSFKc} can be written in partially symmetric form in terms of the $W$ variables. To that end, we notice that $D_U F^j(\bU,0)^{-1} = D_U f^j(\bU)^{-1}$ for $j =0,1$, and hence the coefficients of system \eqref{newNSFKc}, when evaluated at equilibrium states, coincide with those of the NSF system. Consequently, we may employ the Hessian of the standard entropy function, $\cE = - \rho \eta$, as symmetrizer (see Appendix \ref{ConvExtN-S}). If we regard $\cE$ as a function of the conserved quantities for the NSF system, $V = f^0(U)$, then the properties of its Hessian $D_V^2 \cE$ yield,
\[
\begin{aligned}
D_Uf^0(\bU)^\top &D_V^2 \cE(f^0(\bU)) F^0(U,U_x)_t = \\
&= D_U f^0(\bU)^\top D_V^2 \cE(f^0(\bU)) D_U F^0(\bU,0) \partial_t \Big[ D_U F^0(\bU,0)^{-1} \big( F^0(U,U_x) - F^0(\bU,0) \big) \Big]\\
&= D_U f^0(\bU)^\top D_V^2 \cE(f^0(\bU)) D_U f^0(\bU) W_t\\
&= A^0(\bU) W_t,
\end{aligned}
\]
where $A^0(\bU)$ is given by \eqref{15_1}. Likewise,
\[
\begin{aligned}
D_Uf^0(\bU)^\top &D_V^2 \cE(f^0(\bU)) D_U F^1(\bU,0) D_UF^0(\bU,0)^{-1} F^0(U,U_x)_x = \\
&= D_Uf^0(\bU)^\top D_V^2 \cE(f^0(\bU)) D_U f^1(\bU) \partial_x \Big[ D_U F^0(\bU,0)^{-1} \big( F^0(U,U_x) - F^0(\bU,0) \big) \Big]\\
&= D_Uf^0(\bU)^\top D_V^2 \cE(f^0(\bU)) D_U f^1(\bU) W_x\\
&= A^1(\bU) W_x,
\end{aligned}
\]
\[
\begin{aligned}
D_Uf^0(\bU)^\top &D_V^2 \cE(f^0(\bU)) G(\bU) D_U F^0(\bU,0)^{-1} F^0(U,U_x)_{xx} = \\
&= D_Uf^0(\bU)^\top D_V^2 \cE(f^0(\bU)) G(\bU) \partial_x^2 \Big[ D_U F^0(\bU,0)^{-1} \big( F^0(U,U_x) - F^0(\bU,0) \big) \Big]\\
&= D_Uf^0(\bU)^\top D_V^2 \cE(f^0(\bU)) G(\bU)  W_{xx}\\
&= B(\bU) W_{xx},
\end{aligned}
\]
and
\[
\begin{aligned}
D_Uf^0(\bU)^\top &D_V^2 \cE(f^0(\bU)) H(\bU) D_U F^0(\bU,0)^{-1} F^0(U,U_x)_{xxx} = \\
&= D_Uf^0(\bU)^\top D_V^2 \cE(f^0(\bU)) H(\bU) \partial_x^3 \Big[ D_U F^0(\bU,0)^{-1} \big( F^0(U,U_x) - F^0(\bU,0) \big) \Big]\\
&= D_Uf^0(\bU)^\top D_V^2 \cE(f^0(\bU)) H(\bU)  W_{xxx}\\
&= C(\bU) W_{xxx},
\end{aligned}
\]
where $A^1(\bU)$ and $B(\bU)$ are given by \eqref{15_2} and \eqref{15_3}, respectively (see Appendix \ref{ConvExtN-S} for details). Moreover, here the capillarity coefficient matrix is defined as 
\begin{equation}
\label{defC}
C(\bU) := D_Uf^0(\bU)^\top D_V^2 \cE(f^0(\bU)) H(\bU).
\end{equation}
As a result, if we multiply system \eqref{newNSFKc} on the left by $D_Uf^0(\bU)^\top D_V^2 \cE(f^0(\bU))$ we obtain 
\begin{equation}
\label{Wsystem}
\begin{aligned}
A^0(\bU) W_t + A^1(\bU) W_x &- B(\bU) W_{xx} - C(\bU) W_{xxx} = \\&= \partial_x \Big[ r(U,U_x) + R(U,U_x) U_x + I(U,U_x,U_{xx}) + g(U,U_x) \Big],
\end{aligned}
\end{equation} 
where
\begin{equation}
\label{nonlinterms}
\begin{aligned}
r(U,U_x) &:= D_Uf^0(\bU)^\top D_V^2 \cE(f^0(\bU)) \widetilde{r}(U,U_x),\\
R(U,U_x) &:= D_Uf^0(\bU)^\top D_V^2 \cE(f^0(\bU)) \widetilde{R}(U,U_x),\\
I(U,U_x,U_{xx}) &:= D_Uf^0(\bU)^\top D_V^2 \cE(f^0(\bU)) \widetilde{I}(U,U_x,U_{xx}),\\
g(U,U_x) &:= D_Uf^0(\bU)^\top D_V^2 \cE(f^0(\bU)) \widetilde{g}(U,U_x).
\end{aligned}
\end{equation}
For convenience of the reader, let us write here the coefficients of system \eqref{Wsystem} (which are computed in Appendix \ref{ConvExtN-S}):
\begin{equation}
\label{barA0A1B}
\begin{aligned}
A^{0}(\bU) &= \frac{1}{\bthe}\begin{pmatrix}
\bp_{\rho}/ \brho & 0 & 0 \\ 0 & \brho & 0 \\ 0 & 0 &  \be_{\theta} \brho / \bthe
\end{pmatrix},\\
A^{1}(\bU) &= \frac{1}{\bthe} \begin{pmatrix}
 \bp_{\rho} \bu / \brho & \bp_{\rho} & 0 \\ \bp_{\rho} & \brho \, \bu  & \bp_{\theta}  \\ 0 & \bp_{\theta} & \brho \, \bu \, \be_{\theta}  / \bthe
\end{pmatrix},\\
B(\bU) &= \frac{1}{\bthe}  \begin{pmatrix}
0 & 0 & 0 \\ 0 & \bmu & 0 \\ 0 & 0 & \balp/ \bthe
\end{pmatrix}.
\end{aligned}
\end{equation}
Furthermore, after some straightforward algebra, we use identities \eqref{idstar}, \eqref{7Z} and \eqref{inverseDf0} in Appendix \ref{ConvExtN-S} in order to obtain the capillarity coefficient matrix,
\begin{equation}
\label{barC}
C(\bU) = \frac{1}{\bthe}\begin{pmatrix}
0 & 0 & 0 \\ \bk \brho & 0 & 0 \\ 0 & 0 & 0
\end{pmatrix}.
\end{equation}

To sum up, for a given solution $U$ to the Cauchy problem for system \eqref{NSFKc}, hypotheses \hyperref[H1]{\rm{(H$_1$)}} - \hyperref[H4]{\rm{(H$_4$)}} allow us to define new perturbation variables $W$ and to recast the equations as a partially symmetric system for $W$ of the form \eqref{Wsystem}. The change of variables is valid thanks to hypothesis \hyperref[H4]{\rm{(H$_4$)}} (see Proposition \ref{propgoodW}). Notice that the left hand side of the resulting system \eqref{Wsystem} contains constant coefficient matrices, from which $A^{0}(\bU)$, $A^{1}(\bU)$ and $B(\bU)$ are symmetric, whereas $C(\bU)$ is not (and hence the term partially symmetric). Also note that $A^{0}(\bU)$ is positive definite and invertible, and that $B(\bU)$ is positive semi-definite. It is important to observe, in addition, that the right hand side of \eqref{Wsystem} is written in divergence form. 

\subsection{Regularity of $W$ and estimates on the nonlinear terms}
\label{secregW}

In order to close the energy estimates at a nonlinear level (see Section \ref{secglobal} below), it is crucial to examine the nonlinear terms appearing on the right hand side of system \eqref{Wsystem}. 

\begin{lemma}
\label{lemorderN}
The nonlinear terms appearing on the right hand side of system \eqref{Wsystem}, namely,
\begin{equation}
\label{deftiN}
\tiN(U,U_x,U_{xx}) := r(U,U_x) + R(U,U_x) U_x + I(U,U_x,U_{xx}) + g(U,U_x),
\end{equation}
are of order
\[
\begin{aligned}
\tiN(U,U_x,U_{xx}) = O\Big(|U-\bU|^2 + |U-\bU||U_x| + |U-\bU||\rho_{xx}| + |U_x|^2 + |\rho_x||\rho_{xx}|\Big) \begin{pmatrix} 0 \\ 1 \\1 
\end{pmatrix},
\end{aligned}
\]
as $|\partial_x^j(U - \bU)| \to 0$, $j =0,1$.
\end{lemma}
\begin{proof}
See Appendix \ref{appB}.
\end{proof}

\begin{remark}
\label{remfirstentryzero}
Notice that the first entry of the nonlinear terms in system \eqref{Wsystem} is equal to zero. This feature will play a crucial role in the establishment of the energy estimates with the appropriate regularity for the density variable in Section \ref{secglobal}. A quick way to verify this property is simply to compute the evolution equation for the first coordinate of the $W$ variables. Indeed, from inspection of the formula for $D_U F^0(\bU,0)^{-1} = D_Uf^0(\bU)^{-1}$ in \eqref{inverseDf0} and substitution into \eqref{defWvar} we find that the first and second coordinates of $W$ are $W_1 = \rho - \brho$ and $W_2 = \rho (u - \bu) / \brho$, respectively. Hence, using the expressions for the coefficients in \eqref{barA0A1B} (notice that the first rows of $B(\bu)$ and of $C(\bU)$ are zero) one easily arrives at the left hand side of the first equation of system \eqref{Wsystem}, 
\[
\frac{\bp_\rho}{\brho \,\bthe} \partial_t W_1 + \frac{\bp_\rho \bu}{\brho \,\bthe} \partial_x W_1 + \frac{\bp_\rho}{\bthe} \partial_x W_2 = \frac{\bp_\rho}{\brho \, \bthe} \big( \rho_t + (\rho u)_x \big) = 0,
\]
which is equal to zero because of the continuity equation in system \eqref{NSFK}. Consequently, we confirm that the first entry of the right hand side of \eqref{Wsystem} must be equal to zero as well. 
\end{remark}

From Proposition \ref{propgoodW} we know that the change of variables is well defined. But we need to examine the regularity of the new perturbation variables $W$. For that purpose we require some auxiliary classical results.

\begin{lemma}
\label{lemaux1}
$\,$
\begin{itemize}
\item[(a)] For any $s > n/2$, $n \in \N$, the space $H^s(\R^n)$ is a Banach algebra. Moreover, there exists a constant $C_s > 0$ such that
\[
\| uv \|_s \leq C_s \|u\|_s \|v \|_s,
\] 
for all $u,v \in H^s(\R^n)$.
\item[(b)] Let $u, v \in H^s(\R^n) \cap L^\infty(\R^n)$. Then
\[
\| uv \|_s \leq C \big( \| u \|_s \| v \|_\infty + \|u \|_\infty \|v \|_s \big),
\]
for some uniform constant $C > 0$.
\item[(c)] Let $s \geq 0$ and $k \geq 0$ be such that $s+k \geq [\frac{n}{2}] +1$. Assume that $u \in H^s(\R^n)$, $v \in H^k(\R^n)$. Then for $\ell = \min \{ s, k, s+k - [\frac{n}{2}] -1 \}$ we have $uv \in H^\ell(\R^n)$ and there exists a uniform $C > 0$ such that
\[
\| uv \|_\ell \leq C \| u \|_s \| v \|_k.
\]
In particular, if $s \geq [\frac{n}{2}] +1$, $0 \leq \ell \leq s$ and $u \in H^s(\R^n)$, $v \in H^\ell(\R^n)$, then
\[
\| uv \|_\ell \leq C \| u \|_s \| v \|_\ell.
\]
\end{itemize}
\end{lemma}
\begin{proof}
For the proof of (a) see Theorem 7.77, p. 359, in Iorio and Iorio \cite{IoIo01}. For (b) see Lemma 3.2 in \cite{HaLi96a}. The proof of (c) is a corollary of the interpolation inequalities obtained by Nirenberg \cite{Nir59} (see also Corollary 2.2 and Lemmata 2.1 and 2.3 in \cite{KaTh83}). 
\end{proof}

We also need some estimates on composite functions.

\begin{lemma}
\label{lemaux2}
Let $s \geq 1$ and suppose that $Y = (Y_1, \ldots, Y_m)$, $m \in \N$, $Y_i \in H^s(\R^n) \cap L^\infty(\R^n)$, for all $1 \leq i \leq m$. Let $\Lambda = \Lambda(Y)$, $\Lambda : \R^m \to \R^m$, be a $C^\infty$ function. Then for each $1 \leq j \leq s$ there hold $\partial_x \Lambda(Y) \in H^{j-1}(\R^n)$ and
\[
\| \partial_x \Lambda (Y) \|_{j-1} \leq C M \big( 1+\| Y \|_\infty \big)^{j-1} \| \partial_x Y \|_{j-1},
\]
where $C > 0$ is a uniform constant and
\[
M = \sum_{k=1}^j \sup_{\substack{V \in \R^m\\ |V|\leq \| Y \|_\infty}} | D^k_Y \Lambda(Y)| > 0.
\]
\end{lemma}
\begin{proof}
See Vol$'$pert and Hudjaev \cite{VoH72} (see also Lemma 2.4 in \cite{KaTh83}).
\end{proof}

Now we prove the following result.

\begin{lemma}
\label{lemregW}
Assume \hyperref[H1]{\rm{(H$_1$)}} - \hyperref[H4]{\rm{(H$_4$)}}. Suppose that $\rho_0 - \brho \in H^{s+1}(\R)$, $u_0 - \bu, \theta_0 - \bthe \in H^s(\R)$, with $s \geq 3$, satisfy the hypotheses of the local existence theorem \ref{themlocale}. If
\[
U = (\rho,u,\theta) \in X_s\big((0,T); \tfrac{1}{2}m_1,2M_1,\tfrac{1}{2}m_2,2M_2\big),
\]
denotes the local solution to the Cauchy problem \eqref{NSFK} with initial data $U(0) = (\rho_0, u_0, \theta_0)$ for some $0 < T \leq T_1$, and where the constants $m_j, M_j$, $j = 1,2$, are such that $U(\cdot, t) \in \cU$ \emph{a.e.} in $x\in \R$ and for all $t \in [0,T]$, then the new perturbation variables,
\[
W= D_Uf^0(\bU)^{-1} \big( F^0(U,U_x) - F^0(\bU,0)\big),
\]
satisfy:
\begin{itemize}
\item[(a)] $W \in X_s\big((0,T); \tfrac{1}{2}\bm_1,2\bM_1,\tfrac{1}{2}\bm_2,2\bM_2\big)$ for the same $T > 0$ and some constants $\bm_j$, $\bM_j$, $j = 1,2$, $\bar{m}_{1},\bar{M}_{1}>0$, depending on $a_0$, $\bU$, $m_j$, $M_j$.
\item[(b)] There exists a uniform constant $C_0 > 0$, depending only on $a_0, \bU, m_j, M_j$ (and hence independent of $T$) such that
\begin{equation}
\label{equivestUW}
C_0^{-1} \vertiii{U(t) - \bU}_s \leq \vertiii{W(t)}_s \leq C_0 \vertiii{U(t) - \bU}_s,
\end{equation}
for all $t \in [0,T]$. If, in addition, $U(0) -\bU \in L^1(\R)^3$ then the initial condition for $W$ satisfies the estimate
\begin{equation}
\label{L1estUW}
\| W(0) \|_{L^1} \leq C \| U(0) -\bU \|_{L^1} + C \| \rho_0 - \brho \|_1^2,
\end{equation}
for some uniform $C > 0$.
\item[(c)] $W$ solves the nonlinear system \eqref{Wsystem}.
\end{itemize}
\end{lemma}
\begin{proof}
Since the local solution is such that $U(t) \in \cU$ for all $t \in [0,T]$, then it is clear that Proposition \ref{propgoodW} holds, the matrix field $D_U F^0(U,U_x)$ is invertible and, consequently, the new perturbation variables $W$ solve system \eqref{Wsystem} as we have proved in Section \ref{secnewpert}. This shows (c) and we only need to prove (a) and (b).

In the sequel, we suppress the dependence on the time variable for the sake of brevity. Let us write
\[
W = \tiW + D_U f^0(\bU)^{-1} \Gamma^0(U,U_x), \qquad \tiW := D_U f^0(\bU)^{-1} (f^0(U) - f^0(\bU)).
\]
From the expression for $D_U f^0(\bU)^{-1}$ (see \eqref{inverseDf0}) and \eqref{decF0} it is easily verified that $W_1 = \tiW_1 = \rho - \brho$, yielding
\begin{equation}
\label{estW1l}
\| W_1 \|_{\ell+1} = \| \tiW_1 \|_{\ell+1} = \| \rho - \brho \|_{\ell +1},
\end{equation}
for all $0 \leq \ell \leq s$ and all $t\in [0,T]$. In particular, we have $W_1 \in H^{s+1}(\R)$. Also from inspection we have
\[
W_2 = \tiW_2 = - \frac{\bu}{\brho}(\rho - \brho) + \frac{1}{\brho} (\rho u - \brho \,\bu) = \frac{1}{\brho}(\rho - \brho)(u-\bu) + u - \bu.
\]
Therefore, from Lemma \ref{lemaux1} (a) we obtain
\[
\| W_2 \|_s \leq C \big(\| (\rho - \brho)(u - \bu) \|_s+ \| u - \bu \|_s \big) \leq C_s \big( \| \rho - \brho \|_s \| u - \bu \|_s + \| u - \bu \|_s \big),
\]
for some uniform $C > 0$, showing that $W_2 \in H^s(\R)$ for all $t \in [0,T]$. Moreover, since we already know that $\| U - \bU\|_\infty$ is uniformly bounded (see Remark \ref{remLinftybound}), from Lemma \ref{lemaux1} (b) we deduce that for each $0 \leq \ell \leq s$ there holds
\[
\begin{aligned}
\| W_2 \|_\ell^2 &\leq C \big(\| (\rho - \brho)(u - \bu) \|_\ell + \| u - \bu \|_\ell \big)^2 \\
 & \leq C \big( \| \rho - \brho \|_\infty \| u - \bu \|_\ell + \| u - \bu \|_\infty \| \rho - \brho \|_\ell +  \| u - \bu \|_\ell \big)^2\\ 
 & \leq C \big( \| \rho - \brho \|_\ell +  \| u - \bu \|_\ell \big)^{2} \\
&\leq C \big( \| \rho - \brho \|_\ell^2 + \| u - \bu \|_\ell^2 \big),
\end{aligned}
\]
for some uniform $C > 0$ depending only on $a_0$ and $\bU$. This yields, in turn,
\begin{equation}
\label{estW2l}
\| W_2 \|_\ell^2 \leq C \big( \| \rho - \brho \|_{\ell}^2 + \| u - \bu \|_\ell^2 \big), \qquad 0 \leq \ell \leq s.
\end{equation}
Upon inspection, we notice that the third component of $\tiW$ is of the form
\[
\tiW_3 = \bc_1 (\rho - \brho) + \bc_2 (\rho u - \brho \, \bu) + \bc_3 \big( \rho ( e + \tfrac{1}{2}u^2 ) - \brho( \be + \tfrac{1}{2} \bu^2) \big),
\]
for some constants $\bc_j$ depending on $\bU$. From the estimates above we already know that $\rho - \brho \in H^{s+1}(\R)$ and $\rho u - \brho \, \bu \in H^s(\R)$ and we only need to examine the third term. By a Taylor expansion of the standard total energy, $\rho \big( e(\rho,\theta) + \tfrac{1}{2}u^{2} \big)\in C^{\infty}(\overline{\cU})$, we have
\[
\begin{aligned}
\rho \big( e + \tfrac{1}{2}u^{2} \big) -\brho \big( \be + \tfrac{1}{2}\bu^{2} \big) &=  \big( \be+ \tfrac{1}{2}\bu^{2} \big)(\rho -\brho) + \brho \: \bu (u- \bu) + \brho \, \be_{\theta}(\theta- \bthe) + \\ 
& \quad+ O ( |\rho - \brho|^2 + |u - \bu|^2+ |\theta - \bthe|^2).
\end{aligned}
\]
Whence, from the Banach algebra properties of Lemma \ref{lemaux1} we get
\[
\begin{aligned}
\| \rho \big( e + \tfrac{1}{2}u^{2} \big) -\brho \big( \be + \tfrac{1}{2}\bu^{2} \big) \|^2_\ell & \leq C \big( \| \rho - \brho \|_\ell^2 + \| u - \bu \|_\ell^2 + \| \theta - \bthe \|_\ell^2 +   \\ & \quad +  \| |\rho - \brho|^2 + |u - \bu|^2+ |\theta - \bthe|^2  \|_\ell^2 \big) \\ & \leq C \big( \| \rho - \brho \|_\ell^2 + \| u - \bu \|_\ell^2 + \| \theta - \bthe \|_\ell^2 +  \\ & \quad +\| \rho - \brho \|_\infty^2 \| \rho - \brho \|_\ell^2 + \| u - \bu \|_\infty^2 \| u - \bu \|_\ell^2 + \| \theta - \bthe \|_\infty^2 \| \theta - \bthe \|_\ell^2   \big) \\ & \leq C \big( \| \rho - \brho \|_\ell^2 + \| u - \bu \|_\ell^2 + \| \theta - \bthe \|_\ell^2 \big) 
\end{aligned}
\]
for some $C > 0$ depending only on $a_0$ and $\bU$ and all $0 \leq \ell \leq s$. This shows that $\tiW_3 \in H^s(\R)$ and 
\begin{equation}
\label{estW3l}
\| \tiW_3 \|_\ell^2 \leq C \big( \| \rho - \brho \|_{\ell}^2 + \| u - \bu \|_\ell^2 + \| \theta - \bthe \|_\ell^2\big),
\end{equation}
uniformly for all $0 \leq \ell \leq s$ and all times $t \in [0,T]$.

Let us now examine the non-standard part of $W$, namely, the term 
\[
 D_U f^0(\bU)^{-1} \Gamma^0(U,U_x) = D_U f^0(\bU)^{-1} \begin{pmatrix} 0 \\ 0 \\ \rho(\kappa - \theta \kappa_\theta) \rho_x^2
\end{pmatrix}.
\]
Rewrite $\rho(\kappa - \theta \kappa_\theta) \rho_x^2= \big( \rho(\kappa- \theta \kappa_{\theta})- \brho(\bkp -\bthe \bkp_{\theta}) \big)\rho_{x}^{2}+ \brho(\bkp -\bthe \bkp_{\theta})\rho_{x}^{2} $. Then by making a Taylor expansion of the function $\rho( \kappa- \theta \kappa_{\theta})\in C^{\infty}(\overline{\cD})$, as we did for the standard total energy, one can easily show that $\big( \rho(\kappa- \theta \kappa_{\theta})- \brho(\bkp -\bthe \bkp_{\theta}) \big) \in H^{\ell}(\R)$, for $0 \leq \ell \leq s$, and that there holds
\[
\|\rho(\kappa- \theta \kappa_{\theta})- \brho(\bkp -\bthe \bkp_{\theta})  \|_{\ell}^{2} \leq C \big( \| \rho -\brho \|_{\ell}^{2} + \| \theta- \bthe \|_{\ell}^{2} \big),
\]
for some uniform constant $C>0$. Furthermore, as $\kappa \in C^\infty(\overline{\cD})$, with $\tfrac{1}{2} m_1 \leq \rho \leq 2 M_1$, $\tfrac{1}{2}m_2 \leq \theta \leq 2M_2$ a.e. in $x \in \R$, $t \in [0,T]$ and recalling Remark \ref{remLinftybound}, we obtain $\| \kappa - \theta \kappa_\theta\|_\infty, \| \rho \|_\infty, \| \rho_x \|_\infty \leq C$ for some $C > 0$ depending only on $a_0$, $\bU$, $m_j$ and $M_j$. Hence, by Lemma \ref{lemaux1} (b), we obtain
\[
\| \big( \rho(\kappa- \theta \kappa_{\theta})- \brho(\bkp -\bthe \bkp_{\theta}) \big)\rho_{x} \|_{\ell}^{2} \leq C \big( \| \rho -\brho \|_{\ell+1}^{2} + \| \theta- \bthe \|_{\ell}^{2} \big)
\]
Combine last estimate with Lemma \ref{lemaux1} (b) to obtain
\[
\| \big( \rho(\kappa- \theta \kappa_{\theta})- \brho(\bkp -\bthe \bkp_{\theta}) \big)\rho_{x}^{2} \|_{\ell}^{2} \leq C \big( \| \rho -\brho \|_{\ell+1}^{2} + \| \theta- \bthe \|_{\ell}^{2} \big).
\]
In addition, one can easily see, again by Lemma \ref{lemaux1}, that there holds
\[
\| \brho(\bkp -\bthe \bkp_{\theta}) \big)\rho_{x}^{2} \|_{\ell}^{2} \leq C \| \rho - \brho \|_{\ell+1}^{2},
\]
for $0 \leq \ell \leq s$, and some uniform constant $C>0$.Therefore we have
\begin{equation}
\label{estnsl}
\| \rho (\kappa - \theta \kappa_\theta) \rho_x^2 \|_\ell^2 \leq C \big( \| \rho - \brho \|_{\ell+1}^2 + \| \theta - \bthe \|_{\ell}^2  \big),
\end{equation}
with uniform $C > 0$. This shows that $\rho (\kappa - \theta \kappa_\theta) \rho_x^2 \in H^s(\R)$ and we conclude that $W_3 \in H^s(\R)$.  Moreover, gathering estimates \eqref{estW1l}, \eqref{estW2l}, \eqref{estW3l} and \eqref{estnsl} we deduce that there exists a uniform constant $\widetilde{C}_0 > 0$, depending only on $a_0$, $\bU$, $m_j$ and $M_j$, such that
\begin{equation}
\label{rightest}
\| W_1 \|_{\ell+1}^2 + \| W_2 \|_{\ell}^2 + \| W_3  \|_{\ell}^2 \leq \widetilde{C}_0 \big( \| \rho - \brho \|_{\ell+1}^2 + \| u - \bu \|_\ell^2 + \| \theta - \bthe \|_\ell^2 \big),
\end{equation}
for each $0 \leq \ell \leq s$ and all $t \in [0,T]$.

Now, let us define the mapping $Y = \Lambda(U)$, $\Lambda : \R^3 \to \R^3$, $\Lambda(U) := f^0(U) - f^0(\bU)$. Clearly, $\Lambda$ is smooth, $\Lambda(\bU) = 0$ and $D_U \Lambda(\bU) = D_U f^0(\bU)$ is non-singular. By the inverse function theorem there exist neighborhoods,
$B_1 := B_{\zeta_1}(\bU) = \{ U \in \R^3 \, : \, |U - \bU| < \zeta_1 \}$ and $B_2 := B_{\zeta_2}(0) = \{ Y \in \R^3 \, : \, |Y| < \zeta_2 \}$, for some radii $\zeta_i > 0$, such that $\Lambda(B_1) \subset B_2$ and $\Lambda : B_1 \to B_2$ is bijective. Moreover, the inverse mapping, $U = \Lambda^{-1}(Y)$, $\Lambda^{-1} : \Lambda(B_{1}) \to B_1$, is also smooth and satisfies $D_Y \Lambda^{-1} (\Lambda(\bU)) = D_Y \Lambda^{-1}(0) = D_U f^0(\bU)^{-1}$. Since $f^0$ and $\Lambda^{-1}$ are smooth, they are locally Lipschitz and there exists a uniform constant $C > 0$ (depending only on $a_0$ and $\bU$) such that
\[
\begin{aligned}
| f^0(U) - f^0(\bU) | &\leq C | U - \bU |, \\
| \Lambda^{-1}(Y) - \Lambda^{-1}(0) | = | U - \bU | &\leq C |Y| = C | f^0(U) - f^0(\bU)|,
\end{aligned}
\]
for all $U \in B_1$ (equivalently, for all $Y \in  \Lambda(B_1)$). Hence, we arrive at
\begin{equation}
\label{zeroth}
C^{-1} |U - \bU| \leq | f^0(U) - f^0(\bU)| \leq C |U - \bU|.
\end{equation}
As a first consequence of \eqref{zeroth} we have
\begin{align}
\| U - \bU \|_0^2 \leq C \| f^0(U) - f^0(\bU) \|_0^2 &= C \| D_U f^0(\bU) W - \Gamma^0(U,U_x) \|_0^2 \nonumber\\
&\leq C \| W \|_0^2 + C \| \rho(\kappa - \theta \kappa_\theta) \rho_x \|_\infty^2 \| \rho_x \|_0^2 \nonumber\\
&\leq C \big( \| W_1 \|_1^2 + \| W_2 \|_0^2 + \| W_3 \|_0^2\big). \label{firstleft}
\end{align}

Now, observe that we have already proved that $\tiW = D_U f^0(\bU)^{-1} (f^0(U) - f^0(\bU)) \in H^{s+1}(\R) \times H^s(\R) \times H^s(\R)$. Henceforth, $Y = \Lambda(U) \in H^{s+1}(\R) \times H^s(\R) \times H^s(\R)$ as well. Moreover, from \eqref{zeroth} and Remark \ref{remLinftybound} we also deduce that $Y = \Lambda(U) \in L^\infty(\R)^3$. Therefore, being $\Lambda^{-1}$ smooth, from Lemma \ref{lemaux2} we obtain
\[
\| \partial_x \Lambda^{-1}(Y) \|_{j-1} \leq C M (1 + \|Y\|_\infty)^{j-1} \| \partial_x Y \|_{j-1},
\]
for all $1 \leq j \leq s$, which is equivalent to
\begin{equation}
\label{secondleft}
\| \partial_x (U - \bU) \|_{j-1} \leq \widetilde{C} \| \partial_x (f^0(U) - f^0(\bU)) \|_{j-1} \leq C \| \partial_x \tiW \|_{j-1},
\end{equation}
for $1 \leq j \leq s$ and some uniform $C > 0$. Combining estimates \eqref{estnsl}, \eqref{firstleft} and \eqref{secondleft}, and recalling that $W_1 = \tiW_1 = \rho - \brho \in H^{s+1}(\R)$, we obtain the estimate
\[
\begin{aligned}
\| \rho - \brho \|_{\ell +1}^2 + \| u - \bu \|_{\ell}^2 + \| \theta - \bthe\|_{\ell}^2 & \leq C \big( \| \tiW_1 \|_{\ell+1}^2 + \| \tiW_2 \|_\ell^2 + \| \tiW_3 \|_\ell^2 \big)\\
&\leq  C \big( \| W_1 \|_{\ell+1}^2 + \| W_2 \|_\ell^2 + \| W_3 \|_\ell^2 + \| \rho(\kappa - \theta \kappa_\theta) \rho_x^2 \|_\ell^2 \big)\\
&\leq  C \big( \| W_1 \|_{\ell+1}^2 + \| W_2 \|_\ell^2 + \| W_3 \|_\ell^2 + \| \rho- \brho \|_{\ell+1}^2+\| \theta- \bthe \|_{\ell}^2  \big)\\
&\leq  C \big( \| W_1 \|_{\ell+1}^2 + \| W_2 \|_\ell^2 + \| W_3 \|_\ell^2 \big).
\end{aligned}
\]
We conclude that there exists a uniform $C_0 > 0$ such that
\begin{equation}
\label{equivl}
C_0^{-1} \vertiii{U(t) - \bU}_\ell \leq \vertiii{W(t)}_\ell \leq C_0 \vertiii{U(t) - \bU}_\ell,
\end{equation}
for all $0 \leq \ell \leq s$ and all $0 \leq t \leq T$. Taking $\ell = s$ we obtain estimate \eqref{equivestUW}. This implies, in particular, that $W \in X_s\big((0,T); \tfrac{1}{2}\bm_1,2\bM_1,\tfrac{1}{2}\bm_2,2\bM_2\big)$ for the same $T > 0$, with constants $\bm_j$, $\bM_j$, that depend on $C_0 >0$, as well as on,  $a_0$, $\bU$, $m_j$, $M_j$. This shows (a).

Finally, let us further assume that $U(0) - \bU \in L^1(\R)^3$. Since, 
\[
|W(0)| = \big| D_U f^0(\bU)^{-1} \big( f^0(U(0)) - f^0(\bU) + \Gamma^0(U(0), \partial_x U(0)) \big) \big|,
\]
and the non-standard part on the right hand side is proportional to $\rho_0 (\kappa(\rho_0, \theta_0) - \theta_0 \kappa_\theta( \rho_0, \theta_0))(\partial_x \rho_0)^2$, use \eqref{zeroth} to obtain
\[
\begin{aligned}
\| W(0) \|_{L^1} &\leq C \big( \| U(0) - \bU \|_{L^1} + \| \rho_0 (\kappa(\rho_{0},\theta_{0})-\theta_{0}\kappa_{\theta}(\rho_{0},\theta_{0})) (\partial_x \rho_0)^2 \|_{L^1} \big)\\
&\leq C \big( \| U(0) - \bU \|_{L^1} + \|  \rho_0 (\kappa(\rho_{0},\theta_{0})-\theta_{0}\kappa_{\theta}(\rho_{0},\theta_{0})) \|_\infty \| (\partial_x \rho_0)^{2}\|_{L^{1}} \big)\\
& \leq C \big( \| U(0) - \bU \|_{L^1} +  \| \partial_{x}\rho_{0} \|_{0}^{2} \big) \\
&\leq C \big( \| U(0) - \bU \|_{L^1} + \| \rho_0 - \brho \|_1^2 \big).
\end{aligned}
\]
This completes the proof of (b). The lemma is proved.
\end{proof}

\begin{remark}
\label{remequivUW}
Estimate \eqref{equivestUW} essentially means that the new perturbation variables $W$ are equivalent to $U - \bU$ in the appropriate function space. For later use, more precisely, for obtaining estimates at a nonlinear level, we will need to assume that initial perturbations have finite mass as well (i.e., in $L^1$). Thus, estimate \eqref{L1estUW} shows the same for the initial condition in the $W$ variables. Notice the extra term depending on the Sobolev norm of the density, which comes from the non-standard part of the conserved quantities.
\end{remark}

\section{Linear decay}
\label{seclindecay}

This section is devoted to obtain the decay of the solutions to the linearized NSFK system around a constant equilibrium state $\bU = \big( \brho, \bu, \bthe \big) \in \cU$. For that purpose, we examine its dissipative structure within the framework of Humpherys' work \cite{Hu05}. This structure will be studied through its equivalence to the genuine coupling condition and to the existence of an appropriate compensating matrix symbol. 

\subsection{Linearization and symmetrizability of the symbol}

Let us consider the linear part of equation \eqref{Wsystem}, namely,
\begin{equation}
\label{27}
A^{0}(\bU)W_{t} + A^{1}(\bU)W_{x} - B(\bU)W_{xx}- C(\bU)W_{xxx}= 0,
\end{equation}
which is a linear, constant coefficient system for $W$. The coefficient matrices are given by \eqref{barA0A1B} and \eqref{barC}. Following Humpherys \cite{Hu05}, we define the linear operator $\cL$ and write system \eqref{27} as
\begin{equation}
\label{29}
A^{0}(\bU)W_{t} =  \cL W := -D_{1}W_{x} - D_{2}W_{xx} - D_{3}W_{xxx}, 
\end{equation}
with $D_{1}= A^{1}(\bU)$, $D_{3}= -B(\bU)$, $D_{3}= -C(\bU)$. Then apply the Laplace-Fourier transform to obtain the eigenvalue problem
\[
\left( \lambda A^{0}(\bU) + i \xi D_{1} + (i\xi)^2 D_{2} + (i \xi)^3 D_{3} \right) \hW = 0,
\]
with $\lambda \in \C$ and $\xi \in \R$ denoting the frequency and the Fourier wave numbers, respectively. To simplify the notation, let us write $A^0 \equiv A^0(\bU)$, a constant matrix. Humpherys \cite{Hu05} then splits the symbol into odd and even terms. Define the symbols,
\[
A(\xi) := \sum_{j \, \text{odd}} (i\xi)^{j-1} D_j, \qquad B(\xi) := \sum_{j \, \text{even}} (-1)^{j/2} \xi^j D_j,
\]
so that the eigenvalue problem in Fourier space is reduced to
\begin{equation}
\label{30}
\left( \lambda A^{0} + i \xi A(\xi) + B(\xi) \right) \hW = 0,
\end{equation}
where
\begin{equation}
\label{31}
\begin{aligned}
A(\xi) = D_{1}- \xi^{2}D_{3} &= A^{1}(\bU) + \xi^{2} C(\bU) \\ &=
\begin{pmatrix}
\bp_{\rho} \bu / \bthe \brho & \bp_{\rho}/ \bthe & 0 \\ \bp_{\rho}/ \bthe+ \xi^{2} \bk \brho / \bthe & \bu \:\brho/ \bthe & \bp_{\theta} / \bthe \\ 0 & \bp_{\theta}/ \bthe & \brho \: \bu \, \be_{\theta}  / \bthe^{2}
\end{pmatrix} \\ & = \frac{1}{\bthe}\begin{pmatrix}
\bp_{\rho} \bu/ \brho  & \bp_{\rho}  & 0 \\ \beta(\xi) & \brho \, \bu & \bp_{\theta}  \\ 0 & \bp_{\theta}& \brho \:  \bu \, \be_{\theta}  / \bthe
\end{pmatrix},
\end{aligned}
\end{equation}
\[
B(\xi) = \xi^2 B(\bU) =  \frac{\xi^2}{\bthe}  \begin{pmatrix}
0 & 0 & 0 \\ 0 & \bmu & 0 \\ 0 & 0 & \balp/ \bthe
\end{pmatrix} =: \xi^2 B \geq 0,
\]
clearly $B$ is a constant matrix, and where we have defined
\begin{equation}
\label{defbeta}
\beta(\xi) := \bp_{\rho} + \xi^2 \bk \brho > 0, \qquad \xi \in \R.
\end{equation}
Observe that, although $A^{0}$ and $B$ are symmetric, $A(\xi)$ is not. Then the triplet $(A^{0},A(\xi),B(\xi))$, associated to system \eqref{27} in the Fourier space after splitting the symbol in even and odd terms, is not in symmetric form. However it can be symmetrized in the sense described next.

Let us remember that a generic linear system of the form \eqref{29} is said to be symmetrizable in the classical sense of Friedrichs \cite{Frd54} if, for any constant state $\bU$, there exists a symmetric, positive definite matrix $S=S(\bU) > 0$ such that $SA^{0}$, $SD_{j}$, $j=1,2,3$, are symmetric with $SA^{0}$ positive definite. In that context, Humpherys \cite{Hu05} (see also \cite{PlV22}) realized that, even though the linearized isothermal Korteweg system is not Friedrichs symmetrizable, it can be symmetrized in the following sense.
\begin{definition}[symbol symmetrizability \cite{Hu05}]
\label{defsymH}
The operator $\cL$ is \emph{symbol symme\-trizable}, or \emph{symmetrizable in the sense of Humpherys}, if there exists a smooth, symmetric matrix-valued function, $S=S(\xi)>0$, positive-definitive, such that $S(\xi)A^{0}$, $S(\xi)A(\xi)$ and $S(\xi)B(\xi)$ are symmetric, with  $S(\xi)A^{0}>0$ and $S(\xi)B(\xi)\geq 0$ (positive semi-definite) for all $\xi \in \R$.
\end{definition}  
\begin{lemma}
\label{Sym2Full}
The linear (non-isothermal) Korteweg system \eqref{27} is symbol symmetrizable but not symmetrizable in the sense of Friedrichs. One symbol symmetrizer is of the form
\begin{equation}
\label{32}
S(\xi) = \begin{pmatrix}
\beta(\xi)/\bp_{\rho} & 0 & 0 \\ 0 & 1 & 0 \\ 0 & 0 & 1 
\end{pmatrix} \in C^{\infty}\big( \R; \R^{3\times 3}  \big).
\end{equation} 
\end{lemma}
\begin{proof}
Clearly, $S(\xi)$ is smooth, symmetric and positive definite. From direct computations one easily obtains
\[
S(\xi)A^{0} =  \frac{1}{\bthe} \begin{pmatrix}
\beta(\xi)/ \brho  & 0 & 0 \\ 0 & \brho  & 0 \\ 0 & 0 &  \be_{\theta} \brho / \bthe
\end{pmatrix}, \quad 
S(\xi)A(\xi)= \frac{1}{\bthe}\begin{pmatrix}
\beta(\xi)\bu/ \brho  & \beta(\xi)  & 0 \\ \beta(\xi) & \brho \, \bu & \bp_{\theta}  \\ 0 & \bp_{\theta}& \brho\: \bu  \, \be_{\theta}  / \bthe
\end{pmatrix},
\]
and $S(\xi)B(\xi)= B(\xi)$, which are symmetric matrices with $S(\xi)A^{0}>0$ and $S(\xi)B(\xi) \geq 0$. This shows that the operator $\cL$ is symbol symmetrizable. To verify that the system is not Friedrichs symmetrizable, suppose there exists a symmetrizer of the form 
\[
S = \begin{pmatrix}
s_{11} & s_{12} & s_{13} \\ s_{12} & s_{22} & s_{23} \\ s_{13} & s_{23} & s_{33}
\end{pmatrix}.
\]
Then the condition on $SD_{2}(\bU)$ and $SD_{3}(\bU)$ to be simultaneously symmetric matrices implies that $S$ cannot be positive definite, as the reader may easily verify. The lemma is proved.
\end{proof}
\begin{remark}
Up to our knowledge, the non-isothermal Korteweg system under consideration in this paper is the second example (only after the isothermal Korteweg model \cite{Hu05,PlV22}) of a physically relevant system which is not Friedrichs symmetrizable but symmetrizable in the sense of Humpherys. By virtue of the importance of symmetrization of differential operators in relation to the existence of convex entropies, it is difficult to underestimate the relevance of Definition \ref{defsymH}.
\end{remark}

\subsection{Genuine coupling and the compensating matrix symbol}

Next, we show that the linear system \eqref{27} satisfies the genuine coupling condition and exhibit an appropriate compensating matrix symbol for it. First, we need to make precise some definitions.

\begin{definition}[compensating matrix symbol \cite{Hu05}]
\label{ComMat}
Let $A^{0}$, $A$, $B \in C^{\infty} \left( \R; \R^{3 \times 3} \right)$ be smooth, real matrix-valued functions of the variable $\xi \in \R$. Assume that $A^{0},$ $A$, $B$ are symmetric for all $\xi \in \R$, $A^{0} >0$ is positive definite and $B \geq 0$ is positive semi-definite. If a smooth, real matrix valued function,  $K\in C^{\infty} \left( \R; \R^{3 \times 3} \right)$, satisfies
\begin{itemize}
\item[(a)] $K(\xi)A^{0}(\xi)$ is skew-symmetric for all $\xi\in \R$; and,
\item[(b)] $\left[K(\xi)A(\xi)\right]^{s}+ B(\xi) \geq \gamma(\xi) I > 0$ for all $\xi \in \R$, $\xi \neq 0$, and some $\gamma = \gamma(\xi) > 0$,
\end{itemize}
then $K$ is said to be a \emph{compensating matrix symbol} for the triplet $(A^{0}, A, B)$. Here $[M]^{s} := \frac{1}{2}(M+M^\top)$ denotes the symmetric part of any real matrix $M$. 
\end{definition}

\begin{definition}[strict dissipativity and genuine coupling \cite{Hu05,ShKa85}]
\label{defstrictdiss}
Consider the linear system \eqref{29} and its associated eigenvalue problem \eqref{30} in the Fourier space. Then we say that:
\begin{itemize}
\item[(i)] The operator $\mathcal{L}$ is called \emph{strictly dissipative} if for each $\xi \neq 0$ then all solutions to the spectral problem \eqref{30} satisfy $\Re  \lambda (\xi) < 0$.
\item[(ii)] The triplet $(A^0, A, B)(\xi)$ is said to be genuinely coupled provided that, for all $\xi \neq 0$, every vector $V \in \ker B(\xi)$, with $V \neq 0$, satisfies the condition $\big( \varrho A^{0}(\xi)  +    A(\xi) \big) V \neq 0$ for any $\varrho \in \R$. In that case we say that the operator $\cL$ satisfies the \emph{genuine coupling condition}. 
\end{itemize}
\end{definition}

\begin{remark}
\label{remsdgc}
Strict dissipativity is tantamount to the stability of the essential spectrum of the constant coefficient differential operator $\cL$ when computed, for instance, with respect to the space $L^2(\R)$ of finite energy perturbations. Genuine coupling underlies a simple algebraic property of the operator $\cL$, inasmuch as no eigenvector of the hyperbolic/dispersive term, $A(\xi)$, lies in the kernel of the dissipative symbol, $B(\xi)$.
\end{remark}

Humpherys then proved the following equivalence theorem (see Theorems 3.3 and 6.3 in \cite{Hu05}).
\begin{theorem}[equivalence theorem \cite{Hu05}]
\label{HuThSym}
Suppose that a symbol symmetrizer, $S= S(\xi)$, $S \in C^\infty(\R; \R^{3 \times 3})$, exists for the operator $\mathcal{L}$ in the sense of Definition \ref{defsymH}, and that $S(\xi)A(\xi)$ is of constant multiplicity in $\xi$, that is, all its eigenvalues are semi-simple and with constant multiplicity for all $\xi \in \R$. Then the following conditions are equivalent:
\begin{itemize}
\item[(a)] $\mathcal{L}$ is strictly dissipative.
\item[(b)] $\mathcal{L}$ is genuinely coupled.
\item[(c)] There exists a compensating matrix function for the triplet $(SA^{0},SA, SB)$.
\end{itemize}
\end{theorem}

\begin{remark}
\label{remeqth}
Theorem \ref{HuThSym} constitutes the extension to higher order one-dimensio\-nal systems, of the classical equivalence theorem for second order systems, generalizing in this fashion the work by Ellis and Pinsky \cite{EllPin75a,EllPin75b} and by Kawashima and Shizuta \cite{ShKa85,KaSh88a}. We remind the reader that for the ``inviscid"  system (that is, without dissipation, dispersion or relaxation effects) to be hyperbolic, the matrix $A^1 = A^1(\bU)$ must be diagonalizable over $\R$, with all eigenvalues real and semi-simple.  In the present context, Humpherys works with real, smoothly varying matrix symbols, instead of working with constant matrices. Then Humpherys assumes the symmetric matrix symbol $S(\xi) A(\xi)$ (which contains higher order terms too) to be diagonalizable over $\R$ with constant multiplicity for all $\xi \in \R$. This condition is imposed as a key ingredient in order to perform the aforementioned extension.
\end{remark}

\begin{remark}
\label{remtypediss}
A few comments about the proof of Theorem \ref{HuThSym} are in order. Humpherys actually provides an explicit formula for the compensating matrix symbol in the general case, which is based on the Drazin's inverse of the commutator operator. From the expression of the latter, it is then proved that property (c) implies property (a) by showing the estimate
\[
\Re \lambda(\xi) \leq - \, \frac{\xi^2 \gamma(\xi)^2}{4 |K(\xi)| \big( |\xi| \gamma(\xi) + |K(\xi)| |B(\xi)|\big)} < 0,
\]
where $\gamma(\xi) > 0$ is the bound appearing in Definition \ref{ComMat} (b) and $K = K(\xi)$ denotes the compensating matrix (for details, see \cite{Hu05}). This estimate establishes, in turn, the decay rate of the solutions to the linear system. It is to be observed, however, that the compensating matrix symbol \emph{is not necessarily unique}, leading to different rates of decay. For instance, in a previous work \cite{PlV22}, we constructed a compensating matrix function for the isothermal Korteweg model that yields a better (actually, optimal) decay rate. Systems underlying different rates of decay at a linear level have been classified by Ueda \emph{et al.} \cite{UDK12, UDK18} as follows: the linear system is called strictly dissipative of type $(p,q)$, with $p, q \in \Z$, $p, q \geq 0$, provided that the solutions of the spectral problem \eqref{30} satisfy
\[
\Re \lambda(\xi) \leq - \, \frac{C |\xi|^{2p}}{(1 + |\xi|^2)^{q}}, \qquad \forall \xi \neq 0, 
\]
for some uniform constant $C > 0$. The system is of standard type when $p = q$ \cite{UDK12}, and of regularity-loss type when $p < q$ \cite{UDK18}. Notice that strict dissipativity of type $(1,0)$ is precisely that of the heat kernel. Hence, the third case when $p > q$ is called dissipativity of regularity-gain type \cite{KSX22}. It has been shown, for instance, that both the isothermal \cite{PlV22} and the non-isothermal \cite{KSX22} Korteweg fluids with viscosity belong to the latter class. Our main observation here is that, even though Humpherys' construction of the compensating matrix symbol is fundamental from a theoretical point of view, in practice one may profit from the degree of freedom in selecting the symbol in order to obtain a better linear decay rate, as we show in the sequel.
\end{remark}

Let us now consider the original linear system in Fourier space,
\begin{equation}
\label{63}
A^{0}\hW_{t} + \big(i \xi A(\xi) + \xi^{2} B\big) \hW= 0, 
\end{equation}
where $A^{0} \equiv A^0(\bU)$ and $B \equiv B(\bU)$ and $A(\xi)$ are given in \eqref{barA0A1B} and in \eqref{31}, respectively. Make the change of variables,
\begin{equation}\label{hU}
\hV := S(\xi)^{1/2}\big( A^{0} \big)^{1/2}\hW,
\end{equation}
where $S(\xi) > 0$ is the symmetrizer from Lemma \ref{Sym2Full}. Hence, system \eqref{63} can be recast as
\begin{equation}
\label{64}
\hV_{t} + \big( i \xi \tiA(\xi) + \xi^{2} \tiB(\xi) \big) \hV = 0, 
\end{equation}
where
\begin{equation}
\label{65}
\begin{aligned}
\tiA(\xi) &:= S(\xi)^{1/2}\big(A^{0}\big)^{-1/2}A(\xi) \big( A^{0} \big)^{-1/2} S(\xi)^{-1/2},\\
\tiB(\xi) &:= S(\xi)^{1/2}\big(A^{0}\big)^{-1/2}B \big(A^{0}\big)^{-1/2} S(\xi)^{-1/2}.
\end{aligned}
\end{equation}
Direct computations yield
\begin{equation}
\label{exprtiA}
\tiA(\xi) = \begin{pmatrix}
\bu & \beta^{1/2} & 0 \\ \beta^{1/2} & \bu & \bc \\ 0 & \bc & \bu
\end{pmatrix},
\end{equation}
where
\[
\bc := \frac{\bp_\theta \bthe^{1/2}}{\be_\theta^{1/2} \brho} > 0,
\]
is a constant, and
\[
\tiB(\xi) \equiv \tiB := \begin{pmatrix}
0 & 0 & 0 \\ 0 & \bmu/ \brho & 0 \\ 0 & 0 & \balp/\be_{\theta}\brho
\end{pmatrix},
\]
is a constant matrix. Notice that $\tiA(\xi)$ and $\tiB$ are symmetric.

\begin{lemma}
\label{lemgencoup}
The triplet $(I, \tiA, \tiB)$ is genuinely coupled. Moreover, the matrix symbol $\tiA(\xi)$ in \eqref{exprtiA} is of constant multiplicity in $\xi \in \R$.
\end{lemma}
\begin{proof}
Clearly, $\ker \tiB = \{ (a,0,0)^\top \in \R^3 \, : \, a \in \R\}$. Hence, for any $\tiV \in \ker \tiB$, $\tiV \neq 0$, there holds
\[
(\varrho I + \tiA(\xi)) \tiV = \begin{pmatrix} (\varrho + \bu)a \\ \beta(\xi)^{1/2}a \\ 0
\end{pmatrix} \neq 0,
\] 
for all $\varrho \in \R$, because $a \neq 0$. This shows genuine coupling of the triplet $(I, \tiA, \tiB)$. Now, let us compute the eigenvalues of $\tiA(\xi)$. From its characteristic polynomial, namely
\[
0 = \det \big(\tiA(\xi) - \lambda I \big) = (\bu - \lambda) \big( (\bu - \lambda)^2 - \bc^2 - \beta(\xi)\big),
\]
we obtain
\[
\lambda_1(\xi) := \bu - (\bc^2 + \beta(\xi))^{1/2} < \lambda_2(\xi) := \bu < \lambda_3(\xi) := \bu + (\bc^2 + \beta(\xi))^{1/2}.
\]
Thus, all eigenvalues are real and simple for any value of $\xi \in \R$. The lemma is proved.
\end{proof}

\begin{corollary}
The linear Korteweg operator $\cL$ defined in \eqref{29} satisfies the genuine coupling condition. 
\end{corollary}
\begin{proof}
Under the transformation defined in \eqref{hU}, namely, $\hV \mapsto S(\xi)^{1/2} (A^0)^{1/2} \hW$, it is easy to prove that the triplet $(A^0, A, B)$ is genuinely coupled if and only if the triplet $(I, \tiA, \tiB)$ is genuinely coupled as well (details are left to the reader). The conclusion then follows from Lemma \ref{lemgencoup}.
\end{proof}

By virtue of the equivalence Theorem \ref{HuThSym} and Lemma \ref{lemgencoup}, we deduce the existence of a compensating matrix symbol for the triplet $(I, \tiA, \tiB )$ associated to the linear symmetric system \eqref{64}. However, as discussed in Remark \ref{remtypediss}, in applications it is more convenient to construct the symbol directly. In that sense, the following result provides more valuable information than the equivalence theorem.

\begin{lemma}
\label{lemourK}
There exists a smooth compensating matrix symbol, $\tiK = \tiK(\xi) \in C^\infty(\R; \R^{3 \times 3})$, for the triplet $(I, \tiA(\xi), \tiB)$ of system \eqref{64}. In other words, $\tiK$ is skew-symmetric and 
\begin{equation}
\label{compmatprop}
[\tiK(\xi) \tiA(\xi)]^s + \tiB \geq \overline{\gamma} I > 0, 
\end{equation}
for some uniform constant $\overline{\gamma} > 0$ independent of $\xi \in \R$. Moreover, there hold
\begin{equation}
\label{tiKbded}
|\xi \tiK(\xi)|, | \tiK(\xi) | \leq C,
\end{equation}
for all $\xi \in \R$ and some uniform constant $C>0$. 
\end{lemma} 
\begin{remark}
\label{remgoodK}
The new features of this compensating matrix symbol, which cannot be deduced from the equivalence theorem, are that, (i) the constant $\overline{\gamma} > 0$ in \eqref{compmatprop} can be chosen uniformly in $\xi \in \R$; and, (ii) that both $|\tiK(\xi)|$ and $|\xi \tiK(\xi)|$ are uniformly bounded above. This last property has important consequences, such as the decay rate of the solutions to the linearized problem, yielding a dissipativity of regularity-gain type (see Remark \ref{remtypediss}).
\end{remark}
\begin{proof}[Proof of Lemma \ref{lemourK}]
We proceed by inspection and propose $\tiK(\xi)$ to be of the form
\[
\tiK(\xi) := \delta \begin{pmatrix}
0 & a & 0 \\ -a & 0 & b \\ 0 & -b & 0
\end{pmatrix},
\]
where $\delta$, $a$ and $b$ are to be determined later. Clearly $\tiK$ is skew-symmetric. Let us compute $\tiK(\xi) \tiA(\xi)$. The result is,
\[
\begin{aligned}
\tiK(\xi) \tiA(\xi) &= \delta \begin{pmatrix}
0 & a & 0 \\ -a & 0 & b \\ 0 & -b & 0
\end{pmatrix} \begin{pmatrix}
\bu & \beta(\xi)^{1/2} & 0 \\ \beta(\xi)^{1/2} & \bu & \bc \\ 0 & \bc &\bu
\end{pmatrix} \\ &= \delta \begin{pmatrix}
a \beta(\xi)^{1/2} & a\bu & a\bc \\ -a\bu & b\bc -a\beta(\xi)^{1/2} & b\bu \\ -b\beta(\xi)^{1/2} & -b \bu & -b\bc
\end{pmatrix}.
\end{aligned}
\]
We choose $a \equiv1$ and $b=\bc\beta(\xi)^{-1/2}$, so that 
\[
\tiK(\xi) \tiA(\xi) = \delta \begin{pmatrix}
\beta(\xi)^{1/2} & \bu & \bc \\ -\bu & \bc^2\beta(\xi)^{-1/2} - \beta(\xi)^{1/2} & b\bu \\ -\bc & -b\bu & -\bc^2 \beta(\xi)^{-1/2}
\end{pmatrix}.
\]
The symmetric part of last matrix is
\[
\big[\tiK(\xi) \tiA(\xi) \big]^{s} = \delta \begin{pmatrix}
\beta(\xi)^{1/2} & 0 & 0 \\ 0 & \bc^2\beta(\xi)^{-1/2} - \beta(\xi)^{1/2} & 0 \\ 0 & 0 & -\bc^2 \beta(\xi)^{-1/2}
\end{pmatrix}.
\]
We then choose $\delta = \epsilon \beta(\xi)^{-1/2}$ for some $0 < \epsilon \ll 1$ small. This yields
\[
\begin{aligned}
\big[\tiK(\xi) \tiA(\xi) \big]^{s}+ \tiB &= \begin{pmatrix}
\epsilon & 0 & 0 \\ 0 & \bmu/ \brho + \epsilon (\bc^2 / \beta(\xi) \, -1) & 0 \\ 0 & 0 & \balp/\be_{\theta}\brho - \epsilon \bc^2/\beta(\xi)
\end{pmatrix}.
\end{aligned}
\]
For this matrix to be positive definite we need $\epsilon > 0$ to satisfy
\begin{align}
\epsilon &> \overline{\gamma} > 0, \label{lai}\\
\frac{\bmu}{\brho} + \epsilon \Big( \frac{\bc^2}{\beta(\xi)} -1 \Big) &> \overline{\gamma} > 0, \label{laii}\\
\frac{\balp}{\be_\theta \brho} - \epsilon \frac{\bc^2}{\beta(\xi)} &> \overline{\gamma} > 0, \label{laiii}
\end{align}
for all $\xi \in \R$ and for some uniform $\overline{\gamma} > 0$. From \eqref{defbeta} we have
\[
\min_{\xi \in \R} \beta(\xi) = \min_{\xi \in \R} \big( \bp_\rho + \xi^2 \bk \brho \big) = \bp_\rho > 0.
\]
Therefore, it is clear that choosing any $\epsilon > 0$ satisfying
\begin{equation}
\label{choiceep}
0 < \overline{\gamma} := \frac{1}{4} \min \Big\{ \frac{\balp}{\be_\theta \brho}, \, \frac{\bmu}{\brho}, \, \frac{\balp \,\bp_\rho}{\be_\theta \brho \,\bc^2} \Big\} < \epsilon < \frac{1}{2}\min \Big\{ \frac{\bmu}{\brho}, \, \frac{\balp \,\bp_\rho}{\be_\theta \brho \,\bc^2}\Big\},
\end{equation}
we clearly obtain \eqref{lai}, together with
\[
\frac{\bmu}{\brho} + \epsilon \Big( \frac{\bc^2}{\beta(\xi)} -1 \Big) \geq \frac{\bmu}{\brho} - \epsilon > \frac{\bmu}{2\brho} > \overline{\gamma} > 0,
\]
and,
\[
\frac{\balp}{\be_\theta \brho} - \epsilon \frac{\bc^2}{\beta(\xi)} \geq \frac{\balp}{\be_\theta \brho} - \epsilon \frac{\bc^2}{\bp_\rho} > \frac{\balp}{2\be_\theta \brho} > \overline{\gamma} > 0,
\]
yielding \eqref{laii} and \eqref{laiii}, respectively. In this fashion, we conclude that the selected symbol,
\begin{equation}
\label{tildeK}
\tiK(\xi) := \frac{\epsilon}{\beta(\xi)^{1/2}} \begin{pmatrix} 0 & 1 & 0 \\ -1 & 0 & \bc \beta(\xi)^{-1/2} \\ 0 & - \bc \beta(\xi)^{-1/2} & 0
\end{pmatrix},
\end{equation}
where $\epsilon > 0$ satisfies \eqref{choiceep}, is a compensating matrix function for the triplet $(I, \tiA(\xi), \tiB)$ of system \eqref{64}, because it is smooth, $\tiK \in C^{\infty}(\R;\R^{3 \times 3})$, it is clearly skew-symmetric, and satisfies
\[
\big[\tiK(\xi) \tiA(\xi) \big]^{s}+ \tiB \geq \overline{\gamma} I > 0,
\]
for all $\xi \in \R$ for some uniform $\overline{\gamma} > 0$ defined in \eqref{choiceep}. Finally, since $\beta(\xi) = \bp_\rho + \xi^2 \bk \, \brho = O(\xi^2)$, it is then clear that there exists a positive constant $C > 0$ such that $|\beta(\xi)^{-1/2}|$, $|\xi \beta(\xi)^{-1/2}| \leq C$ and $|\beta(\xi)^{-1}|$, $|\xi \beta(\xi)^{-1}| \leq C$ for all $\xi \in \R$, yielding \eqref{tiKbded} as claimed. The lemma is proved.
\end{proof}

\subsection{Energy estimate and linear decay of the semigroup}
\label{secee}

To conclude this section, we now establish the basic energy estimate in the Fourier space for the full linear Korteweg system \eqref{27}. This energy estimate implies the decay rate of the associated semigroup. The precise form of the compensating matrix symbol from Lemma \ref{lemourK} plays a key role.

\begin{lemma}[basic energy estimate]
\label{lembee}
The solutions $\hV = \hV(\xi,t)$ to the linear system \eqref{64} satisfy the pointwise estimate
\begin{equation}
\label{bestV}
|\hV(\xi,t)| \leq C \exp (- c_0 \xi^2 t) |\hV(\xi,0)|,
\end{equation}
for all $\xi \in \R$, $t \geq 0$ and some uniform constants $C,c_0 > 0$.
\end{lemma}
\begin{proof}
Take the inner product in $\C^3$ of $\hV$ with equation \eqref{64} to obtain
\[
\langle \hV, \hV_t \rangle + \langle \hV, i\xi \tiA(\xi) \hV \rangle + \xi^2 \langle \hV, \tiB \hV \rangle = 0
\]
(recall that $\tiB(\xi) = \tiB$ is constant). Since $\tiA(\xi)$ and $\tiB$ are symmetric, taking the real part of last equation yields
\begin{equation}
\label{la8}
\tfrac{1}{2} \partial_t |\hV|^2 + \xi^2 \< \hV, \tiB \hV \> = 0.
\end{equation}
Let us multiply equation \eqref{64} by $- i \xi \tiK(\xi)$, where $\tiK(\xi)$ is the compensating matrix symbol constructed in Lemma \ref{lemourK}. The result is
\[
- i \xi \tiK(\xi) \hV_{t} + \xi^2 \tiK(\xi) \tiA(\xi) \hV - i \xi^3 \tiK(\xi) \tiB \hV = 0.
\]
If we take the $\C^3$-inner product of this equation with $\hV$ then we arrive at
\begin{equation}
\label{la9}
- \< \hV, i \xi \tiK(\xi) \hV_t \> + \xi^2 \< \hV, \tiK(\xi) \tiA(\xi) \hV \> - \< \hV, i \xi^3 \tiK(\xi) \tiB \hV \> = 0.
\end{equation}
Now, by virtue of $\tiK$ being skew-symmetric, we clearly have the relation 
\[
\Re \< \hV, i \xi \tiK \hV_t \> = \tfrac{1}{2} \xi \partial_t \< \hV, i \tiK \hV \>.
\] 
This implies that taking the real part of equation \eqref{la9} leads to
\[
- \tfrac{1}{2} \xi \partial_t \< \hV, i \tiK(\xi) \hV \> + \xi^2 \< \hV, [\tiK(\xi) \tiA(\xi)]^s \hV \> = \Re \big( i \xi^3 \< \hV, \tiK(\xi) \tiB \hV \> \big).
\]
Since $\tiB$ is positive semi-definite and symmetric, we now estimate,
\[
\begin{aligned}
\Re \big( i \xi^3 \< \hV, \tiK(\xi) \tiB \hV \> \big) &\leq \left| \xi^3 \langle \hV, \tiK(\xi) \tiB \hV \rangle \right| \\
&= \left| \langle \sqrt{\vep} \xi \hV, (\xi \tiK(\xi)) \tiB^{1/2} \frac{1}{\sqrt{\vep}} \xi \tiB^{1/2} \hV \rangle \right|\\
&\leq \vep \xi^2 |\hV|^2 + C_\vep \xi^2 | \tiB^{1/2} \hV |^2\\
&= \vep \xi^2 |\hV|^2 + C_\vep \xi^2 \langle \hV, \tiB \hV \rangle,
\end{aligned}
\]
for any $\vep > 0$ and where $C_\vep > 0$ is a uniform constant that only depends on $|\tiB^{1/2}|$, $\vep > 0$ and on the uniform bound for $|\xi \tiK(\xi)|$ in \eqref{tiKbded}. Upon substitution, we obtain
\begin{equation}
\label{la10}
- \tfrac{1}{2} \xi \partial_t \< \hV, i \tiK(\xi) \hV \> + \xi^2 \< \hV, [\tiK(\xi) \tiA(\xi)]^s \hV \> \leq \vep \xi^2 |\hV|^2 + C_\vep \xi^2 \langle \hV, \tiB \hV \rangle.
\end{equation}
Now multiply inequality \eqref{la10} by any $\delta > 0$ and add it to equation \eqref{la8} in order to get the following estimate,
\begin{equation}
\label{la11}
\begin{aligned}
\tfrac{1}{2} \partial_t \big( |\hV|^2 - \delta \xi \< \hV, i \tiK(\xi) \hV \> \big) + \xi^2 \big( \delta \< \hV, [\tiK(\xi) \tiA(\xi)]^s \hV \> &+ (1-\delta C_\vep)  \< \hV, \tiB \hV \> \big) \\&\leq \vep \delta \xi^2 |\hV|^2.
\end{aligned}
\end{equation}
Let us define the quantity,
\[
\Upsilon := |\hV|^2 - \delta \xi\< \hV, i \tiK(\xi) \hV \>.
\]
By virtue of $\tiK(\xi)$ being skew-symmetric, it is easy to verify that the quantity $\Upsilon$ is real. In addition, since $|\xi \tiK(\xi)| \leq C$ uniformly for all $\xi$ (see \eqref{tiKbded}), we deduce the existence of $1 \gg \delta_0 > 0$ sufficiently small such that there exists a uniform constant $C_1 > 0$ for which
\[
C_1^{-1} |\hV|^2 \leq \Upsilon \leq C_1 |\hV|^2,
\]
provided that $0 < \delta < \delta_0$. Therefore, $\Upsilon$ is an energy, equivalent to $|\hV|^2$, for $\delta > 0$ small.

Let $\overline{\gamma} > 0$ be the uniform constant appearing in \eqref{compmatprop} and choose $\vep = \tfrac{1}{2}\overline{\gamma}$. (This fixes the constant $C_\ep$.) Hence, we may take $0 < \delta < \delta_0$ small enough such that $\delta= \mbox{min} \left\lbrace \delta, 1-\delta C_\vep \right\rbrace$. From one of the main properties of compensating matrix symbols, namely, estimate \eqref{compmatprop}, we obtain 
\[
\delta \< \hV, [\tiK(\xi) \tiA(\xi)]^s \hV \> + (1-\delta C_\vep)\<\hV, \tiB \hV \> \geq \delta \langle \hV, ([\tiK(\xi) \tiA(\xi)]^s + \tiB) \hV \rangle 
\geq \delta \overline{\gamma} |\hV|^2.
\] 
Substituting into \eqref{la11} yields
\[
\partial_t \Upsilon + c_0 \xi^2 \Upsilon \leq 0,
\]
where $c_0 := \delta \overline{\gamma}/ C_1 > 0$. This implies estimate \eqref{bestV} and the lemma is proved.
\end{proof}
\begin{remark}
From estimate \eqref{bestV} we directly deduce that the eigenvalues in Fourier space of system \eqref{64} (equivalently, of system \eqref{30}) satisfy 
$\lambda(\xi) \leq  - c_0 \xi^2$, with $c_0 > 0$, yielding a decay rate equal to that of the heat kernel. This dissipative structure can be said to be of regularity-gain type (see Remark \ref{remtypediss}).
\end{remark}

As a consequence of Lemma \ref{lembee} we obtain the following estimate for the solutions to \eqref{63}.

\begin{corollary}
\label{corlindecayW}
The solutions $\hW(\xi,t)$ to the linear system \eqref{63} satisfy the estimate 
\begin{equation}
\label{39}
\begin{aligned}
(1+\xi^{2})\vert &\hW_{1}(\xi, t) \vert^{2} + \vert \hW_{2}(\xi, t) \vert^{2} + \vert \hW_{3}(\xi, t) \vert^{2} \leq  \\ 
&\leq C \exp(- 2c_0 \xi^{2} t ) \big[ (1+\xi^{2})\vert \hW_{1}(\xi, 0) \vert^{2} + \vert \hW_{2}(\xi, 0) \vert^{2}+ \vert \hW_{3}(\xi, 0) \vert^{2} \big],  
\end{aligned}
\end{equation}
for all $t\geq 0$, $\xi \in \R $ and some uniform constant $C>0$.
\end{corollary}
\begin{proof}
Suppose $\hW = \hW(\xi,t)$ is a solution to system \eqref{63}. Then from transformation \eqref{hU} we know that $\hV = S(\xi)^{1/2} (A^0)^{1/2} \hW$ satisfies \eqref{64} and estimate \eqref{bestV}. Let us compute
\[
\begin{aligned}
|\hV|^2 &= | S(\xi)^{1/2} (A^0)^{1/2} \hW |^2 \\
&= \left| \frac{1}{\bthe^{1/2}} \begin{pmatrix} \beta(\xi)^{1/2}/ \, \bp_\rho^{1/2} & 0 & 0 \\ 0 & 1 & 0 \\ 0 & 0 & 1 \end{pmatrix} 
\begin{pmatrix}  \bp_\rho^{1/2} / \brho^{1/2} & 0 & 0 \\ 0 & \brho^{1/2} & 0 \\ 0 & 0 & \be_\theta^{1/2} \brho^{1/2} / \bthe^{1/2} \end{pmatrix} \hW \right|^2\\
&= \frac{\beta(\xi)}{\brho \, \bthe} |\hW_1|^2 + \frac{\brho}{\bthe}  |\hW_2|^2 + \frac{\be_\theta \brho}{\bthe^2}  |\hW_3|^2.
\end{aligned}
\]
Therefore, from estimate \eqref{bestV} we obtain
\[
\begin{aligned}
\beta(\xi)  |\hW_1|^2 +  |\hW_2|^2 &+  |\hW_3|^2 \leq \tilde{C} |\hV|^2 \leq \tilde{C} C \exp (- 2c_0 \xi^2 t) |\hV(\xi,0)|^2\\
&\leq \bar{C} \exp (- 2c_0 \xi^2 t) \big[ \beta(\xi)  |\hW_1(\xi,0)|^2 +  |\hW_2(\xi,0)|^2 +  |\hW_3(\xi,0)|^2 \big],
\end{aligned}
\]
for some uniform constant $\bar{C} > 0$. From \eqref{defbeta} we know there exists constants $C_j > 0$ such that $C_2 (1+\xi^2) \leq \beta(\xi) \leq C_1(1+\xi^2)$ for all $\xi \in \R$. Upon substitution we obtain estimate \eqref{39}.
\end{proof}

The pointwise estimate \eqref{39} in Fourier space then implies the decay in time of the solutions to the original linear system \eqref{27}. For later use we now assume that the initial condition for $W$ is also in $L^1(\R)^3$.

\begin{lemma}
\label{lemdecayW}
Assume $W = W(x,t)$ is a solution to the linear system \eqref{27} with initial data $W(0)\in \big( H^{s+1}(\R)\times H^{s}(\R)\times H^{s}(\R) \big) \cap \big(L^{1}(\R) \big)^{3}$ for some $s\geq 3$. Then for each fixed $0 \leq \ell \leq s$ there holds the estimate
\begin{equation}
\label{40}
\begin{aligned}
 \Big( \Vert \partial_{x}^{\ell} W_{1}(t) \Vert_{1}^{2} + \Vert \partial_{x}^{\ell} W_{2}(t) \Vert_{0}^{2} &+ \Vert \partial_{x}^{\ell} W_{3}(t) \Vert_{0}^{2} \Big)^{1/2} \leq  \\ &\leq Ce^{-c_{1}t} \left( \Vert \partial_{x}^{\ell} W_{1}(0) \Vert_{1}^{2} + \Vert \partial_{x}^{\ell} W_{2}(0) \Vert_{0}^{2}+ \Vert \partial_{x}^{\ell} W_{3}(0) \Vert_{0}^{2} \right)^{1/2} + \\ & \quad+ C\big( 1+t \big)^{-(\ell/2+1/4)} \Vert W(0) \Vert_{L^{1}},
\end{aligned}
\end{equation}
holds for all $t\geq 0$ and some uniform constants $C$, $c_{1}>0$.
\end{lemma}
\begin{proof}
Fix $\ell \in [0,s]$, multiply estimate \eqref{39} by $\xi^{2\ell}$ and integrate in $\xi \in \R$. This yields,
\[
\int_\R \xi^{2\ell} \big[ (1+\xi^2) |\hW_1(\xi,t)|^2 + |\hW_2(\xi,t)|^2 + |\hW_3(\xi,t)|^2 \big] \, d\xi \leq C ( J_1(t) + J_2(t)),
\]
where 
\[
\begin{aligned}
J_1(t) &= \int_{-1}^1 \xi^{2\ell} \exp(-2 c_0 \xi^2 t) \big[ (1+\xi^2) |\hW_1(\xi,0)|^2 + |\hW_2(\xi,0)|^2 + |\hW_3(\xi,0)|^2 \big] \, d\xi,\\
J_2(t) &= \int_{\{|\xi| \geq 1\}} \xi^{2\ell} \exp(-2 c_0 \xi^2 t) \big[ (1+\xi^2) |\hW_1(\xi,0)|^2 + |\hW_2(\xi,0)|^2 + |\hW_3(\xi,0)|^2 \big] \, d\xi.
\end{aligned}
\]
Since $\exp(-2 c_0 \xi^2 t) \leq \exp(-c_0 t)$ for all $|\xi| \geq 1$ we clearly have
\[
\begin{aligned}
J_2(t) &\leq e^{-c_0 t} \int_{|\xi| \geq 1} \xi^{2\ell} \big[ (1+\xi^2) |\hW_1(\xi,0)|^2 + |\hW_2(\xi,0)|^2 + |\hW_3(\xi,0)|^2 \big] \, d\xi \\
&\leq e^{-c_0 t} \big( \| \partial_x^\ell W_1(0) \|_1^2 +  \| \partial_x^\ell W_2(0) \|_0^2 +  \| \partial_x^\ell W_3(0) \|_0^2 \big),
\end{aligned}
\]
for all $t \geq 0$. In order to estimate $J_1(t)$, we observe that for any $\ell \in [0,s]$ and any $c_0 > 0$, the integral
\[
I_0(t) := (1+t)^{\ell + 1/2} \int_{-1}^1 \xi^{2\ell} e^{-2c_0 \xi^2 t} \, d\xi,
\]
is uniformly bounded above (see Lemma A.1 in \cite{PlV22}). Therefore,
\[
\begin{aligned}
J_1(t) \leq 2 \int_{-1}^1 \xi^{2\ell} \exp (-2 c_0 \xi^2 t) |\hW(\xi,0)|^2  \,  d\xi &\leq 2 \sup_{\xi \in \R} |\hW(\xi,0)|^2 (1+t)^{-\ell-1/2} I_0(t) \\
&\leq C (1+t)^{-(\ell + 1/2)} \| W(0) \|_{L^1},
\end{aligned}
\]
for some uniform $C > 0$. Combine both estimates to obtain \eqref{40} with $c_1 := \tfrac{1}{2} c_0 > 0$.
\end{proof}

Finally, let us write equation \eqref{29} as
\[
W_t = \cA W,
\]
where the differential operator $\cA$ is given by
\begin{equation}
\label{defcalA}
\cA := (A^0)^{-1} \cL = (A^0)^{-1} \big( - A^1(\bU) \partial_x + B(\bU) \partial_x^2 + C(\bU) \partial_x^3 \big).
\end{equation}
We are interested in the semigroup generated by $\cA$, that is, in the solutions to the Cauchy problem
\begin{equation}
\label{cauchyW}
\begin{aligned}
W_t &= \cA W,\\
W(0) &= f,
\end{aligned}
\end{equation}
for some initial condition $f$ in an appropriate Banach space. Taking the Fourier transform we arrive at
\[
\hW_t + M(i\xi) \hW=0,
\]
with
\begin{equation}
\label{defsymbolM}
M(i\xi) := (A^0)^{-1} \big( i \xi A(\xi) - (i\xi)^2 B\big).
\end{equation}
Hence, the solutions to the Cauchy problem \eqref{cauchyW} can be recast in terms of the inverse Fourier transform as
\[
W(x,t) = \big( e^{t \cA} f \big)(x) = \frac{1}{\sqrt{2\pi}} \int_\R e^{i \xi x} e^{-tM(i\xi)} \widehat{f} (\xi) \, d\xi,
\]
where $\big\{ e^{t \cA} \big\}_{t \geq 0}$ denotes the semigroup generated by the linear operator $\cA$. Therefore, from Lemma \ref{lemdecayW} we immediately obtain the following decay rate of the semigroup.

\begin{corollary}\label{SemDecCor}
For any $f\in \big( H^{s+1}(\R)\times H^{s}(\R)\times H^{s}(\R) \big) \cap \big(L^{1}(\R) \big)^{3} $, $s\geq 3$, and $0 \leq \ell \leq s$, $t\geq 0$, the estimate 
\begin{equation}
\label{44}
\begin{aligned}
\Big( \Vert \partial_{x}^{\ell} \big( e^{t\cA}f \big)_{1}(t) \Vert_{1}^{2} &+ \Vert \partial_{x}^{\ell} \big( e^{t\cA}f \big)_{2}(t) \Vert_{0}^{2}+ \Vert \partial_{x}^{\ell} \big( e^{t\cA}f \big)_{3}(t) \Vert_{0}^{2} \Big)^{1/2}\leq \\ &\leq Ce^{-c_{1}t} \left( \Vert \partial_{x}^{\ell} f_{1} \Vert_{1}^{2} + \Vert \partial_{x}^{\ell} f_{2} \Vert_{0}^{2}+ \Vert \partial_{x}^{\ell} f_{3} \Vert_{0}^{2} \right)^{1/2} + C\big( 1+t \big)^{-(\ell/2+1/4)} \Vert f \Vert_{L^{1}},
\end{aligned}
\end{equation}
holds for some uniform $C$, $c_{1}>0$.
\end{corollary}

\section{Global decay of perturbations of equilibrium states}
\label{secglobal}

In this section we obtain both the global existence and decay of perturbations of equilibrium states for the nonlinear NSFK system \eqref{NSFK} - \eqref{defKw}. The linear estimates from the previous section constitute a key ingredient along the proof. 

\subsection{Nonlinear energy estimates}
\label{secnonlinee}

Fix any constant state $\bU = (\brho, \bu, \bthe) \in \cU$ and suppose that hypotheses \hyperref[H1]{\rm{(H$_1$)}} - \hyperref[H4]{\rm{(H$_4$)}} hold. Then for initial conditions, $U (0) = U_0 = (\rho_0, u_0, \theta_0)$ such that
\[
U_0 - \bU \in \big( H^{s+1}(\R) \times H^s(\R) \times H^s(\R) \big) \cap \big( L^1(\R) \big)^3,
\]
for some $s \geq 3$, and satisfying the assumptions of Theorem \ref{themlocale}, there exists a local solution to the NSFK system \eqref{NSFK} - \eqref{defKw},
\[
U = (\rho,u,\theta) \in X_s\big((0,T); \tfrac{1}{2}m_1,2M_1,\tfrac{1}{2}m_2,2M_2\big),
\]
for some time $T > 0$ and constants $m_j, M_j$, $j = 1,2$, such that $U(\cdot, t) \in \cU$ a.e. in $x\in \R$ and for all $t \in [0,T]$. By virtue of Proposition \ref{propgoodW} and Lemma \ref{lemregW}, the perturbation variables defined in \eqref{defWvar} solve the nonlinear system \eqref{Wsystem}, which we now recast as
\begin{equation}
\label{nonlinearW}
W_t = \cA W + \partial_x N(U,U_x,U_{xx}),
\end{equation}
where $\cA$ is the operator defined in \eqref{defcalA} and the nonlinear term is given by
\[
N(U,U_x,U_{xx}) := (A^0)^{-1} \tiN(U,U_x,U_{xx}).
\]
Thanks to the structure of $A^0$ (see \eqref{barA0A1B}) and the form of $\tiN$ described in Lemma \ref{lemorderN}, we deduce that
\begin{equation}
\label{newN}
N(U,U_x,U_{xx}) = O\big(|U-\bU|^2 + |U-\bU||U_x| + |U-\bU||\rho_{xx}| + |U_x|^2 + |\rho_x||\rho_{xx}|\big) \begin{pmatrix} 0 \\ 1 \\1 
\end{pmatrix}.
\end{equation}
The initial condition for $W$ is given by
\[
W(0) = D_U f^0(\bU)^{-1} ( f^0(U_0) - f^0(\bU)) + D_U f^0(\bU)^{-1} \Gamma^0(U_0, \partial_x U_0),
\]
which, because of Lemma \ref{lemregW}, satisfies
\[
C_0^{-1}\vertiii{U_0 - \bU}_\ell \leq \vertiii{W(0)}_\ell \leq C_0 \vertiii{U_0 - \bU}_\ell,
\]
for each $0 \leq \ell \leq s$ (the $\vertiii{\cdot}_\ell$ norm is defined in \eqref{triplenorm}). In particular, for $\ell = s$ we have
\[
C_0^{-1} E_s(0)^{1/2} \leq \overline{E}_s(0)^{1/2} \leq C_0 E_s(0)^{1/2},
\]
where
\[
\overline{E}_s(t) := \sup_{\tau \in [0,t]} \vertiii{W(\tau)}_s^2,
\]
and $E_s(t) = \sup_{\tau \in [0,t]} \vertiii{U(\tau) - \bU}_s^2$ is defined in \eqref{defEs}. Hence, the solutions to \eqref{nonlinearW} with initial condition $W(0)$ can be expressed via the variation of constants formula,
\[
W(x,t) = e^{t \cA} W(0) + \int_0^t e^{(t - \tau) \cA} \partial_x (N(U,U_x,U_{xx})(\tau)) \, d \tau.
\]

Fix $0 \leq \ell \leq s-1$. Upon application of the semigroup decay estimates from Corollary \ref{SemDecCor} we obtain
\begin{equation}
\label{50}
\begin{aligned}
\Big( \Vert \partial_{x}^{\ell} W_{1}(t) \Vert_{1}^{2} + \Vert \partial_{x}^{\ell} W_{2}(t) \Vert_{0}^{2} &+ \Vert \partial_{x}^{\ell} W_{3}(t) \Vert_{0}^{2} \Big)^{1/2} \leq Ce^{-c_{1}t} \Big( \Vert \partial_{x}^{\ell} W_{1}(0) \Vert_{1}^{2} + \Vert \partial_{x}^{\ell} W_{2}(0) \Vert_{0}^{2}+ \Vert \partial_{x}^{\ell} W_{3}(0) \Vert_{0}^{2} \Big)^{1/2} + \\ 
& \quad+ C( 1+t )^{-(\ell/2+1/4)} \Vert W(0) \Vert_{L^{1}} + \\ 
& \quad+ C\int_{0}^{t} \left( \Vert \partial_{x}^{\ell} \big(e^{-(t-\tau)\cA}\partial_x N\big)_{1}(\tau) \Vert_{1}^{2} + \Vert \partial_{x}^{\ell} \big(e^{-(t-\tau)\cA}\partial_x N\big)_{2}(\tau) \Vert_{0}^{2}+ \right. \\ 
& \qquad \qquad  \left.+ \Vert \partial_{x}^{\ell} (e^{-(t-\tau)\cA}\partial_x N)_{3}(\tau) \Vert_{0}^{2} \right)^{1/2} \: d\tau,
\end{aligned}
\end{equation}
where we have surpressed the dependence on $U$, $U_x$ and $U_{xx}$ of the nonlinear terms for shortness in the notation. Let us now apply the identity
\[
\partial_x^{\ell} \big( e^{t \cA} \partial_x f \big) = \partial_x^{\ell+1} \big( e^{t \cA} f\big),
\]
for any $f\in H^{s}(\R)$ and any $0 \leq \ell \leq s-1$, which can be easily verified using Fourier transform and the symbol for the semigroup, $\hW = e^{-t M(i\xi)} \hW(0)$, in \eqref{defsymbolM}. Using this identity and, for any fixed $0 \leq \ell \leq s-1$, the estimates in \eqref{44} with $\ell +1\leq s$ replacing $\ell$, we obtain
\begin{equation}
\label{51}
\begin{aligned}
\int_{0}^{t} &\left[ \Vert \partial_{x}^{\ell} \big(e^{-(t-\tau)\cA}\partial_x N\big)_{1}(\tau) \Vert_{1}^{2} + \Vert \partial_{x}^{\ell} \big(e^{-(t-\tau)\cA}\partial_x N\big)_{2}(\tau) \Vert_{0}^{2} + \Vert \partial_{x}^{\ell} (e^{-(t-\tau)\cA}\partial_x N)_{3}(\tau) \Vert_{0}^{2} \right]^{1/2} \! d\tau \\
&\leq C \int_{0}^{t} e^{-c_1(t-\tau)} \Big[ \| \partial_x^{\ell+1} N_1(\tau) \|_1^2 + \| \partial_x^{\ell+1} N_2(\tau) \|_0^2 + \| \partial_x^{\ell+1} N_3(\tau) \|_0^2 \Big] \, d\tau +\\
&\qquad + C \int_0^t (1+t-\tau)^{-(\ell/2 + 3/4)} \| N(\tau) \|_{L^1} \, d\tau.
\end{aligned}
\end{equation}
Combine \eqref{50} and \eqref{51} to arrive at
\begin{equation}
\label{52}
\begin{aligned}
\Big( \Vert \partial_{x}^{\ell} W_{1}(t) \Vert_{1}^{2} + \Vert \partial_{x}^{\ell} W_{2}(t) \Vert_{0}^{2} &+ \Vert \partial_{x}^{\ell} W_{3}(t) \Vert_{0}^{2} \Big)^{1/2} \leq \\
&\leq Ce^{-c_{1}t} \Big( \Vert \partial_{x}^{\ell} W_{1}(0) \Vert_{1}^{2} + \Vert \partial_{x}^{\ell} W_{2}(0) \Vert_{0}^{2}+ \Vert \partial_{x}^{\ell} W_{3}(0) \Vert_{0}^{2} \Big)^{1/2} + \\ 
& \quad+ C( 1+t )^{-(\ell/2+1/4)} \Vert W(0) \Vert_{L^{1}} + \\ 
& \quad + C \int_{0}^{t} e^{-c_1(t-\tau)} \Big[ \| \partial_x^{\ell+1} N_1(\tau) \|_1^2 + \| \partial_x^{\ell+1} N_2(\tau) \|_0^2 + \| \partial_x^{\ell+1} N_3(\tau) \|_0^2 \Big] \, d\tau +\\
&\quad + C \int_0^t (1+t-\tau)^{-(\ell/2 + 3/4)} \| N(\tau) \|_{L^1} \, d\tau.
\end{aligned}
\end{equation}
Therefore, recalling that $N_1 \equiv 0$, summing up estimates \eqref{52} for $\ell = 0, 1, \ldots, s-1$, we obtain
\begin{equation}
\label{53}
\begin{aligned}
\vertiii{W(t)}_{s-1} &\leq C(1+t)^{-1/4} \big( \vertiii{W(0)}_{s-1} + \| W(0) \|_{L^1} \big) + C \int_0^t (1+t-\tau )^{-3/4} \| N(\tau) \|_{L^1} \, d \tau +\\
&\qquad + C \int_0^t e^{-c_1 (t - \tau)} \big( \| N_2(\tau) \|_s^2 + \| N_3(\tau) \|_s^2 \big)^{1/2}\, d\tau.
\end{aligned}
\end{equation}

Let us ow estimate the Sobolev norms of the nonlinear terms $N$ that appear in \eqref{53}. From \eqref{newN} we know that
\[
N_i (U,U_x,U_{xx}) = O\big(|U-\bU|^2 + |U-\bU||U_x| + |U-\bU||\rho_{xx}| + |U_x|^2 + |\rho_x||\rho_{xx}|\big), \qquad i = 2,3.
\]
Since $s \geq 3$, use the Banach algebra properties and the Sobolev calculus inequalities of Lemmata \ref{lemaux1}
and \ref{lemaux2} to obtain the following estimates (details are omitted):
\[
\Vert |U-\bU|^{2} \Vert_{s} \leq C \Vert U -\bU \Vert_{s}\Vert U- \bU \Vert_{1}, 
\]
\[
\begin{aligned}
\Vert |U- \bU|| U_{x} |\Vert_{s} & \leq C \big(  \Vert U- \bU \Vert_{s}\Vert U_{x} \Vert_{L^{\infty}} + \Vert U_{x} \Vert_{s} \Vert U- \bU \Vert_{L^{\infty}} \big) \\ & \leq C\big( \Vert U -\bU \Vert_{s} \Vert U-\bU \Vert_{2} + \Vert U_{x} \Vert_{s} \Vert U- \bU \Vert_{1} \big),
\end{aligned} 
\]
\[ 
\begin{aligned}
\Vert |U -\bU || \rho_{xx}| \Vert_{s} & \leq C \big( \Vert U- \bU \Vert_{s} \Vert \rho_{xx} \Vert_{L^{\infty}} + \Vert \rho_{xx} \Vert_{s} \Vert U- \bU \Vert_{L^{\infty}}   \big) \\ &\leq C \big( \Vert U- \bU \Vert_{s}\Vert \rho - \brho \Vert_{3} +   \Vert \rho_{x} \Vert_{s+1} \Vert U- \bU \Vert_{1} \big),
\end{aligned} 
\]
\[
\Vert \vert U_{x} \vert^{2}  \Vert \leq C \Vert U_x \Vert_{s} \Vert U_{x} \Vert_{L^{\infty}} \leq C  \Vert U_{x} \Vert_{s} \Vert U- \bU \Vert_{2},
\]
and,
\[
\begin{aligned}
\Vert |\rho_{x}|| \rho_{xx}| \Vert_{s} &\leq C \big( \Vert (\rho - \brho)_{x} \Vert_{s} \Vert \rho_{xx} \Vert_{L^{\infty}} + \Vert \rho_{xx} \Vert_{s} \Vert (\rho -\brho)_{x} \Vert_{L^{\infty}} \big) \\ &\leq C \big( \Vert \rho -\brho \Vert_{s+1} \Vert \rho- \brho \Vert_{3} + \Vert \rho_{x} \Vert_{s+1}\Vert \rho -\brho \Vert_{2} \big).
\end{aligned}
\]
Therefore, we arrive at
\begin{equation}
\label{54}
\begin{aligned}
\| N_i(\tau) \|_s &\leq C \big( \vertiii{U(\tau) - \bU}_s \vertiii{U(\tau) - \bU}_{s-1} + \vertiii{U(\tau) - \bU}_{s-1}  \| U_x(\tau) \|_s + \\
&\qquad \quad + \vertiii{U(\tau) - \bU}_{s-1} \| \rho_x(\tau) \|_{s+1}\big),
\end{aligned}
\end{equation}
for any $\tau \in [0,T]$, $i = 2,3$. Also, from \eqref{newN} we clearly have
\begin{equation}
\label{55}
\begin{aligned}
\| N_i(\tau) \|_{L^1} \leq C \| U(\tau) - \bU \|_2^2 &\leq C \| U(\tau) - \bU \|_s^2 \leq C \vertiii{U(\tau) - \bU}_{s-1}^2,
\end{aligned}
\end{equation}
for $i = 2,3$, because $s \geq 3$. As a result, we can substitute estimates \eqref{54} and \eqref{55} into \eqref{53} to obtain
\begin{equation}
\label{56}
\begin{aligned}
\vertiii{W(t)}_{s-1} &\leq C(1+t)^{-1/4} \big( \vertiii{W(0)}_{s-1} + \| W(0) \|_{L^1} \big) + \\
&+ C \int_0^t (1+t-\tau )^{-3/4} \vertiii{U(\tau) - \bU}_{s-1}^2 \, d \tau +\\
&+ C \int_0^t e^{-c_1 (t - \tau)} \Big[  \vertiii{U(\tau) - \bU}_s \vertiii{U(\tau) - \bU}_{s-1} + \vertiii{U(\tau) - \bU}_{s-1} \| U_x(\tau) \|_s + \\
&\qquad \qquad \qquad \qquad + \vertiii{U(\tau) - \bU}_{s-1} \| \rho_x(\tau) \|_{s+1}\Big] \, d\tau.
\end{aligned}
\end{equation}

Let us first estimate the last integral of the right hand side of \eqref{56}. Since clearly,
\[
\| U_x(\tau) \|_s +  \| \rho_x(\tau) \|_{s+1} \leq C \big( \| \rho_x(\tau) \|_{s+1}^2 + \| u_x(\tau) \|_s^2 + \| \theta_x(\tau) \|_s^2 \big)^{1/2}, 
\]
for all $\tau \in (0,T)$, we obtain,
\begin{align}
\int_0^t e^{-c_1 (t - \tau)} &\Big[  \vertiii{U(\tau) - \bU}_s \vertiii{U(\tau) - \bU}_{s-1} + \vertiii{U(\tau) - \bU}_{s-1}  \| U_x(\tau) \|_s + \nonumber\\
&+\vertiii{U(\tau) - \bU}_{s-1} \| \rho_x(\tau) \|_{s+1}\Big] \, d\tau \nonumber\\
&\leq C \sup_{0 \leq \tau \leq t} \vertiii{U(\tau)-\bU}_{s} \int_0^t e^{-c_1(t-\tau)} \vertiii{U(\tau) - \bU}_{s-1} \, d\tau + \nonumber\\
& + C \int_0^t e^{-c_1 (t - \tau)} \big( \| \rho_x(\tau) \|_{s+1}^2 + \| u_x(\tau) \|_s^2 + \| \theta_x(\tau) \|_s^2 \big)^{1/2} \vertiii{U(\tau) - \bU}_{s-1} \, d\tau \nonumber\\
&\leq C \sup_{0 \leq \tau \leq t} \vertiii{U(\tau)-\bU}_{s} \int_0^t e^{-c_1(t-\tau)} \vertiii{U(\tau) - \bU}_{s-1} \, d\tau + \nonumber\\
&+ C \left( \int_0^t \| \rho_x(\tau) \|_{s+1}^2 + \| u_x(\tau) \|_s^2 + \| \theta_x(\tau) \|_s^2 \, d\tau \right)^{1/2} \left( \int_0^t e^{-2c_1 (t - \tau)}  \vertiii{U(\tau) - \bU}_{s-1}^2 \, d\tau \right)^{1/2}.\label{57}
\end{align}
From the equivalence of the norms of $W$ and $U - \bU$ expressed in Lemma \ref{lemregW} (see estimates \eqref{equivl} and \eqref{L1estUW}) we have,
\[
\begin{aligned}
\vertiii{W(0)}_{s-1} + \| W(0) \|_{L^1} &\leq C \big( \vertiii{U(0) - \bU}_{s-1} +  \| \rho_0 - \brho \|_1^{2} + \| U(0) - \bU \|_{L^1} \big),
\end{aligned}
\]
and,
\[
\vertiii{U(t)-\bU}_{s-1} \leq C \vertiii{W(t)}_{s-1}, \qquad \qquad \forall \; t \in [0,T],
\]
for some uniform constant $C > 0$. Thus, upon substitution of these inequalities and estimate \eqref{57} into \eqref{56}, we arrive at
\begin{align}
\vertiii{U(t) - \bU}_{s-1} &\leq C(1+t)^{-1/4} \big( \vertiii{U(0) - \bU}_{s-1} + \| \rho_0 - \brho \|_1^{2} + \| U(0) - \bU \|_{L^1} \big) + \nonumber\\
&+  C \sup_{0 \leq \tau \leq t} \vertiii{U(\tau)-\bU}_{s} \int_0^t e^{-c_1(t-\tau)} \vertiii{U(\tau) - \bU}_{s-1} \, d\tau +  \nonumber \\
&+ C \left( \int_0^t \| \rho_x(\tau) \|_{s+1}^2 + \| u_x(\tau) \|_s^2 + \| \theta_x(\tau) \|_s^2 \, d\tau \right)^{1/2} \!\!\!\left( \int_0^t e^{-c_1 (t - \tau)}  \vertiii{U(\tau) - \bU}_{s-1}^2 \, d\tau \right)^{1/2} + \nonumber\\
&+ C \int_0^t (1+t-\tau )^{-3/4} \vertiii{U(\tau) - \bU}_{s-1}^2 \, d \tau. \label{58}
\end{align}
In order to simplify the notation, let us define
\begin{equation}
\label{defGs}
G_s(t) := \sup_{0 \leq \tau \leq t} (1+ \tau)^{1/4} \vertiii{U(\tau)-\bU}_{s-1}, \qquad t \in [0,T].
\end{equation}
Therefore, from estimate \eqref{58} and recalling the definitions of $E_s(t)$ and of $F_s(t)$ in \eqref{defEs} and \eqref{defFs}, respectively, we obtain
\[
G_s(t) \leq C R_s(0) + C I_1(t) (E_s(t) + F_s(t))^{1/2} G_s(t) + C I_2(t) G_s(t)^2,
\]
where
\begin{equation}
\label{defRs0}
R_s(0) := \vertiii{U(0) - \bU}_{s-1} + \| \rho_0 - \brho \|_1^{2} + \| U(0) - \bU \|_{L^1},
\end{equation}
\[
\begin{aligned}
I_{1}(t) &:= \sup_{0 \leq \tau \leq t}(1 + \tau)^{1/4} \int_{0}^{\tau} e^{-c_{1}(\tau-z)}(1+z)^{-1/4} \: dz + \\ &\quad  + \sup_{0 \leq \tau \leq t} (1+\tau)^{1/4} \left[ \int_{0}^{\tau}e^{-2c_{1}(\tau-z)}(1+z)^{-1/2} \: dz \right]^{1/2},\\
I_{2}(t) &:= \sup_{0 \leq \tau \leq t} (1+\tau)^{1/4} \int_{0}^{\tau}(1+\tau -z)^{-3/4}(1+z)^{-1/2} \: dz. 
\end{aligned}
\]
One can prove that $I_1(t)$ and $I_2(t)$ are uniformly bounded in $t \geq 0$ (see Lemma A.1 in \cite{PlV22}). Therefore, we arrive at the estimate
\begin{equation}
\label{60}
G_s(t) \leq \cC R_s(0) + \cC (E_s(t) + F_s(t))^{1/2} G_s(t) + \cC G_s(t)^2,
\end{equation}
for all $t \in [0,T]$ and with a uniform constant $\cC > 0$ depending, at most, on $a_0$ and on $\bU$. 

\subsection{Global existence and decay}

We are now ready to prove our main result.

\begin{theorem}[global existence and time asymptotic decay]
\label{thmgloex} 
Under hypotheses \hyperref[H1]{\rm{(H$_1$)}} - \hyperref[H4]{\rm{(H$_4$)}}, suppose that $U(0)-\bU \in \left( H^{s+1}(\R) \times H^{s}(\R) \times H^{s}(\R) \right) \cap \big( L^{1}(\R) \big)^{3}$ for $s\geq 3$. Then there exists a positive constant $\ep$ ($< a_{0}$, with $a_{0} > 0$ as in Theorem \ref{themlocale}) such that if
\begin{equation}
\label{lessep}
R_{s}(0) = \vertiii{U(0) - \bU}_s + \| \rho_0 - \brho \|_1^{2}+ \| U(0) - \bU \|_{L^1} < \ep,
\end{equation}
then the Cauchy problem for system \eqref{NSFK} - \eqref{defKw} with initial condition $U(0)$ has a unique global solution $U(x,t)= (\rho, u, \theta)(x,t)$ satisfying 
\[
\begin{split}
& \:\: \rho-\brho \in C\left((0,\infty);H^{s+1}(\mathbb{R})) \cap C^{1}((0,\infty); H^{s-1}(\mathbb{R})\right), \\ & \:\: u-\bu, \theta- \bthe \in C\left((0,\infty); H^{s}(\mathbb{R})) \cap C^{1}((0,\infty);H^{s-2}(\mathbb{R})\right), \\ & \:\:(\rho_{x}, u_{x}, \theta_{x})\in L^{2}\left((0,\infty);H^{s+1}(\mathbb{R}) \times H^{s}(\mathbb{R}) \times H^{s}(\mathbb{R})\right).
\end{split}
\]
Moreover, the solution satisfies the estimates
\begin{equation}
\label{eee3}
\big(E_{s}(t) + F_{s}(t) \big)^{1/2} \leq C_2 \vertiii{U(0) - \bU}_s,
\end{equation}
and
\begin{equation}
\label{eee4}
\vertiii{U(t)-\bU}_{s-1} \leq C_3 (1+t)^{-1/4} \big( \vertiii{U(0) - \bU}_{s-1} + \| \rho_0 - \brho \|_1^{2} + \| U(0) - \bU \|_{L^1} \big),
\end{equation}
for all $t \in [0, \infty)$ and some uniform constants $C_j > 0$, $j = 2,3$.
\end{theorem}
\begin{proof}
Let $a_0 > 0$ and $0 < a_2 \leq a_0$ be the fixed constants from Theorem \ref{themlocale} and Corollary \ref{cor26}. Under hypotheses \hyperref[H1]{\rm{(H$_1$)}} - \hyperref[H4]{\rm{(H$_4$)}}, suppose that the initial condition satisfies
\[
U(0) - \bU = (\rho_0 - \brho, u_0 - \bu, \theta_0 - \bthe) \in \big( H^{s+1}(\R) \times H^{s}(\R) \times H^{s}(\R) \big) \cap (L^1(\R))^3,
\]
with $s \geq 3$. Let us assume that $U = (\rho, u, \theta) \in X_s\big((0,T); \tfrac{1}{2}m_1,2M_1,\tfrac{1}{2}m_2,2M_2\big)$ is a local solution to the Cauchy problem, for some $T>0$, such that $( E_{s}(T)+ F_{s}(T) )^{1/2}\leq a_{2}$. As for all $t \in [0,T]$ we have
\[
E_{s}(t)^{1/2} \leq (E_{s}(t)+F_{s}(t) )^{1/2} \leq (E_{s}(T) + F_{s}(T))^{1/2}\leq a_{2},
\]
from Corollary \ref{cor26} we obtain
\begin{equation}\label{estC2}
(E_{s}(t)+F_{s}(t) )^{1/2} \leq C_{2}E_{0}(s)^{1/2},
\end{equation}
for $0 \leq t \leq T$, with $C_{2}=C_{2}(a_{2})$ independent of $t$.

Let us observe that $E_{s}(0)^{1/2}= (E_{s}(0)+ F_{s}(0))^{1/2}\leq ( E_{s}(T)+ F_{s}(T) )^{1/2}\leq a_{2} \leq a_{0}$. Now, recalling Remark \ref{remLinftybound}, we know that $a_0 > 0$ can be chosen and fixed such that the local solution satisfies $U \in \cU$ for all $t \in [0, T]$. Consequently, the estimates established in Section \ref{secnonlinee} hold. Since for all $t \in [0,T]$ we have 
\[
(E_s(t) + F_s(t))^{1/2} \leq (E_s(T) + F_s(T))^{1/2} \leq C_2 E_s(0)^{1/2} \leq C_2 R_{s}(0) \leq C_{2} \ep_{1},
\]
for $R_{s}(0)\leq \ep_{1}$. Then, upon substitution into estimate \eqref{60}, we arrive at
\[
(1 - \cC C_2 \ep_{1}) G_s(t) \leq \cC R_s(0) + \cC G_s(t)^2.
\]
Taking $0 < \ep_{1} \ll 1$ small enough we obtain
\begin{equation}\label{Gexp}
G_s(t) \leq 2 \cC R_s(0) + 2 \cC G_s(t)^2
\end{equation}
for all $t \in [0,T]$. This estimate implies, in turn, that
\[
G_s(t) \leq C_{3} R_s(0),
\]
for all $t \in [0,T]$ and some positive constant $C_{3}\geq 1$, provided $R_{s}(0)\leq \ep_{1}$, with $0 < \ep_{1} \ll 1$ sufficiently small. Indeed, by defining $x=x(t):=G_{s}(t)$ for $t\in [0, T]$, we see that inequality \eqref{Gexp} means that the parabola defined by $y=2 \cC R_{s}(0)+ 2 \cC x^{2}$ is always above the line $y=x$. Then by taking $R_{s}(0)\leq \ep_{1}$, $0< \ep_{1} \ll 1 $ small enough, there exists positive constants $1\leq C_{3} < C_{4}$  such that estimate \eqref{Gexp} implies that for each $t\in [0,T]$ exactly one of the following is satisfied:
\begin{equation}\label{Dec1}
G_{s}(t) \leq C_{3} R_{s}(0), 
\end{equation}
or
\begin{equation}\label{Dec'1}
C_{4}R_{s}(0) \leq G_{s}(t).
\end{equation}
We are going to show that \eqref{Dec1} is satisfied for all $t\in [0, T]$. Given that for $t\in [0, T]$
\[
G_{s}(t):=  \sup_{0 \leq \tau \leq t} (1+ \tau)^{1/4} \vertiii{U(\tau)-\bU}_{s-1},
\]
where the triple norm is defined in \eqref{triplenorm}, and
\[
U-\bU = (\rho -\brho, u- \bu, \theta- \bthe) \in  C((0,T); H^{s+1}(\R) \times H^{s}(\R) \times H^{s}(\R)),
\]
we see that $G_{s}(t)$ is continuous in $t$, Therefore by the fact that $C_{3} < C_{4}$, a continuity argument shows that whenever \eqref{Dec1} or \eqref{Dec'1} is satisfied for some $t\in [0, T]$, then it is satisfied for all $t$ in that interval. Let us observe that
\[
G_{s}(0) = E_{s}(0)^{1/2} \leq R_{s}(0) \leq C_{3}R_{s}(0),
\]
whence we conclude estimate \eqref{Dec1} is satisfied for all $t\in [0, T]$. Therefore, for any $0 \leq t \leq T$ there holds
\[
(1+t)^{1/4} \vertiii{U(t) -\bU}_{s-1} \leq G_s(t) \leq C_{3}  R_s(0),
\]
or, equivalently,
\begin{equation}
\label{estC3}
\vertiii{U(t) -\bU}_{s-1} \leq C_3 (1+t)^{-1/4} \big( \vertiii{U(0) - \bU}_{s-1} +  \| \rho_0 - \brho \|_1^{2} + \| U(0) - \bU \|_{L^1} \big),
\end{equation}
for all $t \in [0,T]$ with some uniform $C_3 > 0$ and by virtue of \eqref{defRs0}.

Hence, we have shown that if $U = (\rho, u, \theta) \in X_s\big((0,T); \tfrac{1}{2}m_1,2M_1,\tfrac{1}{2}m_2,2M_2\big)$ is a local solution for the initial value problem, for some $T>0$, then for $a_{2} \leq a_{0}$, we can choose $\ep_{1}=\ep_{1}(a_{2}, \cC)$ such that whenever $(E_{s}(T)+F_{s}(T))^{1/2}\leq a_{2}$ and $R_{s}(0)\leq \ep_{1}$ hold, the we have estimates \eqref{estC2} and \eqref{estC3} for all  $t\in [0, T]$.

Let us assume that $R_{s}(0)\leq \ep \leq \ep_{1}$, for $\ep$ small enough. Then, clearly,
\begin{equation}
\label{estEs0}
E_s(0)^{1/2} = \vertiii{U(0) - \bU}_s \leq R_{s}(0) \leq \ep \leq \ep_{1} \leq a_{2} \leq a_{0}.
\end{equation}
Hence, by the local existence theorem \ref{themlocale} there exist $T_1 = T_1(a_0) > 0$ and a local solution $U = (\rho, u, \theta) \in X_s\big((0,T_1); \tfrac{1}{2}m_1,2M_1,\tfrac{1}{2}m_2,2M_2\big)$ to the Cauchy problem which, by virtue of estimate \eqref{localEE}, satisfies
\begin{equation}
\label{EsT1FsT1}
E_s(T_1) + F_s(T_1) \leq C_1 E_s(0) \leq C_1 \ep^2 \leq a_2^{2},
\end{equation}
provided that $0 < \ep \ll 1$ is sufficiently small, where $C_1 = C_1(a_0) > 0$ is the uniform constant of Theorem \ref{themlocale}.\\
Therefore, as $R_{s}(0)\leq \ep \leq \ep_{1}$ and $(E_s(T_1) + F_s(T_1))^{1/2} \leq a_2$, by the preceding discussion, we obtain estimated \eqref{estC2} and \eqref{estC3} for all $t\in[0, T_{1}]$.

Now, we observe that if $\ep > 0$ is sufficiently small, then
\[
\vertiii{U(T_1) - \bU}_s \leq (E_s(T_1) + F_s(T_1))^{1/2} \leq C_1^{1/2} E_s(0)^{1/2} \leq C_1^{1/2}R_{s}(0) \leq C_{1}^{1/2}\ep \leq a_{2} \leq a_{0}.
\]
This means that we may take $U(T_1)$ as initial condition and apply the local existence theorem \ref{themlocale} once again in order to obtain a solution in the time interval $[T_1, 2T_1]$. This solution satisfies the corresponding estimate \eqref{localEE}, which now reads
\[
\sup_{T_1 \leq \tau \leq 2T_1} \vertiii{U(\tau) -\bU}_s^2 + \int_{T_1}^{2T_1} \vertiii{U_x(\tau)}_s^2 \, d \tau \leq C_{1} \vertiii{U(T_1) - \bU}_s^{2} \leq  C_1 E_s(T_1).
\]
Therefore, by taking into account estimate \eqref{estC2} for $t\in [0, T_{1}]$, we deduce the estimate
\[
\begin{aligned}
E_s(2T_1) + F_s(2T_1) &= \sup_{0 \leq \tau \leq 2T_1} \vertiii{U(\tau) -\bU}_s^2 + \int_{0}^{2T_1} \vertiii{U_x(\tau)}_s^2 \, d \tau \\
&\leq \sup_{0 \leq \tau \leq T_1} \vertiii{U(\tau) -\bU}_s^2 + \int_{0}^{T_1} \vertiii{U_x(\tau)}_s^2 \, d \tau + \\
& \quad + \sup_{T_1 \leq \tau \leq 2T_1} \vertiii{U(\tau) -\bU}_s^2 + \int_{T_1}^{2T_1} \vertiii{U_x(\tau)}_s^2 \, d \tau \\
&\leq E_s(T_1) +F_s(T_1) + C_1 E_s(T_1)\\
& \leq (1+C_{1})(E_s(T_1) +F_s(T_1))  \\
& \leq (1+C_{1})C_{2}^{2}E_{s}(0) \\
& \leq (1 + C_1) C_2^{2}R_{s}(0)^{2} \leq (1 + C_1) C_2^{2}\ep^{2} \leq a_2^{2},
\end{aligned}
\]
provided that, once again, $0 < \ep \ll 1$ is small enough. This last estimate and the condition $R_{s}(0)\leq \ep \leq \ep_{1}$ allows to obtain estimates \eqref{estC2} and \eqref{estC3} for all $t\in [0,2T_{1}]$, with the same constants $C_{2}$ and $C_{3}$.

By an iteration process and repeating the argument in the intervals $[2T_1, 3T_1]$, $[3T_1, 4T_1]$ etc., we obtain local solutions on every interval of the form $[kT_1, (k+1)T_1]$, $k \in \N$, satisfying estimates \eqref{estC2} and \eqref{estC3}. Hence, we conclude that we can find a sufficiently small $0 < \ep \ll 1$ such that if the initial condition satisfies \eqref{lessep} then the solution can be extended globally in time. (Actually, from the above estimates it is clear that it suffices to choose
\begin{equation}
\label{chooseep}
0 < \ep = \min \left\{ \ep_{1}, a_{2}, \frac{a_{2}}{C_{1}^{1/2}}, \frac{a_{2}}{(1+C_{1})^{1/2}C_{2}} \right\},
\end{equation}
where all the constants involved in \eqref{chooseep} depend upon the fixed constants $a_0$ and $a_2$, insamuch as $C_1 = C_1(a_0)$, $C_2 = C_2(a_2)$, and $\ep_{1}= \ep_{1}(a_{2}, \cC)$, $\cC = \cC(a_0,\bU)$ being the constant of estimate \eqref{60}, and, hence, they are all uniform.) Moreover estimate \eqref{eee3} and the asymptotic decay rate in time \eqref{eee4} hold for all $t \in (0, \infty)$. The theorem is now proved.
\end{proof}

\section*{Acknowledgements}

The authors are grateful to Prof. H. Hattori for providing a copy of the paper \cite{HaLi96b}. R. G. Plaza thanks Prof. I. Santamar\'{\i}a-Holek for some interesting discussions on non-equilibrium thermodynamics. The work of J. M. Valdovinos was partially supported by CONACyT (Mexico), through a scholarship for graduate studies, grant no. 712874. The work of R. G. Plaza was partially supported by DGAPA-UNAM, project PAPIIT IN-104922.

\appendix
\section{Convex extension and symmetrization of the Navier-Stokes-Fourier model}
\label{ConvExtN-S}

In this section we review the notion of a convex extension for hyperbolic-parabolic systems of conservations laws (cf. \cite{KaTh83,KaSh88a}), which is a generalization of the concept of convex entropy/entropy flux pair introduced by Godunov \cite{Godu61a} and by Friedrichs and Lax \cite{FLa3} for hyperbolic systems. In particular, we recollect the calculation of such extension in the case of the NSF system, as performed by Kawashima and Shizuta \cite{KaSh88a}. 

Consider a viscous system of conservation laws in one dimension of the form
\begin{equation}
\label{vscl}
f^0(U)_t + f^1(U)_x = (G(U) U_x)_x,
\end{equation}
where $x \in \R$, $t \geq 0$, and the state variables $U$ take values in an open convex set $\cU \subset \R^n$. The vector valued functions $f^j : \cU \to\R^n$, $j = 0,1$, are smooth enough (at least, $f^j \in C^2(\cU;\R^n)$) and the viscosity tensor $G$ is an $n \times n$ matrix valued function of class $C^2(\cU;\R^{n \times n})$.

The conserved quantities can be expressed as functions of the state variables as
\[
V := f^0(U),
\]
for which we assume that $f^0 : \cU \to \cV$, with $\cV := f^0(\cU)$, is a diffeomorphism on its range with smooth inverse.

\begin{definition}
\label{defentropypair}
Let $\cE = \cE(V)$, $\cE : \cV \subset \R^n \to \R$, be a smooth function. Then $\cE$ is called an \emph{entropy function} for the viscous system \eqref{vscl} if the following statements hold:
\begin{itemize}
\item[(a)] $\cE$ is strictly convex in $\cV$ (i.e. the Hessian $D_V^2 \cE(V)$, $D_V^2 \cE :\cV \to \R^{n \times n}$, is positive definite for all $V \in \cV$).
\item[(b)] There exists a real valued smooth function $\Theta = \Theta(U)$, $\Theta : \cU \to \R$, such that
\begin{equation}
\label{fluxcond}
D_U\Theta(U) = D_V \cE(f^0(U)) D_U f^1(U),
\end{equation}
for all $U \in \cU$.
\item[(c)] The matrix fields
\begin{equation}
\label{coeffsAj}
A^j(U) := D_U f^0(U)^\top D_V^2\cE(f^0(U)) D_U f^j(U), \qquad j = 0,1,
\end{equation}
are symmetric for all $U \in \cU$.
\item[(d)] The viscosity tensor
\begin{equation}
\label{visctensorB}
B(U) := D_U f^0(U)^\top D_V^2\cE(f^0(U)) G(U),
\end{equation}
is symmetric, positive semi-definite for all $U \in \cU$.
\end{itemize}
In this case the system \eqref{vscl} is called \emph{symmetric dissipative} and the pair $(\cE,\Theta)$ is known as a \emph{convex entropy/entropy flux pair}.
\end{definition}

One of the main consequences of the existence of a convex entropy/entropy flux pair is the symmetrizability of system \eqref{vscl}. We now follow  \cite{KaSh88a} and review such procedure in the case of the NSF system in one dimension, which reads
\begin{equation}
\label{NSF1d}
\begin{aligned}
        \rho_t + (\rho u)_x &= 0,   \\
        (\rho u)_t + \big(\rho u^{2}+ p \big)_x &= ( \mu u_{x})_x, \\
        \big(\rho  e + \tfrac{1}{2}\rho u^{2} \big)_t + \big( \rho u \big( e + \tfrac{1}{2}u^2 \big) + pu \big)_x &= ( \alpha \theta_{x} + \mu u u_{x} )_x.       
\end{aligned} 
\end{equation}
Exactly as before, here $t>0$ and $x\in \R$ denote time and position, respectively, and the unknown scalar functions $\rho$, $u$, and $\theta$ represent the density, velocity field, and temperature, respectively. The pressure $p$ and internal energy $e$, as well as the viscosity $\mu > 0$ and heat conductivity $\alpha > 0$ coefficients, are assumed to be smooth functions of $(\rho, \theta) \in \cD$, and satisfying the usual conditions for classical fluids described in assumptions \hyperref[H1]{\rm{(H$_1$)}} and \hyperref[H3]{\rm{(H$_3$)}}. The state variables, denoted as $U = (\rho, u, \theta) \in \cU \subset \R^3$, range over the open convex set $\cU = \{ (\rho, u, \theta) \in \R^3 \, : \, (\rho,\theta) \in \cD \}$. Then system \eqref{NSF1d} can be recast in the form \eqref{vscl} where,
\begin{equation}
\label{3}
f^{0}(U) = \begin{pmatrix} \rho \\ \rho u \\ \rho(e + \tfrac{1}{2}u^{2}) \end{pmatrix}, \qquad f^{1}(U) = \begin{pmatrix}  \rho u \\ \rho u^{2} + p \\ \rho u (e + \tfrac{1}{2}u^{2})+ pu \end{pmatrix},
\end{equation}
and 
\begin{equation}
\label{viscG}
G(U) = \begin{pmatrix}
0 & 0 & 0 \\ 0 & \mu & 0 \\ 0 & \mu u & \alpha  
\end{pmatrix}.
\end{equation}
Notice that $f^0, f^1$ and $G$ are smooth functions of $U \in \cU$. The conserved quantities are $V = f^0(U) \in \cV$, where $\cV = f^0(\cU) \subset \R^3$. Notice also that $f^0$ defines a map $U \mapsto V$ whose Jacobian is given by
\begin{equation}
\label{Jacobf0}
D_UV = D_U f^0(U) = \begin{pmatrix}
1 & 0 & 0 \\ u & \rho & 0 \\ e + \tfrac{1}{2}u^{2} + \rho e_{\rho} & \rho u & \rho e_{\theta}
\end{pmatrix}.
\end{equation}
In view that the thermodynamic assumptions \hyperref[H1]{\rm{(H$_1$)}} and \hyperref[H3]{\rm{(H$_3$)}} hold, it is easy to verify that the mapping $U \mapsto V$ is injective on $\cU$ and that $D_U f^0$ is non-singular. Thus, $f^0$ is a diffeomorphism on its range (see \cite{KaSh88a}). The inverse of the Jacobian matrix \eqref{Jacobf0} is given by
\begin{equation}
\label{inverseDf0}
D_U f^0(U)^{-1} = \begin{pmatrix}
1 & 0 & 0 \\ -u/\rho & 1/\rho & 0 \\ (\tfrac{1}{2}u^{2} - e - \rho e_{\rho})/\rho e_\theta & -u/\rho e_\theta & 1/\rho e_\theta
\end{pmatrix}.
\end{equation}
Also, from \eqref{3} one obtains
\begin{equation}
\label{Jacobf1}
D_Uf^1(U) = \begin{pmatrix}
u & \rho & 0 \\ u^2 + p_{\rho} & 2 \rho u & p_{\theta} \\ u(e+\tfrac{1}{2}u^{2})+ \rho u e_{\rho} + u p_{\rho} & \rho(e + \tfrac{1}{2}u^{2}) + \rho u^{2} + p & \rho u e_{\theta} + u p_{\theta}
\end{pmatrix}. 
\end{equation}

From hypothesis \hyperref[H3]{\rm{(H$_3$)}}, the thermodynamic potentials satisfy \eqref{StdTherRel}, \eqref{Weyl} and \eqref{las4}, where $\eta = \eta(\rho,\theta)$ denotes the specific entropy. Hence, the (mathematical) entropy function is defined as
\[
\cE := - \rho \eta.
\]
We regard $\cE$ as a function of $V \in \cV$ via the mapping $U \mapsto V$, that is, $\cE = \cE(V)$, $\cE : \cV \to \R$. Hence, it can be shown that
\begin{equation}
\label{6Z}
Z(U) := D_V \cE (f^0(U))^\top = \begin{pmatrix} -\eta+\big(e-\tfrac{1}{2}u^{2} + p/\rho\big)/\theta \\ u/\theta \\  -1/\theta \end{pmatrix},
\end{equation}
defines a mapping $U \mapsto Z$ with domain $\cU$. Its Jacobian matrix can be calculated using \eqref{las4}. This yields,
\begin{equation}
\label{7Z}
D_UZ(U)  = \frac{1}{\theta}\begin{pmatrix}
p_{\rho}/\rho & -u & - \big(e-\tfrac{1}{2}u^{2}+ \rho e_{\rho}\big)/\theta \\ 0 & 1 & -u/\theta \\ 0 & 0 & 1/\theta
\end{pmatrix},
\end{equation}
which is clearly non-singular. Therefore the mapping $U \mapsto Z$ is a diffeomorphism from $\cU$ onto $\cZ := Z(\cU)$. Noticing that 
\begin{equation}
\label{idstar}
D_V^2 \cE(V) = D_V Z(V) = D_U Z(U(V)) D_VU(V) = D_U Z(U(V)) D_U f^0(U(V))^{-1},
\end{equation}
then it is easy to compute the matrix $A^0(U) =  D_U f^0(U)^\top D_V^2\cE(f^0(U)) D_U f^0(U)$ defined in \eqref{coeffsAj}. The result is
\begin{equation}
\label{exprA0}
A^0(U) = \frac{1}{\theta}\begin{pmatrix}
 p_{\rho}/\rho & 0 & 0 \\ 0 & \rho & 0 \\ 0 & 0 & \rho e_{\theta}/ \theta
\end{pmatrix},
\end{equation}
which is clearly symmetric and positive-definite, showing, in turn, that $D_V^2\cE$ is positive definite as well. Likewise, after some straightforward calculations that involve \eqref{inverseDf0} and \eqref{Jacobf1}, one obtains
\begin{align}
A^1(U) &= D_U f^0(U)^\top D_V^2\cE(f^0(U)) D_U f^1(U) \nonumber \\
&= D_U f^0(U)^\top D_UZ(U) D_U f^0(U(V))^{-1} D_U f^1(U) \nonumber \\
&= \frac{1}{\theta} \begin{pmatrix}
 up_{\rho}/\rho & p_\rho & 0 \\ p_\rho & u\rho & p_\theta \\ 0 & p_\theta & \rho u e_{\theta}/ \theta
\end{pmatrix}, \label{exprA1}
\end{align}
which is symmetric, and,
\begin{align}
B(U) &= D_U f^0(U)^\top D_V^2\cE(f^0(U)) G(U) \nonumber \\
&= D_U f^0(U)^\top D_UZ(U) D_U f^0(U(V))^{-1} G(U) \nonumber \\
&= \frac{1}{\theta} \begin{pmatrix}
 0 & 0 & 0 \\ 0 & \mu & 0 \\ 0 & 0 & \alpha/\theta
\end{pmatrix}, \label{exprB}
\end{align}
which is symmetric and positive semi-definite. Hence, the function $\cE$ satisfies properties (a), (c) and (d) of Definition \ref{defentropypair}. We now define the entropy flux as
\[
\Theta:= - \rho u \eta,
\]
from which we have
\[
D_U \Theta = \begin{pmatrix} -u \eta - \rho u \eta_\rho, & - \rho \eta, & - \rho u \eta_\theta \end{pmatrix}.
\]
Using \eqref{Jacobf1} and the thermodynamic relations \eqref{las4} it is easy to verify that
\[
D_V \cE (f^0(U)) D_U f^1(U) = Z(U)^\top D_Uf^1(U) = D_U \Theta(U),
\]
for all $U \in \cU$. This shows \eqref{fluxcond} and, consequently, property (b) of Definition \ref{defentropypair}. We conclude that $(\cE, \Theta)$ defined above is a convex entropy/entropy flux pair for the symmetric dissipative NSF system \eqref{NSF1d}.

The concept of convex entropy is useful to recast a system of the form \eqref{vscl} as a symmetric system for perturbations of a constant state. Let $\bU = (\brho, \bu, \bthe)\in \cU$ be a constant equilibrium state. We then rewrite system \eqref{NSF1d} as
\begin{equation}
\label{10}
\begin{aligned}
f^0(U)_t + D_U f^1(\bU) D_U f^0(\bU)^{-1} f^0(U)_x - &G(\bU) D_U f^0(\bU)^{-1} f^0(U)_{xx} = \widetilde{q}(U)_x + \big( \widetilde{Q}(U) U_x \big)_x,
\end{aligned}
\end{equation}
where
\[
\begin{aligned}
\widetilde{q}(U) &:= D_U f^1(\bU) D_U f^0(\bU)^{-1} \big( f^0(U) - f^0(\bU) \big) - \big( f^1(U) - f^1(\bU) \big),\\
\widetilde{Q}(U) &:= \big( G(U) D_U f^0(U)^{-1} - G(\bU) D_U f^0(\bU)^{-1} \big) D_U f^0(U).
\end{aligned}
\]
Let us define the perturbed variables as
\[
W := D_U f^0(\bU)^{-1} \big( f^0(U) - f^0(\bU)).
\]
From the discussion above it is now clear that the mapping $U \mapsto W$ is invertible. Multiply system \eqref{10} on the left by $D_U f^0(\bU)^\top D_V^2 \cE (\bV)$, where $\bV = f^0(\bU) \in \cV$. The result is the following symmetric (at the leading order) system
\begin{equation}
\label{12}
A^0(\bU) W_t + A^1(\bU) W_x - B(\bU) W_{xx} = q(U)_x + (Q(U)U_x)_x,
\end{equation}
where
\[
\begin{aligned}
q(U) &:= D_U f^0(\bU)^{\top} D_V^2 \cE(\bV) \widetilde{q}(U),\\
Q(U) &:= D_U f^0(\bU)^{\top} D_V^2 \cE(\bV) \widetilde{Q}(U),
\end{aligned}
\]
and the matrices $A^j(\bU)$, $j = 0,1$, and $B(\bU)$ are defined in \eqref{coeffsAj} and \eqref{visctensorB}, respectively, and are evaluated at the constant state $\bU$. If we denote all the thermodynamic potentials and their derivatives evaluated at $\bU = (\brho, \bu, \bthe)$ with overlined variables (for example, $\bp = p(\brho,\bthe)$, $\be_\theta = e_\theta(\brho,\bthe)$, etc.) then from expressions \eqref{exprA0}, \eqref{exprA1} and \eqref{exprB} we obtain
\begin{equation}
\label{15_1}
A^{0}(\bU) = \frac{1}{\bthe}\begin{pmatrix}
\bp_{\rho}/ \brho & 0 & 0 \\ 0 & \brho & 0 \\ 0 & 0 &  \be_{\theta} \brho / \bthe
\end{pmatrix},
\end{equation}
\begin{equation}
\label{15_2}
A^{1}(\bU) = \frac{1}{\bthe} \begin{pmatrix}
 \bp_{\rho} \bu / \brho & \bp_{\rho} & 0 \\ \bp_{\rho} & \bu \, \brho & \bp_{\theta}  \\ 0 & \bp_{\theta} & \brho \, \bu \, \be_{\theta}  / \bthe
\end{pmatrix},
\end{equation}
and,
\begin{equation}
\label{15_3}
B(\bU) = \frac{1}{\bthe}  \begin{pmatrix}
0 & 0 & 0 \\ 0 & \bmu & 0 \\ 0 & 0 & \balp/ \bthe
\end{pmatrix}.
\end{equation}

Clearly, these matrices are all symmetric. In addition, $A^0$ is positive definite and $B$ is positive semi-definite. Hence, system \eqref{NSF1d} can be recast as a system for the perturbed variables $W$ of the form \eqref{12}, with symmetric constant coefficients at leading order and with higher order perturbations of any constant state encoded in the terms of the right hand side of \eqref{12} (those involving $q$ and $Q$).

\section{Proof of Lemma \ref{lemorderN}}
\label{appB}

Let us first examine the expression of $D_U f^0(\bU)^\top D_V^2 \cE(f^0(\bU))$, a constant matrix. From formula \eqref{idstar} in Appendix \ref{ConvExtN-S} we have
\[
D_V^2 \cE(f^0(U)) = D_U Z(U) D_U f^0(U)^{-1},
\]
for all $U \in \cU$, where $Z = Z(U)$ is defined in \eqref{6Z}. Using \eqref{7Z}, \eqref{Jacobf0} and \eqref{inverseDf0} and performing straightforward calculations we obtain the following structure
\begin{equation}
\label{structDD}
\begin{aligned}
D_U f^0(\bU)^\top D_V^2 \cE(f^0(\bU)) &= D_U f^0(\bU)^\top D_U Z(\bU) D_U f^0(\bU)^{-1} = \begin{pmatrix}
\bp_\rho / \brho \, \bthe & 0 & 0 \\ * & * & * \\ * & * & *
\end{pmatrix}.
\end{aligned}
\end{equation}
Here the symbol $\, * \,$ indicates a constant which does not play a role in the forthcoming computations. (In other words, we only need the first row.)

Let us now establish some estimates on the nonlinear terms \eqref{nonlinterms}. For example, consider the term $r(U,U_x)$. Using the decompositions of $F^0(U,U_x)$ and $F^1(U,U_x)$ (equations \eqref{decF0} and \eqref{decF1}) we have
\[
\begin{aligned}
\widetilde{r}(U,U_x) &= - \big( F^1(U,U_x) - F^1(\bU,0) \big) + D_U F^1(\bU,0) D_U F^0(\bU,0)^{-1} \big( F^0(U,U_x) - F^0(\bU,0) \big)\\
&= - \big( f^1(U) - f^1(\bU) - D_Uf^1(\bU) D_U f^0(\bU)^{-1} (f^0(U) - f^0(\bU)) \big) + \\
&\quad - \Gamma^1(U,U_x) + D_Uf^1(\bU) D_U f^0(\bU)^{-1} \Gamma^0(U,U_x)\\
&= - \big( f^1(U) - f^1(\bU) - D_Uf^1(\bU) D_U f^0(\bU)^{-1} (D_Uf^0(\bU)(U-\bU) + O(|U - \bU|^2)) \big) + \\
&\quad - \Gamma^1(U,U_x) + D_Uf^1(\bU) D_U f^0(\bU)^{-1} \Gamma^0(U,U_x)\\
&= - \big( f^1(U) - f^1(\bU) - D_Uf^1(\bU) (U-\bU) + O(|U - \bU|^2) \big) + O(\rho_x^2)\\
&= O(|U - \bU|^2) \big) + O(\rho_x^2),
\end{aligned}
\]
inasmuch as $\Gamma^j(U,U_x) = O(\rho_x^2)$, $j = 0,1$. But more can be said about the form of this term: its first entry is equal to zero. This fact can be verified by inspection. Upon substitution of the expressions for $D_Uf^0(\bU)^{-1}$ and $D_U f^1(\bU)$ (see formulae \eqref{inverseDf0} and \eqref{Jacobf1} in Appendix \ref{ConvExtN-S}) it is not hard to verify that $\widetilde{r}(U,U_x)$ has the following structure,
\[
\begin{aligned}
\widetilde{r}(U,U_x) &= - \big( f^1(U) - f^1(\bU) - D_Uf^1(\bU) D_U f^0(\bU)^{-1} (f^0(U) - f^0(\bU)) \big) + \\
&\quad - \Gamma^1(U,U_x) + D_Uf^1(\bU) D_U f^0(\bU)^{-1} \Gamma^0(U,U_x)\\
&= - \begin{pmatrix} \rho u - \brho \, \bu \\ * \\ * \end{pmatrix} + \begin{pmatrix} 0 & 1 & 0 \\ * & * & * \\ * & * & *  \end{pmatrix}\begin{pmatrix} \rho - \brho \\ \rho u - \brho \, \bu \\ * \end{pmatrix} - \Gamma^1(U,U_x) + \\&\quad + \begin{pmatrix} 0 & 1 & 0 \\ * & * & * \\ * & * & *  \end{pmatrix} \Gamma^0(U,U_x) = \begin{pmatrix} 0 \\ * \\ * \end{pmatrix},
\end{aligned}
\]
due to the form of the expressions for $\Gamma^j(U,U_x)$ in \eqref{decF0} and \eqref{decF1}. Given the form of the matrix in \eqref{structDD} we obtain
\begin{equation}
\label{ordr}
r(U, U_{x}) = D_Uf^0(\bU)^\top D_V^2 \cE(f^0(\bU)) \widetilde{r}(U,U_x) =O \big( \vert U- \bU \vert^{2} + \vert \rho_{x} \vert ^{2} \big) \begin{pmatrix}
0 \\ 1 \\ 1
\end{pmatrix}.
\end{equation}

In the same fashion one can examine the rest of the terms involved in the right hand side of \eqref{Wsystem}. We shall gloss over some of the computational details, which are left to the dedicated reader. For instance, let us now consider the term $R(U, U_{x})U_{x}$. Since, 
\[
\widetilde{R}(U,U_x) := \big[ G(U) D_UF^0(U,U_x)^{-1} - G(\bU) D_U F^0(\bU,0)^{-1} \big] D_U F^0(U,U_x),
\]
then from a direct computation one sees that the first rows of the matrices $G(U) D_UF^0(U,U_x)^{-1}$ and $G(\bU) D_U F^0(\bU,0)^{-1}$ are zero. In view of the fact that
\[
G(U) D_UF^0(U,U_x)^{-1} - G(\bU) D_U F^0(\bU,0)^{-1} = O \big( \vert U - \bU \vert + \vert \rho_{x} \vert ^{2} \big),
\]
we obtain
\[
\widetilde{R}(U,U_{x})U_{x} = O \big( \vert U- \bU \vert \vert U_{x} \vert + \vert U_{x} \vert^{2} \big) \begin{pmatrix}
0 \\ 1 \\ 1
\end{pmatrix}. 
\]
By the structure of $D_Uf^0(\bU)^\top D_V^2 \cE(f^0(\bU))$ in \eqref{structDD} we end up with
\begin{equation}\label{ordR}
R(U,U_{x}) = O \big( \vert U- \bU \vert \vert U_{x} \vert + \vert U_{x} \vert^{2} \big) \begin{pmatrix}
0 \\ 1 \\ 1
\end{pmatrix}.
\end{equation}
Now we compute the order of the term $I(U, U_{x}, U_{xx})$. Let us recall the expression for $\widetilde{I}(U, U_{x}, U_{xx})$: 
\[
\begin{aligned}
\widetilde{I} (U,U_x,U_{xx}) &= - G(\bU) D_U F^0(\bU,0)^{-1} D_{U_x} F^0(U,U_x) U_{xx} + \\
&\quad + \big[ H(U) D_U F^0(U,U_x)^{-1} - H(\bU) D_UF^0(\bU,0)^{-1} \big] D_U F^0(U,U_x) U_{xx} +\\
&\quad - H(\bU) D_UF^0(\bU,0)^{-1} \big[ \partial_x (D_U F^0(U,U_x)) U_x + D_{U_x} F^0(U,U_x) U_{xxx} + \\
&\qquad \qquad\qquad\qquad\qquad \qquad+ \partial_x (D_{U_x} F^0(U,U_x))  U_{xx}\big].
\end{aligned}
\]
Let us look at the first term in the expression above. As $F^{0}(U, U_{x}) = f^{0}(U) + \Gamma^{0}(U, U_{x})$, so that
\[
D_{U_{x}} F^{0}(U, U_{x}) = D_{U_{x}} \Gamma^{0}(U, U_{x}) = \begin{pmatrix}
0 & 0 & 0 \\ 0 & 0 & 0 \\ 2\rho(\kappa-\theta \kappa_{\theta}) \rho_{x} & 0 & 0
\end{pmatrix},
\]
we have
\[
D_{U_{x}}F^{0}(U, U_{x})U_{xx} = \begin{pmatrix}
0 \\ 0 \\ 2 \rho (\kappa - \theta \kappa_{\theta}) \rho_{x}\rho_{xx} \end{pmatrix} = O \big( \vert \rho_{x} \vert \vert \rho_{xx} \vert \big) \begin{pmatrix}
0 \\ 0 \\ 1
\end{pmatrix}.
\]
Thus, using the fact that the first row of the matrix $G(\bU)D_{U}F^{0}(U, 0)$ is zero, we are led to
\begin{equation}\label{ordI-1}
G(\bU)D_{U}F^{0}(\bU, 0)^{-1}D_{U_{x}}F^{0}(U,U_{x})U_{xx} = O \big( \vert \rho_{x} \vert \vert \rho_{xx} \vert \big) \begin{pmatrix}
0 \\ 1 \\ 1 
\end{pmatrix}.
\end{equation}
Let us now consider the second term in the expression of $\widetilde{I}(U, U_{x}, U_{xx})$. A straightforward computation shows that
\[
H(U)D_{U}F^{0}(U, U_{x})^{-1} - H(\bU)D_{U}F^{0}(\bU, 0)^{-1}= \begin{pmatrix}
0 & 0 & 0 \\ k \rho - \bk \brho & 0 & 0 \\ k \rho u - \bk \brho \, \bu & 0 & 0
\end{pmatrix}.
\]
This yields, in turn,
\begin{equation}
\label{ordI-2}
\begin{aligned}
\big[ H(U)D_{U}F^{0}(U, U_{x})^{-1} -& H(\bU)D_{U}F^{0}(\bU, 0)^{-1} \big] D_{U}F^{0}(U,U_{x})U_{xx} = O \big( \vert U - \bU \vert \vert \rho_{xx} \vert \big) \begin{pmatrix}
0 \\ 1 \\ 1
\end{pmatrix}.
\end{aligned}
\end{equation}

In the calculation above we have used the fact that the first coordinate of the vector $D_{U}F^{0}(U,U_{x})U_{xx}$ is just $\rho_{xx}$. To conclude the examination of $\widetilde{I}(U, U_{x}, U_{xx})$, let us estimate the last term, namely,
\[
\begin{aligned}
H(\bU) D_UF^0(\bU,0)^{-1} & \big[ \partial_x (D_U F^0(U,U_x)) U_x + D_{U_x} F^0(U,U_x) U_{xxx} + \partial_x (D_{U_x} F^0(U,U_x))  U_{xx}\big].
\end{aligned}
\]
First, let us observe that 
\[
H(\bU)D_{U}F^{0}(\bU, 0)^{-1} = \begin{pmatrix}
0 & 0 & 0 \\ \bk \brho & 0 & 0 \\ \bk \brho \bu & 0 & 0 
\end{pmatrix}.
\]
Use the already computed matrix, $D_{U_{x}}F^{0}(U, U_{x})$, in order to obtain
\[
D_{U_{x}}F^{0}(U, U_{x})U_{xxx} = \begin{pmatrix}
0 \\ 0 \\ 2 \rho (\kappa- \theta \kappa_{\theta} ) \rho_{x} \rho_{xxx}
\end{pmatrix}.
\]
For the terms $\partial_x (D_U F^0(U,U_x)) U_x$ and $\partial_x (D_{U_x} F^0(U,U_x))  U_{xx}$ we have
\[
\partial_x (D_U F^0(U,U_x)) U_x = \nabla_{U} \cdot D_{U}F^{0}(U, U_{x}) (U_{x}, U_{x}) + \nabla_{U_{x}} \cdot D_{U}F^{0}(U, U_{x})(U_{xx}, U_{x}),
\]
and
\[
\partial_x (D_{U_x} F^0(U,U_x))  U_{xx} = \nabla_{U} \cdot D_{U_{x}}F^{0}(U, U_{x}) (U_{x}, U_{xx}) + \nabla_{U_{x}}\cdot D_{U_{x}}F^{0}(U, U_{x})(U_{xx}, U_{xx}).
\]

Now, from the expression for $D_{U}F^{0}(U, U_{x})$ and $D_{U_{x}}F^{0}(U,U_{x})$ one can easily show that
the vectors $\nabla_{U} \cdot D_{U}F^{0}(U, U_{x})$, $\nabla_{U_{x}} \cdot D_{U}F^{0}(U, U_{x})$, $\nabla_{U} \cdot D_{U_{x}}F^{0}(U, U_{x})$, and  $\nabla_{U_{x}} \cdot D_{U_{x}}F^{0}(U, U_{x})$ all have first and second coordinates equal to zero, and so does the vector 
\[
\partial_x (D_U F^0(U,U_x)) U_x + D_{U_x} F^0(U,U_x) U_{xxx} + \partial_x (D_{U_x} F^0(U,U_x))  U_{xx}.
\]
Therefore by the form of $H(\bU)D_{U}F^{0}(\bU, 0)^{-1}$ we obtain
\begin{equation}
\label{ordI-3}
\begin{aligned}
H(\bU) D_UF^0(\bU,0)^{-1} & \big[ \partial_x (D_U F^0(U,U_x)) U_x + D_{U_x} F^0(U,U_x) U_{xxx} +  \partial_x (D_{U_x} F^0(U,U_x))  U_{xx}\big] = 0. 
\end{aligned}
\end{equation}
Then combining \eqref{ordI-1}, \eqref{ordI-2} and \eqref{ordI-3}, together with the form of the matrix $D_Uf^0(\bU)^\top D_V^2 \cE(f^0(\bU))$ in \eqref{structDD}, we arrive at
\begin{equation}
\label{ordI}
I(U,U_{x},U_{xx}) = O \big( \vert U - \bU \vert \vert \rho_{xx} \vert + \vert \rho_{x} \vert \vert \rho_{xx} \vert \big) \begin{pmatrix}
0 \\ 1 \\ 1
\end{pmatrix}. 
\end{equation}

It remains to estimate the term $g(U,U_{x})$. From the expressions for $\widetilde{g}(U,U_{x})$ and for the matrix $D_Uf^0(\bU)^\top D_V^2 \cE(f^0(\bU))$ which are given, respectively, by \eqref{deftildeg} and \eqref{structDD}, one easily obtains 
\begin{equation}\label{ordg}
g(U,U_{x})= D_Uf^0(\bU)^\top D_V^2 \cE(f^0(\bU)) \widetilde{g}(U,U_{x}) = O \big( \vert U_{x} \vert^{2} \big) \begin{pmatrix}
0 \\ 1 \\ 1
\end{pmatrix}. 
\end{equation}
A combination of estimates \eqref{ordr}, \eqref{ordR}, \eqref{ordI} and \eqref{ordg} yields the result. Lemma \ref{lemorderN} is proved.

\qed

\def\cprime{$'$}


\begin{thebibliography}{10}

\bibitem{AMW98}
{\sc D.~M. Anderson, G.~B. McFadden, and A.~A. Wheeler}, {\em Diffuse-interface
  methods in fluid mechanics}, in Annual Reviews of Fluid Mechanics, J.~L.
  Lumley, M.~Van~Dyke, and H.~L. Reed, eds., vol.~30 of Annu. Rev. Fluid Mech.,
  Annual Reviews, Palo Alto, CA, 1998, pp.~139--165.

\bibitem{BDD06}
{\sc S.~Benzoni-Gavage, R.~Danchin, and S.~Descombes}, {\em Well-posedness of
  one-dimensional {K}orteweg models}, Electron. J. Differ. Equ. \textbf{2006}
  (2006), no.~59, pp.~1--35 (electronic).

\bibitem{BDDJ05}
{\sc S.~Benzoni-Gavage, R.~Danchin, S.~Descombes, and D.~Jamet}, {\em Structure
  of {K}orteweg models and stability of diffuse interfaces}, Interfaces Free
  Bound. \textbf{7} (2005), no.~4, pp.~371--414.

\bibitem{BrDjL03}
{\sc D.~Bresch, B.~Desjardins, and C.-K. Lin}, {\em On some compressible fluid
  models: {K}orteweg, lubrication, and shallow water systems}, Commun. Partial
  Differ. Equ. \textbf{28} (2003), no.~3-4, pp.~843--868.

\bibitem{CTX15}
{\sc H.~Cai, Z.~Tan, and Q.~Xu}, {\em Time periodic solutions of the
  non-isentropic compressible fluid models of {K}orteweg type}, Kinet. Relat.
  Models \textbf{8} (2015), no.~1, pp.~29--51.

\bibitem{Callen-2e}
{\sc H.~B. Callen}, {\em Thermodynamics and an Introduction to
  Thermostatistics}, John Wiley \& Sons, New York, NY, second~ed., 1985.

\bibitem{ChDX21}
{\sc F.~Charve, R.~Danchin, and J.~Xu}, {\em Gevrey analyticity and decay for
  the compressible {N}avier-{S}tokes system with capillarity}, Indiana Univ.
  Math. J. \textbf{70} (2021), no.~5, pp.~1903--1944.

\bibitem{ChHa11}
{\sc F.~Charve and B.~Haspot}, {\em Convergence of capillary fluid models: from
  the non-local to the local {K}orteweg model}, Indiana Univ. Math. J.
  \textbf{60} (2011), no.~6, pp.~2021--2059.

\bibitem{CCD15}
{\sc Z.~Chen, X.~Chai, B.~Dong, and H.~Zhao}, {\em Global classical solutions
  to the one-dimensional compressible fluid models of {K}orteweg type with
  large initial data}, J. Differ. Equ. \textbf{259} (2015), no.~8,
  pp.~4376--4411.

\bibitem{CHZ15}
{\sc Z.~Chen, L.~He, and H.~Zhao}, {\em Nonlinear stability of traveling wave
  solutions for the compressible fluid models of {K}orteweg type}, J. Math.
  Anal. Appl. \textbf{422} (2015), no.~2, pp.~1213--1234.

\bibitem{CHZ17}
{\sc Z.~Chen, L.~He, and H.~Zhao}, {\em Global smooth
  solutions to the nonisothermal compressible fluid models of {K}orteweg type
  with large initial data}, Z. Angew. Math. Phys. \textbf{68} (2017), no.~4,
  pp.~Paper No. 79, 37.

\bibitem{ChXi13}
{\sc Z.~Chen and Q.~Xiao}, {\em Nonlinear stability of viscous contact wave for
  the one-dimensional compressible fluid models of {K}orteweg type}, Math.
  Methods Appl. Sci. \textbf{36} (2013), no.~17, pp.~2265--2279.

\bibitem{ChXM14}
{\sc Z.~Chen, L.~Xiong, and Y.~Meng}, {\em Convergence to the superposition of
  rarefaction waves and contact discontinuity for the 1-{D} compressible
  {N}avier-{S}tokes-{K}orteweg system}, J. Math. Anal. Appl. \textbf{412}
  (2014), no.~2, pp.~646--663.

\bibitem{ChZha14}
{\sc Z.~Chen and H.~Zhao}, {\em Existence and nonlinear stability of stationary
  solutions to the full compressible {N}avier-{S}tokes-{K}orteweg system}, J.
  Math. Pures Appl. (9) \textbf{101} (2014), no.~3, pp.~330--371.

\bibitem{DD01}
{\sc R.~Danchin and B.~Desjardins}, {\em Existence of solutions for
  compressible fluid models of {K}orteweg type}, Ann. Inst. H. Poincar\'e Anal.
  Non Lin\'eaire \textbf{18} (2001), no.~1, pp.~97--133.

\bibitem{DS85}
{\sc J.~E. Dunn and J.~Serrin}, {\em On the thermomechanics of interstitial
  working}, Arch. Ration. Mech. Anal. \textbf{88} (1985), no.~2, pp.~95--133.

\bibitem{EllPin75a}
{\sc R.~S. Ellis and M.~A. Pinsky}, {\em The first and second fluid
  approximations to the linearized {B}oltzmann equation}, J. Math. Pures Appl.
  (9) \textbf{54} (1975), pp.~125--156.

\bibitem{EllPin75b}
{\sc R.~S. Ellis and M.~A. Pinsky}, {\em The projection of
  the {N}avier-{S}tokes equations upon the {E}uler equations}, J. Math. Pures
  Appl. (9) \textbf{54} (1975), pp.~157--181.

\bibitem{Erng66}
{\sc A.~C. Eringen}, {\em A unified theory of thermomechanical materials},
  Internat. J. Engrg. Sci. \textbf{4} (1966), pp.~179--202.

\bibitem{FrKo15a}
{\sc H.~Freist\"{u}hler and M.~Kotschote}, {\em Models of two-phase fluid
  dynamics \`a la {A}llen-{C}ahn, {C}ahn-{H}illiard, and {$\ldots$}
  {K}orteweg!}, Confluentes Math. \textbf{7} (2015), no.~2, pp.~57--66.

\bibitem{FrKo17}
{\sc H.~Freist\"{u}hler and M.~Kotschote}, {\em Phase-field and
  {K}orteweg-type models for the time-dependent flow of compressible two-phase
  fluids}, Arch. Ration. Mech. Anal. \textbf{224} (2017), no.~1, pp.~1--20.

\bibitem{FrKo19}
{\sc H.~Freist\"{u}hler and M.~Kotschote}, {\em Phase-field
  descriptions of two-phase compressible fluid flow: interstitial working and a
  reduction to {K}orteweg theory}, Quart. Appl. Math. \textbf{77} (2019),
  no.~3, pp.~489--496.

\bibitem{Frd54}
{\sc K.~O. Friedrichs}, {\em Symmetric hyperbolic linear differential
  equations}, Comm. Pure Appl. Math. \textbf{7} (1954), pp.~345--392.

\bibitem{FLa3}
{\sc K.~O. Friedrichs and P.~D. Lax}, {\em Systems of conservation equations
  with a convex extension}, Proc. Nat. Acad. Sci. U.S.A. \textbf{68} (1971),
  pp.~1686--1688.

\bibitem{GLZh20}
{\sc J.~Gao, Z.~Lyu, and Z.-a. Yao}, {\em Lower bound of decay rate for
  higher-order derivatives of solution to the compressible fluid models of
  {K}orteweg type}, Z. Angew. Math. Phys. \textbf{71} (2020), no.~4, p.~108.

\bibitem{GaZoYa15}
{\sc J.~Gao, Y.~Zou, and Z.-a. Yao}, {\em Long-time behavior of solution for
  the compressible {N}avier-{S}tokes-{K}orteweg equations in {$\Bbb{R}^3$}},
  Appl. Math. Lett. \textbf{48} (2015), pp.~30--35.

\bibitem{Godu61a}
{\sc S.~K. Godunov}, {\em An interesting class of quasi-linear systems}, Dokl.
  Akad. Nauk SSSR \textbf{139} (1961), pp.~521--523.

\bibitem{Gurt65}
{\sc M.~E. Gurtin}, {\em Thermodynamics and the possibility of spatial
  interaction in elastic materials}, Arch. Ration. Mech. Anal. \textbf{19}
  (1965), pp.~339--352.

\bibitem{Gurt81}
{\sc M.~E. Gurtin}, {\em An introduction to
  continuum mechanics}, vol.~158 of Mathematics in Science and Engineering,
  Academic Press, Inc. (Harcourt Brace Jovanovich, Publishers), New
  York-London, 1981.

\bibitem{Gurt96}
{\sc M.~E. Gurtin}, {\em Generalized
  {G}inzburg-{L}andau and {C}ahn-{H}illiard equations based on a microforce
  balance}, Phys. D \textbf{92} (1996), no.~3-4, pp.~178--192.

\bibitem{GPV96}
{\sc M.~E. Gurtin, D.~Polignone, and J.~Vi\~{n}als}, {\em Two-phase binary
  fluids and immiscible fluids described by an order parameter}, Math. Models
  Methods Appl. Sci. \textbf{6} (1996), no.~6, pp.~815--831.

\bibitem{HaEv82}
{\sc H.~J.~M. Hanley and D.~J. Evans}, {\em A thermodynamics for a system under
  shear}, J. Chem. Phys. \textbf{76} (1982), no.~6, pp.~3225--3232.

\bibitem{Hasp09}
{\sc B.~Haspot}, {\em Existence of strong solutions for nonisothermal
  {K}orteweg system}, Ann. Math. Blaise Pascal \textbf{16} (2009), no.~2,
  pp.~431--481.

\bibitem{Hasp16}
{\sc B.~Haspot}, {\em Existence of global
  strong solution for {K}orteweg system with large infinite energy initial
  data}, J. Math. Anal. Appl. \textbf{438} (2016), no.~1, pp.~395--443.

\bibitem{HaLi94}
{\sc H.~Hattori and D.~N. Li}, {\em Solutions for two-dimensional system for
  materials of {K}orteweg type}, SIAM J. Math. Anal. \textbf{25} (1994), no.~1,
  pp.~85--98.

\bibitem{HaLi96b}
{\sc H.~Hattori and D.~N. Li}, {\em The existence of
  global solutions to a fluid dynamic model for materials for {K}orteweg type},
  J. Partial Differ. Equ. \textbf{9} (1996), no.~4, pp.~323--342.

\bibitem{HaLi96a}
{\sc H.~Hattori and D.~N. Li}, {\em Global solutions of
  a high-dimensional system for {K}orteweg materials}, J. Math. Anal. Appl.
  \textbf{198} (1996), no.~1, pp.~84--97.

\bibitem{HdMl10}
{\sc M.~Heida and J.~M\'{a}lek}, {\em On compressible {K}orteweg fluid-like
  materials}, Internat. J. Engrg. Sci. \textbf{48} (2010), no.~11,
  pp.~1313--1324.

\bibitem{HPZ18}
{\sc X.~Hou, H.~Peng, and C.~Zhu}, {\em Global well-posedness of the 3{D}
  non-isothermal compressible fluid model of {K}orteweg type}, Nonlinear Anal.
  Real World Appl. \textbf{43} (2018), pp.~18--53.

\bibitem{HYZ17}
{\sc X.~Hou, L.~Yao, and C.~Zhu}, {\em Vanishing capillarity limit of the
  compressible non-isentropic {N}avier-{S}tokes-{K}orteweg system to
  {N}avier-{S}tokes system}, J. Math. Anal. Appl. \textbf{448} (2017), no.~1,
  pp.~421--446.

\bibitem{Hu05}
{\sc J.~Humpherys}, {\em Admissibility of viscous-dispersive systems}, J.
  Hyperbolic Differ. Equ. \textbf{2} (2005), no.~4, pp.~963--974.

\bibitem{IoIo01}
{\sc R.~J. Iorio, Jr. and V.~d.~M. Iorio}, {\em Fourier analysis and partial
  differential equations}, vol.~70 of Cambridge Studies in Advanced
  Mathematics, Cambridge University Press, Cambridge, 2001.

\bibitem{JCL10}
{\sc D.~Jou, J.~Casas-V\'{a}zquez, and G.~Lebon}, {\em Extended Irreversible
  Thermodynamics}, Springer-Verlag, New York, fourth~ed., 2010.

\bibitem{KaTh83}
{\sc S.~Kawashima}, {\em Systems of a Hyperbolic-Parabolic Composite Type, with
  Applications to the Equations of Magnetohydrodynamics}, PhD thesis, Kyoto
  University, 1983.

\bibitem{KSX22}
{\sc S.~Kawashima, Y.~Shibata, and J.~Xu}, {\em Dissipative structure for
  symmetric hyperbolic-parabolic systems with {K}orteweg-type dispersion},
  Commun. Partial Differ. Equ. \textbf{47} (2022), no.~2, pp.~378--400.

\bibitem{KaSh88a}
{\sc S.~Kawashima and Y.~Shizuta}, {\em On the normal form of the symmetric
  hyperbolic-parabolic systems associated with the conservation laws}, Tohoku
  Math. J. (2) \textbf{40} (1988), no.~3, pp.~449--464.

\bibitem{KMR23}
{\sc J.~Keim, C.-D. Munz, and C.~Rohde}, {\em A relaxation model for the
  non-isothermal {N}avier-{S}tokes-{K}orteweg equations in confined domains},
  J. Comput. Phys. \textbf{474} (2023), pp.~Paper No. 111830, 28.

\bibitem{Kortw1901}
{\sc D.~J. Korteweg}, {\em Sur la forme que prennent les \'{e}quations du
  mouvement des fluides si l'on tient compte des forces capillaires causées
  par des variations de densit\'{e} consid\'{e}rables mais continues et sur la
  th\'{e}orie de la capillarit\'{e} dans l'hypoth\`{e}se d'une variation
  continue de la densit\'{e}}, Arch. N\'{e}erl. Sci. Exactes Nat. Ser. II
  \textbf{6} (1901), pp.~1--24.

\bibitem{Kot08}
{\sc M.~Kotschote}, {\em Strong solutions for a compressible fluid model of
  {K}orteweg type}, Ann. Inst. H. Poincar\'e Anal. Non Lin\'eaire \textbf{25}
  (2008), no.~4, pp.~679--696.

\bibitem{Kot10}
{\sc M.~Kotschote}, {\em Strong
  well-posedness for a {K}orteweg-type model for the dynamics of a compressible
  non-isothermal fluid}, J. Math. Fluid Mech. \textbf{12} (2010), no.~4,
  pp.~473--484.

\bibitem{Kot12a}
{\sc M.~Kotschote}, {\em Dynamics of
  compressible non-isothermal fluids of non-{N}ewtonian {K}orteweg type}, SIAM
  J. Math. Anal. \textbf{44} (2012), no.~1, pp.~74--101.

\bibitem{LeJC08}
{\sc G.~Lebon, D.~Jou, and J.~Casas-V\'{a}zquez}, {\em Understanding
  non-equilibrium thermodynamics. Foundations, applications, frontiers},
  Springer-Verlag, Berlin, 2008.

\bibitem{MoVi16}
{\sc A.~Morro and M.~Vianello}, {\em Interstitial energy flux and stress-power
  for second-gradient elasticity}, Math. Mech. Solids \textbf{21} (2016),
  no.~4, pp.~403--412.

\bibitem{Nir59}
{\sc L.~Nirenberg}, {\em On elliptic partial differential equations}, Ann.
  Scuola Norm. Sup. Pisa Cl. Sci. (3) \textbf{13} (1959), pp.~115--162.

\bibitem{Pipp57}
{\sc A.~B. Pippard}, {\em Elements of classical thermodynamics for advanced
  students of physics}, Cambridge University Press, New York, 1957.

\bibitem{PlV22}
{\sc R.~G. Plaza and J.~M. Valdovinos}, {\em Dissipative structure of
  one-dimensional isothermal compressible fluids of {K}orteweg type}, J. Math.
  Anal. Appl. \textbf{514} (2022), no.~2, p.~126336.

\bibitem{Se81}
{\sc J.~Serrin}, {\em Phase transitions and interfacial layers for van der
  {W}aals fluids}, in Recent methods in nonlinear analysis and applications
  ({N}aples, 1980), A.~Canfora, S.~Rionero, C.~Sbordone, and G.~Trombetti,
  eds., Liguori, Naples, 1981, pp.~169--175.

\bibitem{Se83}
{\sc J.~Serrin}, {\em The form of
  interfacial surfaces in {K}orteweg's theory of phase equilibria}, Quart.
  Appl. Math. \textbf{41} (1983), no.~3, pp.~357--364.

\bibitem{SSZh22}
{\sc W.~Shi, Z.~Song, and J.~Zhang}, {\em Large-time behavior of solutions in
  the critical spaces for the non-isentropic compressible {N}avier-{S}tokes
  equations with capillarity}, J. Math. Fluid Mech. \textbf{24} (2022), no.~3,
  pp.~Paper No. 59, 33.

\bibitem{ShKa85}
{\sc Y.~Shizuta and S.~Kawashima}, {\em Systems of equations of
  hyperbolic-parabolic type with applications to the discrete {B}oltzmann
  equation}, Hokkaido Math. J. \textbf{14} (1985), no.~2, pp.~249--275.

\bibitem{Sl83}
{\sc M.~Slemrod}, {\em Admissibility criteria for propagating phase boundaries
  in a van der {W}aals fluid}, Arch. Ration. Mech. Anal. \textbf{81} (1983),
  no.~4, pp.~301--315.

\bibitem{Szk20a}
{\sc Y.~Suzuki}, {\em A {GENERIC} formalism for {K}orteweg-type fluids: {I}.
  {A} comparison with classical theory}, Fluid Dyn. Res. \textbf{52} (2020),
  no.~1, pp.~015516, 28.

\bibitem{Szk20b}
{\sc Y.~Suzuki}, {\em A {GENERIC}
  formalism for {K}orteweg-type fluids: {II}. {H}igher-order models and
  relation to microforces}, Fluid Dyn. Res. \textbf{52} (2020), no.~2,
  pp.~025510, 13.

\bibitem{TWX12}
{\sc Z.~Tan, H.~Wang, and J.~Xu}, {\em Global existence and optimal {$L^2$}
  decay rate for the strong solutions to the compressible fluid models of
  {K}orteweg type}, J. Math. Anal. Appl. \textbf{390} (2012), no.~1,
  pp.~181--187.

\bibitem{TZh14}
{\sc Z.~Tan and R.~Zhang}, {\em Optimal decay rates of the compressible fluid
  models of {K}orteweg type}, Z. Angew. Math. Phys. \textbf{65} (2014), no.~2,
  pp.~279--300.

\bibitem{TXKV15}
{\sc L.~Tian, Y.~Xu, J.~G.~M. Kuerten, and J.~J.~W. van~der Vegt}, {\em A local
  discontinuous {G}alerkin method for the (non)-isothermal
  {N}avier-{S}tokes-{K}orteweg equations}, J. Comput. Phys. \textbf{295}
  (2015), pp.~685--714.

\bibitem{UDK12}
{\sc Y.~Ueda, R.~Duan, and S.~Kawashima}, {\em Decay structure for symmetric
  hyperbolic systems with non-symmetric relaxation and its application}, Arch.
  Ration. Mech. Anal. \textbf{205} (2012), no.~1, pp.~239--266.

\bibitem{UDK18}
{\sc Y.~Ueda, R.~Duan, and S.~Kawashima}, {\em New structural
  conditions on decay property with regularity-loss for symmetric hyperbolic
  systems with non-symmetric relaxation}, J. Hyperbolic Differ. Equ.
  \textbf{15} (2018), no.~1, pp.~149--174.

\bibitem{vdW1894}
{\sc J.~D. van~der Waals}, {\em Thermodynamische {T}heorie der
  {K}apillarit\"{a}t unter {V}oraussetzung stetiger {D}ichte\"{a}nderung}, Z.
  Phys. Chem. \textbf{13} (1894), no.~1, pp.~657--725.

\bibitem{VoH72}
{\sc A.~I. Vol\cprime$\!\!$~pert and S.~I. Hudjaev}, {\em The {C}auchy problem
  for composite systems of nonlinear differential equations}, Math. USSR Sb.
  \textbf{16} (1972), no.~4, pp.~517--544.

\bibitem{WaYa14}
{\sc W.~Wang and L.~Yao}, {\em Vanishing viscosity limit to rarefaction waves
  for the full compressible fluid models of {K}orteweg type}, Commun. Pure
  Appl. Anal. \textbf{13} (2014), no.~6, pp.~2331--2350.

\bibitem{WaTa11}
{\sc Y.~Wang and Z.~Tan}, {\em Optimal decay rates for the compressible fluid
  models of {K}orteweg type}, J. Math. Anal. Appl. \textbf{379} (2011), no.~1,
  pp.~256--271.

\bibitem{We49}
{\sc H.~Weyl}, {\em Shock waves in arbitrary fluids}, Comm. Pure Appl. Math.
  \textbf{2} (1949), pp.~103--122.

\bibitem{ZhTa14}
{\sc X.~Zhang and Z.~Tan}, {\em Decay estimates of the non-isentropic
  compressible fluid models of {K}orteweg type in {$R^3$}}, Commun. Math. Sci.
  \textbf{12} (2014), no.~8, pp.~1437--1456.

\end{thebibliography}
\end{document}